\documentclass[a4paper,12pt]{article}
\usepackage{graphicx} 
\usepackage[a4paper]{geometry}
\usepackage{fancyhdr}
\usepackage{amsmath}
\usepackage{amssymb}
\usepackage{amsthm}
\usepackage{xcolor}
\usepackage{graphicx}
\usepackage{pgfplots}
\usepackage{nameref}
\usepackage{subcaption}
\usepackage{enumerate}
\pgfplotsset{compat=1.11}
\usepackage{dsfont}
\usetikzlibrary{decorations.markings}
\usepackage{bm}
\usepackage{tikz}
\usepackage{float}
\usepackage{tkz-euclide}
\usetikzlibrary{angles}
\usetikzlibrary{quotes}
\usepackage{cite}
\usepackage{multirow}

\usepackage[hidelinks]{hyperref}
\usepackage{hyperref}
\usepackage{csquotes}
\usepackage{caption}

\numberwithin{equation}{section}

\newcommand{\R}{\mathbb{R}}

\newcommand{\N}{\mathbb{N}}

\newcommand{\F}{\mathcal{F}}

\newcommand{\LM}{\mathcal{L}}

\newcommand{\RH}{Rankine-Hugoniot }
\newcommand{\1}{\mathds{1}}
\newcommand{\sgn}{\operatorname{sign}}

\newtheorem{teo}{Theorem}[section]
\newtheorem{prop}[teo]{Proposition}
\newtheorem{cor}[teo]{Corollary}
\newtheorem{lem}[teo]{Lemma}
\theoremstyle{definition}
\newtheorem{defi}[teo]{Definition}
\newtheorem{rmk}[teo]{Remark}

\usepackage{ulem}
\renewcommand{\L}[1]{\mathbf{L^#1}}
\newcommand{\Lloc}[1]{\mathbf{L^{#1}_{loc}}}
\newcommand{\C}[1]{\mathbf{C^{#1}}}
\newcommand{\Cc}[1]{\mathbf{C_c^{#1}}}
\newcommand{\W}[1]{\mathbf{W^{#1}}}
\newcommand{\BV}{\mathrm{BV}}

\newcommand{\B}{\mathbf{B}}

\title{Well-posedness {and numerical approximation} \\ of nonlinear conservation laws with hysteresis} 
\author{Paola Goatin
\thanks{Universit\'e C\^ote d'Azur, Inria, CNRS, LJAD, 2004 route des
  Lucioles - BP 93, 06902 Sophia Antipolis Cedex, France. E-mail:
  \texttt{paola.goatin@inria.fr}}
\and Stefan Moreti\thanks{Department of Mathematics, University of Trento, Italy. E-mail:
  \texttt{stefan.moreti@unitn.it}}}
\date{}

\begin{document}
\maketitle
\begin{abstract}
\noindent This article studies the Cauchy problem for the scalar conservation law 
\[ 
\partial_t u + \partial_t w + \partial_x f(u) = 0, 
\]
where $w(x,t) = [\mathcal{F}(u)(x,t)]$ is the output of a specific hysteresis operator, namely the Play hysteresis operator, and $f$ is a $\C2$ convex flux function. The hysteresis operator models a rate-independent memory effect, introducing a specific non-local feature into the partial differential equation. We define a suitable notion of  entropy weak solution and analyse in detail the Riemann problem. Furthermore, a Godunov-type finite volume numerical scheme is developed to compute approximate solutions.
The convergence of the scheme for $\BV$ initial data provides the existence of an entropy weak solution. Finally, a stability estimate is established, implying the uniqueness and overall well-posedness of the entropy weak solution.
    
\end{abstract}


\section{Introduction}\label{S0}

In this work, we deal with the Cauchy problem for nonlinear conservation laws with hysteresis as follows:

\begin{equation}\label{eq: intro2}
    \begin{cases}
         \partial_t u+ \partial_t w+\partial_x f(u)=0 & \text{in } \R \times [0,T[,\\
        w=[\F(u,w_0)] & \text{in } \R \times [0,T[, \\
        u(x,0)=u_0(x) & \text{in } \R ,\\
        w(x,0)=w_0(x) & \text{in } \R,
    \end{cases}
\end{equation}
where $\cal F$ is the so-called Play hysteresis operator, $u_0\in \BV(\R)\cap \L1(\R)$ is the initial datum for the solution $u$, $w_0\in \BV(\R)\cap \L1(\R)$ is a suitable space-dependent function for the initial values of the output $w$, $f\in \C2(\R;\R)$ is a nonlinear strictly convex function and $T>0$ is fixed. As for classical conservation laws, the results for $f$ concave can be deduced from the convex case, so we will not describe it in detail.

Hysteresis is a phenomenon commonly observed in various natural and engineered systems, typically characterized by a lag or delay in the system’s response to changes in the input. For comprehensive accounts of mathematical models for hysteresis and their use in connection with PDEs, we refer the reader to \cite{CORR1} and \cite{AVH}. Among the various mathematical models used to describe such behaviour, one classical choice is the Play hysteresis operator $\F$
\[
\F : \B([0,T]) \times \R \mapsto \B([0,T])
\]
which represents a memory dependent input-output relationship between a pair of time-dependent scalar functions belonging to a certain functional space $\B([0,T])$. As the unknown $u$ in the PDE is a function both depending on a space and on a time variable, we consider this relationship between the pair $(t\mapsto u(x,t), t \mapsto w(x,t))$ for every $x$, so formally we define
\begin{equation}
     w(x,t):= [\F(u(x,\cdot),w_0(x))](t), \quad \text{a.e.}\ (x,t).
\end{equation} 
For each fixed $x\in \R$, we briefly describe the input-output relationship of the Play operator referring to Figure~\ref{fig: linplay} (for a more detailed description see \cite[Section III.1]{AVH}). Given an amplitude $a>0$, we denote 
\[
{\cal L}:=\left\{ (u,w)\in\R^2,\,|u-w|\le a\right\}
\]
the strip of the feasible states of the system. If  the pair $(u(x,t),w(x,t))$ satisfies $|u(x,t)-w(x,t)|<a$ for some $t>0$, that is if it belongs in the interior of $\cal L$, and if the input $u(x,\cdot)$ changes in time, then the output $w(x,\cdot)$ will not change until the pair $(u,w)$ will possibly reach one of the two boundary lines of $\cal L$. If $w(x,t)=u(x,t)-a$, that is the pair $(u,w)$ is on the lower boundary of $\cal L$ and if the input $u$ increases, then the output $w$ will increase together with $u$; if instead $u$ decreases then $w$ stays constant and the pair $(u,w)$ enters the interior of $\cal L$. If $(u,w)$ belongs to the upper boundary of $\cal L$, then the behaviour is symmetric, reversing the role of the monotonicity of $u$.
\begin{figure}
    \centering
    \begin{tikzpicture}[scale=1]
         \draw[->, gray] (-3, 0) -- (3, 0) node[right] {$u$}; 
    \draw[->, gray] (0, -2.75) -- (0, 2.75) node[right] {$w$};
    \draw[-, thick] (-3, -2) -- (1.5, 2.5);
    \draw[->] (-0.2,1)--(-0.8,0.4);
    \draw[-, thick] (-1.5, -2.5) -- (3, 2); 
    \draw[-, dashed] (0, 1) -- (2, 1);
    \draw[-, dashed] (0,-1)--(-2,-1);
    \draw[->] (0.2,-1)--(0.8,-0.4);
    \draw[<->] (0.5,1.2)--(1.5,1.2);
    \draw[<->] (-0.5,-1.2)--(-1.5,-1.2);
    \node at (-1.4,0.2) {$-a$} ;
    \node at (1.2,-0.2) {$a$} ;
    \end{tikzpicture}
    \caption{The Play hysteresis operator}
    \label{fig: linplay}
\end{figure}
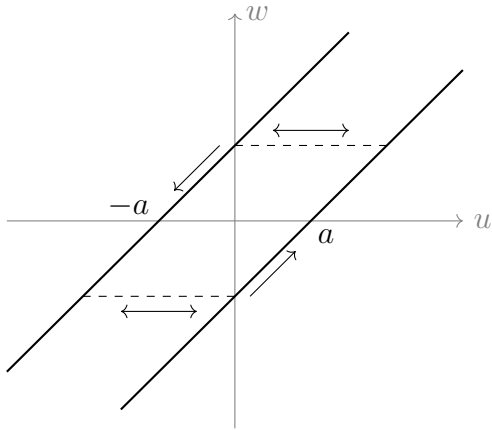
 
Given an initial state $(u_0(x),w_0(x))\in \mathcal{L}$ and the evolution of $u(x,\cdot)$, we can then trace the evolution of $w(x,\cdot)$. We can notice that the value of $w(x,t)$ is not determined pointwisely by $u(x,t)$, indeed for any fixed value of $u$ we have more than one possible value $w$ in the strip $\LM$. Hence $w(x,t)$, besides its initial value $w(x,0)=w_0(x)$, is determined by the whole history $u(x,\cdot)|_{[0,t]}$, i.e. there is a memory effect involved. \par 
In particular, this memory effect is rate-independent, which means that the relation between $u$ and $w$ does not depend on the time-derivative of $u$. Notice that this requirement is essential if we want to draw hysteresis relations as in Figure \ref{fig: linplay} and it is a general feature of hysteresis phenomena. 

As shown in \cite{AVH} (see also \cite{CORR1}), if we assume $\W{1,1}$ regularity in time for the input $u$, this heuristic description of the relationship $w = [\F(u,w_0)]$
 can be rigorously characterized by the following variational inequality:
\begin{equation}\label{eq: defplay}
    |u - w| \leq a, \quad (u - w - v) \partial_t w \geq 0 \quad \forall v \text{ such that } |v| \leq a.
\end{equation}
In such case also $w(x,\cdot)\in \W{1,1}([0,T])$ and its time-derivative is described for almost every $t$ as follows
\begin{equation}
    \partial_t w(x,t)= \begin{cases}
        \partial_t u(x,t) &\quad \text{if} \quad w(x,t)=u(x,t)-a \quad \text{and} \quad \partial_t u(x,t) \geq 0,\\
        \partial_t u(x,t) &\quad  \text{if} \quad w(x,t)=u(x,t)+a \quad \text{and} \quad \partial_t u(x,t) \leq 0,\\
        0 &\quad \text{otherwise}.  
    \end{cases}
\end{equation}
In the same work \cite{AVH}, it is also shown how the Play operator can be uniquely extended to an operator that maps continuous inputs to continuous outputs, preserving the same heuristic description.\par 

In this paper, we consider $\F$ as applied to the entropy weak solution of a conservation law, which in general is neither in $\W{1,1}([0,T])$ or in $\C{0}([0,T])$ with respect to $t$,  for any fixed $x$. The extension of hysteresis operators to non-regular inputs is a well studied problem, see \cite{RF2}, \cite{RF1} and \cite{RV}. For our intent, we only extend $\F$ to functions with jump discontinuities, performing the following approximation and limit procedure: if the input $u$ has a jump discontinuity at time $t$, we consider a sequence of continuous functions $u_\varepsilon$ that fill the discontinuity in a monotone way and such that $u_\varepsilon \to u$ pointwisely; then we define $w:=\lim_{\varepsilon \to 0} [\F(u_\varepsilon,w_0)]$ (see Figure~\ref{fig: esempioF} for a specific example). In particular, such a construction of the output $w$ as a function in $\L1([0,T])$ is independent of how we monotonically fill the jump of the input $u$, because of the rate-independence property of the Play operator (see \cite{BFMS} for all the details). We then have the following characterization, which is proven in \cite{BFMS} and will inspire \eqref{eq: genweakhis} in the definition of entropy weak solution. 

\begin{prop}\label{prop: whis}
    Fix $x \in \R$ and suppose $w_0(x) \in \R$, $u(x,\cdot),w(x,\cdot)\in \BV{([0,T[;\R)}$ with a finite number of jump discontinuities, $u(x,0)=u(x,0+)$ and $w(x,0)=w(x,0+)=w_0(x)$. Then the following are equivalent: 
    \begin{enumerate}
        \item $w(x,t) = \F[u(x,\cdot),w_0(x)](t)$ for almost every $t\in [0,T[$;
        \item for almost every $t\in [0,T[$, $|w(x,t+)-u(x,t+)|\leq a$  and 
        \begin{equation}\label{eq: weakhis}
            \int_0^t(u(x,s+)-w(x,s+))d(\partial_t w(x,s)) \geq a |\partial_t w|(]0,t[) ,
        \end{equation}
        where $\partial_t w$ is interpreted as the measure associated to the distributional time-derivative of $w$, $|\partial_t w|$ its total variation and $u(x,s+),w(x,s+)$ are the right-continuous in time representative of $u$ and $w$ for fixed $x\in\R$.
    \end{enumerate}
\end{prop} 

\begin{figure}
    \centering
    \begin{tikzpicture}[scale=0.8]
         \draw[->, gray, yshift=1cm] (-0.5, 0) -- (4, 0) node[right] {$t$}; 
        \draw[->, gray,yshift=1cm] (0, -0.5) -- (0, 2.5) node[right] {$u$};
        
        \draw[-,thick,yshift=1cm] (0,0)--(1.5,0);
        \draw[-,thick,yshift=1cm] (1.5,2)--(3,2);

        \draw[-,yshift=1cm](3,0.1)--(3,-0.1);
        \draw[-,dashed,yshift=1cm](1.5,0)--(1.5,2);

         \node[text = black, below] (r) at (1.5,1) {\scriptsize{$t^*$}};
        \node[text = black, below] (r) at (3,1) {\scriptsize{$T$}};

        \draw[->, gray, yshift=-3cm] (-0.5, 0) -- (4, 0) node[right] {$t$}; 
        \draw[->, gray,yshift=-3cm] (0, -0.5) -- (0, 2.5) node[right] {$w$};
        
        \draw[-,thick,yshift=-3cm] (0,0)--(1.5,0);
        \draw[-,thick,yshift=-3cm] (1.5,1)--(3,1);

        \draw[-,yshift=-3cm](3,0.1)--(3,-0.1);
        \draw[-,dashed,yshift=-3cm](1.5,0)--(1.5,1);

         \node[text = black, below] (r) at (1.5,-3) {\scriptsize{$t^*$}};
        \node[text = black, below] (r) at (3,-3) {\scriptsize{$T$}};

    \draw[->, gray,xshift=10cm] (-3.75, 0) -- (3.75, 0) node[right] {$u$}; 
    \draw[->, gray,xshift=10cm] (0, -3.75) -- (0, 3.75) node[right] {$w$};
    \draw[-, thick,xshift=10cm] (-3.5, -2) -- (2, 3.5);
    \draw[-, thick,xshift=10cm] (-2, -3.5) -- (3.5, 2); 

    \draw[->,red, thick,xshift=10cm] (0,0)--(1.5,0)--(2.95,1.45);
    
    \filldraw[xshift =10cm] (0,0) circle (1.3pt) node at (0,0)[above]{\scriptsize{$(u_-,w_0)$}};
    \filldraw[xshift =10cm] (3,1.5) circle (1.3pt) node at (3,1.7)[left]{\scriptsize{$(u_+,w^*)$}};
    \filldraw[xshift =10cm,red] (3,1.5) circle (0pt) node at (2.3,0.5)[right]{\scriptsize{$(u_\varepsilon(\cdot),w_\varepsilon(\cdot))$}};

    \end{tikzpicture}
    \caption{An explicit example of the operator $\F$ applied to $u$ with a jump discontinuity. Here $a=1$, $u=u_-=0$ for $t<t^*$, $u=u_+=2$ for $t>t^*$ and $w_0=0$; consequently $w=w_0$ for $t<t^*$ and $w=w^*=1$ for $t>t^*$. In red we highlight the path followed by the couple $(u_\varepsilon,w_\varepsilon)$ which, after the limiting procedure, collapses to $(u_-,w_0)$ for $t<t^*$ and $(u_+,w^*)$ for $t>t^*.$}
    \label{fig: esempioF}
\end{figure}
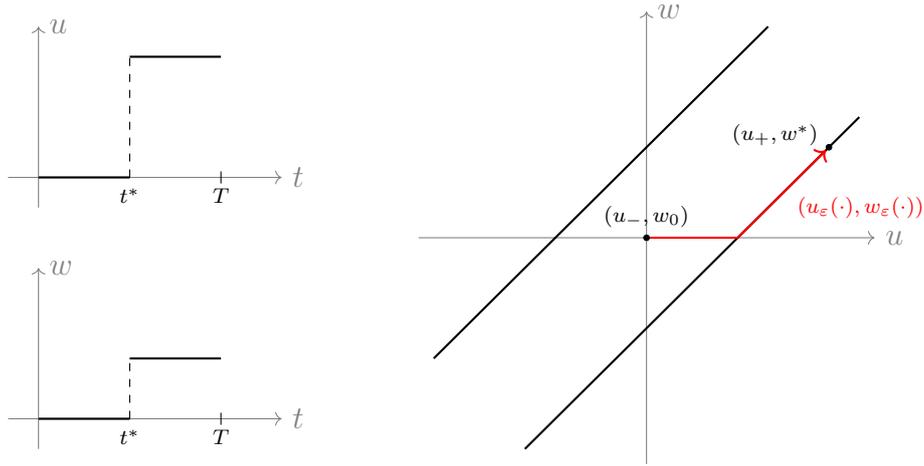

Given the description of the Play operator $\F$, we introduce the definition of entropy weak  solution, which also includes a weaker notion of the relationship $w = [\F(u,w_0)]$.

\begin{defi} \label{def: hweaksol}
    A couple of functions $u,w\in\C0 ([0,T[;\Lloc1(\R;\R)])$] is an entropy weak solution to \eqref{eq: intro2} if: \begin{enumerate}[i)]
    \item it holds 
    \begin{multline}
    \label{eq: hweaksol}
        \int\limits_0^{T} \int\limits_{-\infty }^{+\infty}\left[\left(|u-k|+|w-\hat{k}|\right) \partial_t \phi+\sgn(u-k)\left(f(u)-f(k)\right) \partial_x\phi\right] ~dx~ dt \\
        +\int\limits_{-\infty}^{+\infty} (|u_0(x)-k| +|w_0(x)-\hat{k}|)\phi(x,0) ~dx\geq 0,
    \end{multline} for every $\phi \in \Cc1(\R \times [0,T[;\R^+)$ and for every couple $(k,\hat{k})\in \LM;$
    \item for almost every $(x,t)\in \R \times [0,T[$ it holds \begin{equation}\label{eq: dishis}
        |u(x,t)-w(x,t)|\leq a;
    \end{equation} 
    \item for almost every $x$, $u(x,\cdot),w(x,\cdot)\in \L{2}(\R;\R)$ and the distributional derivative $\partial_t w$ is a measure on $\R \times [0,T[$ that satisfies 
    \begin{equation}
        \label{eq: genweakhis}
        \frac{1}{2}\int_\R (u(x,t)^2-u_0(x)^2)~dx+\frac{1}{2}\int_\R (w(x,t)^2-w_0(x)^2)~dx \leq - a \left| \partial_t w \right|(\R\times\, ]0,t[),
    \end{equation}  
    for almost every $t\in\, ]0,T[$. 
\end{enumerate}
\end{defi}

Condition \eqref{eq: hweaksol} is the adaptation to the conservation equation in~\eqref{eq: intro2} of the classical Kru\v zkov entropy condition~\cite{Krukov} and it is inspired by the one introduced in \cite{AVH1}. Also \eqref{eq: genweakhis} was introduced in \cite{AVH1} and it can be interpreted as a weak hysteresis relationship since, by supposing more regularity on $u$ and $w$, such as space-time Sobolev regularity, and by using the strong form of the PDE: $\partial_t u+\partial_t w +\partial_x f(u)=0$, \eqref{eq: genweakhis} 
is equivalent to
\begin{equation}\label{eq: hismis}
    \int_0^t\int_\R (u(x,s)-w(x,s)) \partial_t w (x,s) \, dx\,ds \geq a \int_0^t \int_\R \left|\partial_t w(x,s)\right| \, dx\,ds,
\end{equation}
(see \cite{BFMS, AVH1} for details). Note that \eqref{eq: hismis} can be interpreted as the extension to space-time dependent functions of the hysteresis relationship \eqref{eq: weakhis}, when the input and the output are $\W{1,1}$ in time functions.

Besides \cite{BFMS, AVH1}, scalar hyperbolic conservation laws with hysteresis have been investigated in various applied contexts such as \cite{ADBA1, ADBA2, BFMS2, CF1, CF2, CF3, F1, KOP1, KOR1, MR, Simile, marchesin, ZHANG}. In \cite{MR, Simile} the authors study an initial-boundary value problem where the PDE is similar to ours, with either linear or monotone flux, motivated by application to transport in porous media, where hysteresis is a common feature. There, the hysteresis operator is respectively either a rather general convex-sided Play model or a finite sum of $K-$nonlinear Play operators. They prove well-posedness for their initial-boundary value problem in the class of functions $C^0([0,T]; L^2]0,l[)$, where $[0,T]$ is the time interval and $]0,l[$, $l>0$, is the spatial domain, and also existence of time-differentiable solutions in the case with linear flux, differentiable boundary data and $0$ initial condition. Their approach relies on the abstract theory of Cauchy problems involving multivalued operators defined on Banach spaces. \cite{AVH1} provides the well-posedness of a Cauchy problem similar to \eqref{eq: intro2}, but with the presence of a different hysteresis operator, namely the completed delayed Relay. These results can be extended to the Preisach operator~\cite{AVH}. The proof is based on the time-discretization of the partial differential equation, the construction of the corresponding approximate solutions, and, thanks to {\it a priori} estimates, their convergence to an exact  solution as the discretisation is refined. We also refer the reader to \cite{KOR1}, where the previously cited results are well summarized, and to \cite{KOP1} which focuses  on entropy conditions for such and similar equations. 

Unlike \cite{Simile, AVH1} and inspired by the recent works \cite{BFMS, BFMS2}, we adopt a more constructive approach focused on the study of characteristics via the resolution of Riemann problems. Further contributions in this direction include \cite{ADBA1,ADBA2,CF1,CF2,CF3,F1},  which study wave propagation in hyperbolic models with hysteresis in the context of traffic flow models (see also \cite{ZHANG}). 
The Cauchy problem \eqref{eq: intro2} was studied in the case of linear flux $f(u)=u$ with the Play operator in \cite{BFMS}, providing a full analysis of the associated Riemann problems. The same full analysis was also done in \cite{BFMS2}, where the case of linear flux is still studied but with a more general, versatile and complex hysteretic relationship, namely the one given by the Preisach operator (see \cite[Section IV.1]{AVH} for the description of such hysteresis operator). In particular, the present paper can be seen as an extension of \cite{BFMS} to the case of a general nonlinear convex flux $f$. While in \cite{BFMS,BFMS2} the well-posedness of the Cauchy problem is based on the wave-front tracking approximations, which is a standard tool for proving existence and stability of hyperbolic systems of conservation laws \cite{AB3}, here we develop a finite volume numerical scheme, which is shown to converge to the entropy weak  solution. Such scheme can be directly applied to the case of linear flux treated in \cite{BFMS}, see Remark \ref{rmk: linear}.

Actually, none of the above references uses a Godunov-type numerical approximation and its limit procedure to prove the existence of solutions. This indeed seems to be a novelty of our analysis, in addition to the treatment of the passage to the limit in the hysteresis relationship (see Section 3.4), similarly to what is done in \cite{BFMS,BFMS2}. To the best of the authors' knowledge, the most closely related work remains \cite{MR}, where also an explicit upwind finite volume numerical scheme is analysed.

We refer the reader to the books  \cite{AB3, EV, HH} for general theory on scalar conservation laws and to \cite{FVM,LEVEQUE,TORO} for more specific results about finite volume numerical schemes.

This article is then structured as follows: in Section 2 we deal with the Riemann problem, in which the solutions are combinations of shock and rarefaction waves as the flux is nonlinear. In Section 3 the Godunov-type numerical scheme is developed; thanks to $\BV$ estimates, it is shown that the approximate solutions generated by the scheme converge as the mesh size tends to $0$ to the entropy weak solution of the Cauchy problem \eqref{eq: intro2} establishing existence; some numerical examples and simulations are highlighted at the end of this section; finally, in Section 4 it is shown that the solutions to the Riemann problem constructed in Section 2 are entropy admissible, and a stability theorem is stated, which implies uniqueness of the entropy solution.


\section{The Riemann problem}\label{S2}
The Riemann problem associated to \eqref{eq: intro2} is the Cauchy problem with the following initial data \begin{equation}\label{eq: datiRP}
    u_0(x) = \begin{cases}
        u_l \quad &x<0,\\
        u_r \quad & x\geq0,
    \end{cases}\quad \text{and} \quad w_0(x) = \begin{cases}
        w_l \quad &x<0,\\
        w_r \quad & x\geq0,
    \end{cases}
\end{equation} such that $|u_0(x)-w_0(x)|\leq a. $\par 
By solutions of the Riemann problem, we intend a couple of functions $(u,w) \in \C0([0,T[ ;\Lloc1(\R;\R^2))$ satisfying the entropy weak formulation of the PDE \eqref{eq: hweaksol} and such that $w(x,t)=[\F(u(x,\cdot),w_0(x)](t)$ for almost every $x\in\R$ and $t\in [0,T[$. It will turn out that, $u$ being the solution of a Riemann problem,  $u(x,\cdot)$ is a piecewise continuous function with only jump discontinuities for almost every $x\in \R$. Hence $\F$ has to be interpreted as applied to functions with jump discontinuities, see Proposition~\ref{prop: whis} and the paragraph above. 

First of all, we notice that, if an entropy solution $(u,w)$ is discontinuous along a curve $(\sigma(t),t)$, with $u_\pm=u_\pm (t):= u(\sigma(t)\pm,t) $ and  $w_\pm=w_\pm (t):= w(\sigma(t)\pm,t)$, then  the \RH condition
\begin{equation}\label{eq: hrh}
    f(u_-)-f(u_+)=\sigma'(t) (u_--u_++w_--w_+).
\end{equation} 
follows from \eqref{eq: hweaksol}.
Indeed, it is enough to choose $k \geq \max(u_-,u_+)$ and $\hat{k} \geq \max(w_-,w_+)$ in \eqref{eq: hweaksol} and integrate it by parts to infer \begin{equation}\label{eq: dis1RH}
        \sigma'(t)[u_-+w_--u_+-w_+]-[f(u_-)-f(u_+)]\geq0;
\end{equation}
similarly, with $k \leq \min(u_-,u_+)$ and $\hat{k} \leq \min(w_-,w_+)$ we deduce 
\begin{equation}\label{eq: dis2RH}
        -\sigma'(t)[u_-+w_--u_+-w_+]+[f(u_-)-f(u_+)]\geq0,
\end{equation}
which, together with the previous inequality, implies \eqref{eq: hrh}. The inequalities \eqref{eq: dis1RH} and \eqref{eq: dis2RH} are derived as in the case of no hysteresis (see e.g. \cite[Theorem 4.3 and proof of Theorem 4.4]{AB3}).

For the explicit construction of solutions, we divide the analysis in the following cases: $u_l=u_r$, $u_l<u_r$ and $u_l>u_r$.

    \subsection{ \texorpdfstring{$u_l=u_r$:}{ul uguale ur}}\label{caso: ul=ur}
    In this case the weak solution is $(u(x,t),w(x,t))=(u_0(x),w_0(x))$ for each $(x,t)$. Notice that if $w_l\not=w_r$ then this solution has a stationary discontinuity for $w$, which satisfies \eqref{eq: hrh}. Moreover, $w= [\F(u,w_0)]$ holds trivially.
    
    \subsection{ \texorpdfstring{$u_l<u_r$:}{ul minore ur} rarefaction waves}\label{caso: ul<ur}
    If there was no hysteresis term in the equation, since the flux is convex, we would expect a rarefaction wave type solution for the unknown $u$. Based on the possible different directions of propagation of the waves, we consider the following subcases.  
    
    \subsubsection{\texorpdfstring{$f'(u_l) \geq 0$:}{zero minore ul minore ur}}\label{sottocaso: 0<ul<ur}
    By convexity, in this case $f'(u)\geq0$ for $u\in [u_l,u_r]$ and the rarefaction has positive speed. Therefore, $u(x,t) \equiv u_l$  and $w(x,t) \equiv w_l$ for $x<0$. For $x>0$ fixed, we expect $t\mapsto u(x,t)$ to decrease in a monotone way from $u_r$ to $u_l$, reaching the value $u_l$ in finite time. If we impose the relationship $w=[\F(u,w_0)]$, then it should hold that for each $t$ such that $w_r-a \leq u(x,t) \leq u_r,$ as $u(x,\cdot)$ is monotone decreasing, at least formally $w_t(x,t)=0$, hence $w(x,t)=w_r$. Whereas, if $u(x,t)\leq w_r-a$ and $u(x,\cdot)$ keeps decreasing, the couple $(u,w)$ should follow the upper boundary $u = w+a$ of the hysteresis region, so $w_t(x,t) = u_t(x,t)$. 
    \par By this analysis we deduce that, at least formally, $u$ 
    should satisfy the conservation law
    \begin{equation}\label{eq: eqftilde}
       \partial_t u+\partial_x\bar{f}_{w_r} (u) = 0,
    \end{equation} 
    where 
    \begin{equation}\label{eq: fbar}
        \bar{f}_{w_r} (u) := \begin{cases}
            \frac{1}{2} f(u) + \frac{1}{2} f(w_r-a) \quad& u\leq w_r-a,\\
            f(u) \quad &w_r-a\leq u \leq u_r,
        \end{cases}
    \end{equation}
    
is still a convex, piecewise $C^1$, flux, see Figure \ref{fig: fbarw}, left. Therefore, the classical Riemann problem for \eqref{eq: eqftilde} is solved in a standard way. 

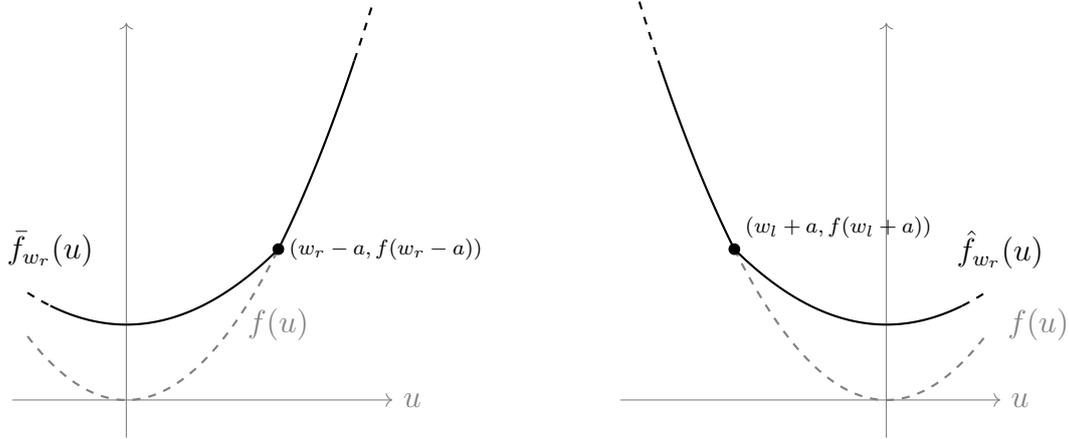
\begin{figure}
        \centering
        \begin{tikzpicture}
        \begin{scope}       
           \draw[->, gray] (-1.5,0) -- (3.5,0) node[right] {$u$};
           \draw[->,gray] (0,-0.5) -- (0,5);

           \draw[domain= -1.3:2, smooth, samples=200, variable=\x, thick, gray, dashed] plot ({\x}, {(1/2)*(\x)*(\x)});
           
           \draw[domain=-1:2, smooth, samples=200, variable=\x, thick]  plot ({\x}, {(1/4)*(\x)*(\x) + 1}) ;
           
           \draw[domain= 2:3, smooth, samples=200, variable=\x, thick] plot ({\x}, {(1/2)*(\x)*(\x)});
           \draw[domain=-1.3:-1, smooth, dashed, samples=100, variable=\x, thick] plot ({\x}, {(1/4)*(\x)*(\x) + 1});

           \draw[domain= 3:3.25, smooth, samples=100, variable=\x, thick, dashed] plot ({\x}, {(1/2)*(\x)*(\x)});

           \filldraw[fill=black] (2,2) circle (2pt) node[right] {\scriptsize $(w_r-a, f(w_r-a))$};

           \node at (-1,2) {$\bar{f}_{w_r}(u)$};

           \node[gray] at (2,1) {${f}(u)$};
        \end{scope}
        \begin{scope}[xshift =  10cm]
        \draw[->, gray] (-3.5,0) -- (1.5,0) node[right] {$u$};
           \draw[->,gray] (0,-0.5) -- (0,5);

           \draw[domain= -2:1.3, smooth, samples=200, variable=\x, thick, gray, dashed] plot ({\x}, {(1/2)*(\x)*(\x)});
           
           \draw[domain=-2:1, smooth, samples=200, variable=\x, thick]  plot ({\x}, {(1/4)*(\x)*(\x) + 1}) ;
           
           \draw[domain= -3:-2, smooth, samples=200, variable=\x, thick] plot ({\x}, {(1/2)*(\x)*(\x)});
           \draw[domain=1:1.3, smooth, dashed, samples=100, variable=\x, thick] plot ({\x}, {(1/4)*(\x)*(\x) + 1});

           \draw[domain= -3.25:-3, smooth, samples=100, variable=\x, thick, dashed] plot ({\x}, {(1/2)*(\x)*(\x)});

           \filldraw[fill=black] (-2,2) circle (2pt) node[above right] {\scriptsize $(w_l+a, f(w_l+a))$};

           \node at (1.5,2) {$\hat{f}_{w_r}(u)$};

           \node[gray] at (2,1) {${f}(u)$};
        \end{scope} 
        \end{tikzpicture}   
    \caption{An example of $\bar{f}_{w_r}$ (left) and $\hat{f}_{w_l}$ (right) compared with $f$ (gray); the additive constants $\tfrac{1}{2} f(w_r-a)$ in \eqref{eq: fbar} and $\tfrac{1}{2} f(w_r+a)$ in \eqref{eq: fhat} are needed for the continuity of $\bar{f}_{w_r}$ and $\hat{f}_{w_l}$ respectively.}
    \label{fig: fbarw}
\end{figure}

In particular, if $w_r-a$ is such that $u_l < w_r-a < u_r$, then $u$ consists of two rarefaction waves, the first from $u_r$ to $w_r-a$, with $\partial_t w=0$, and the second from $w_r-a$ to $u_l$, with $\partial_t u=\partial_t w,$ separated by the intermediate state $u \equiv w_l-a.$ Then also $w$ can be computed as $w=[\F(u,w_0)]$ and it will consist of a rarefaction wave that overlaps with the second rarefaction wave of $u$, connecting $w_r$ to the state $u_l+a$, see Figure \ref{fig: 0minoreulminoreur}.\\
    If instead we suppose  that $w_r-a \leq u_l <u_r$ then $\bar{f}_w(u)=f(u)$ for $u\in [u_l,u_r]$, so $u$ consists of only one rarefaction and $w$ remains constant as $\partial_t w=0$. \\
The opposite extreme is when $u_l < u_r = w_r-a$, which corresponds to the case when the couple $(u,w)$ at time $t=0$ for $x>0$ belongs already to the upper boundary of the hysteresis region. Then $\bar{f}_w(u)=\frac{1}{2}f(u)+\frac{1}{2} f(w_r-a)$ for $u\in [u_l,u_r]$, so we have a rarefaction wave for both $u$ and $w$ with $w(x,t)=u(x,t)+a$ for each $x>0$ and $t>0$. \par
    It is easy to see that by construction $(u,w)$ is a  weak solution of the PDE \eqref{eq: hweaksol}, as whenever $\partial_t w = 0$, $u$ solves $\partial_t u+\partial_x f(u)=0$, and when $\partial_t w=\partial_t u$, then $\partial_t u+\frac{1}{2} \partial_x f(u)=0$ and $\partial_t w + \frac{1}{2} \partial_x f(u)=0$. Moreover, the relationship $w= [\F (u),w_0]$ holds in the strong classical sense as $u(x,\cdot)\in \W{1,1}(0,T)$ for almost every $x$ and we may develop stationary shocks for $w$. 

    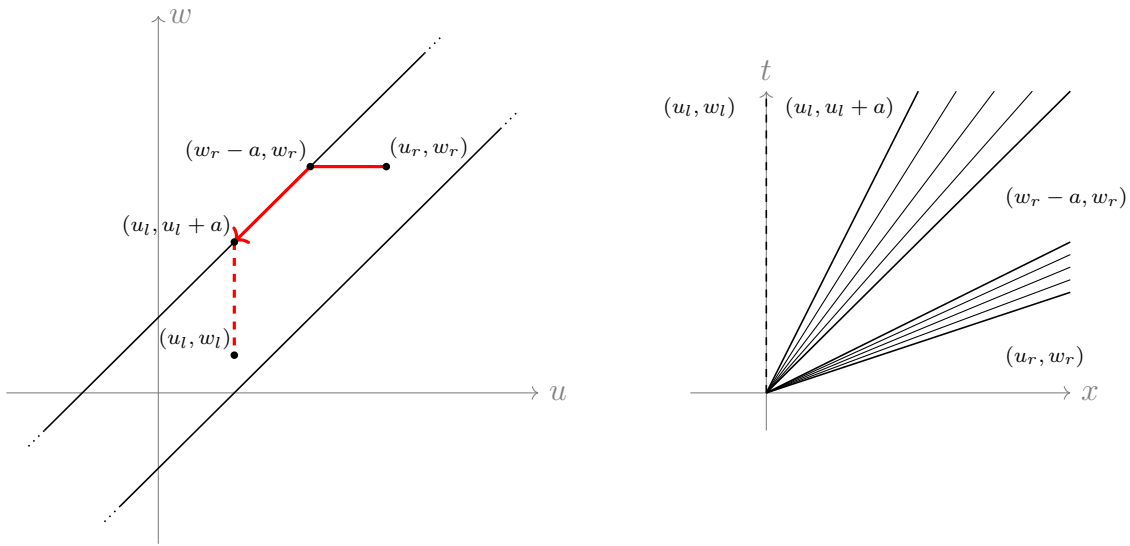
\begin{figure}
    \centering
    \begin{tikzpicture}
        \begin{scope}
            \draw[->, gray] (-2, 0) -- (5, 0) node[right] {$u$}; 
            \draw[->, gray] (0, -2) -- (0, 5) node[right] {$w$};
            \draw[-, semithick] (-0.5, -1.5) -- (4.5, 3.5);
            \draw[-, dotted, semithick] (-0.5, -1.5) -- (-0.75, -1.75);
            \draw[-, dotted, semithick] (4.5, 3.5) -- (4.75, 3.75);
            
            \draw[-, semithick] (-1.5, -0.5) -- (3.5, 4.5); 
            \draw[-, dotted, semithick] (-1.5, -0.5) -- (-1.75, -0.75); 
            \draw[-, dotted, semithick] (3.5, 4.5) -- (3.75, 4.75); 
            
            \draw[->, red, very thick] (3,3) -- (2,3)--(1,2);
            \draw[-, red, very thick, dashed] (1,2)--(1,0.5);

            \draw (3,3) node[circle, fill=black, inner sep=1pt] {} node[above right,xshift=-3pt, yshift=-1pt] {\scriptsize{$(u_r, w_r)$}};
            \draw (1,0.5) node[circle, fill=black, inner sep=1pt] {} node[above left, xshift=3pt, yshift=-2pt] {\scriptsize{$(u_l, w_l)$}};
            \draw (1,2) node[circle, fill=black, inner sep=1pt] {} node[above left,xshift=3pt, yshift=-2pt] {\scriptsize{$(u_l, u_l+a)$}};
            \draw (2,3) node[circle, fill=black, inner sep=1pt] {} node[above left,xshift=3pt, yshift=-2pt] {\scriptsize{$(w_r-a, w_r)$}};
        \end{scope}

        \begin{scope}[xshift=8cm]
            \draw[->, gray] (-1, 0) -- (4, 0) node[right] {$x$}; 
            \draw[->, gray] (0, -0.5) -- (0, 4) node[above] {$t$};
            
            \draw[semithick] (0,0) -- (4,1.333333);
            \draw[thin] (0,0) -- (4,1.5);
            \draw[thin] (0,0) -- (4,1.667);
            \draw[thin] (0,0) -- (4,1.8333);
            \draw[semithick] (0,0) -- (4,2);
        
            \draw[semithick] (0,0) -- (4,4);
            \draw[thin] (0,0) -- (3.5,4);
            \draw[thin] (0,0) -- (3,4);
            \draw[thin] (0,0) -- (2.5,4);
            \draw[semithick] (0,0) -- (2,4);

            \draw[semithick, dashed] (0,0)--(0,4);

            \node[above right] at (3,0.2) {\scriptsize{$(u_r, w_r)$}};
            \node[above right] at (3,2.3) {\scriptsize{$(w_r-a, w_r)$}};
            \node[above right] at (0.1,3.5) {\scriptsize{$(u_l, u_l+a)$}};
            \node[above right] at (-1.5,3.5) {\scriptsize{$(u_l, w_l)$}};
        \end{scope}
    \end{tikzpicture}
    \caption{On the left, the couple $(u,w)$ is shown in the hysteresis plane: the arrow represents the path followed by  $t\mapsto (u(x,t),w(x,t))$ for $x>0$; the dashed line instead represents the stationary shock. On the right, the solution $(u,w)$ in the $(x,t)$ plane: it consists of two rarefaction waves and a stationary shock for $w$ along $x = 0$. In this example, the flux function is $f(u) = \frac{1}{2} u^2$, the parameter $a = 1$, and the initial data $(u_l,w_l)=(1,0.5)$ and $(u_r,w_r)=(3,3).$}
    \label{fig: 0minoreulminoreur}
    \end{figure}

    \subsubsection{\texorpdfstring{$ f'(u_r) \leq 0$:}{ ul minore ur minore di 0}}\label{sottocaso: ul<ur<0}
    In this case we still deal with a rarefaction wave, but moving to the left. As a result, for $x>0$, $u(x,t) = u_r$ and $w(x,t) = w_r$ for each $t$. Instead, for $x<0$, $t\mapsto u(x,t)$ increases in a monotone way from $u_l$ to $u_r$. Reasoning as in Subcase \ref{sottocaso: 0<ul<ur}, we can check that $u$ formally solves 
    \[\partial_t u+\partial_x \hat{f}_{w_l}(u)=0,\]
    with
    \begin{equation} \label{eq: fhat}
        \hat{f}_{w_l} (u)= \begin{cases}
            f(u) \quad& w_l-a\leq u \leq w_l+a,\\
            \frac{1}{2}f(u)+ \frac{1}{2}f(w_l+a) \quad & w_l+a \leq u,
        \end{cases}
    \end{equation} 
    see Figure \ref{fig: fbarw}, right. 
    Hence, if $u_l<w_l+a<u_r$ the solution $(u,w)$ for $x<0$ consists of two rarefaction waves, one only for $u$ and the other for both $u$ and $w$, see Figure \ref{fig: ulminoreurminore0}; if instead $u_r\in [u_l,w_l+a]$, then $w(x,t) = w_0(x)$ for every $t$ and there is only one rarefaction wave for $u$; instead if $u_l=w_l+a$, then we have only one rarefaction wave for both $u$ and $w$ and $w(x,t) = u(x,t)-a$ for $x>0$ and for every $t$.

    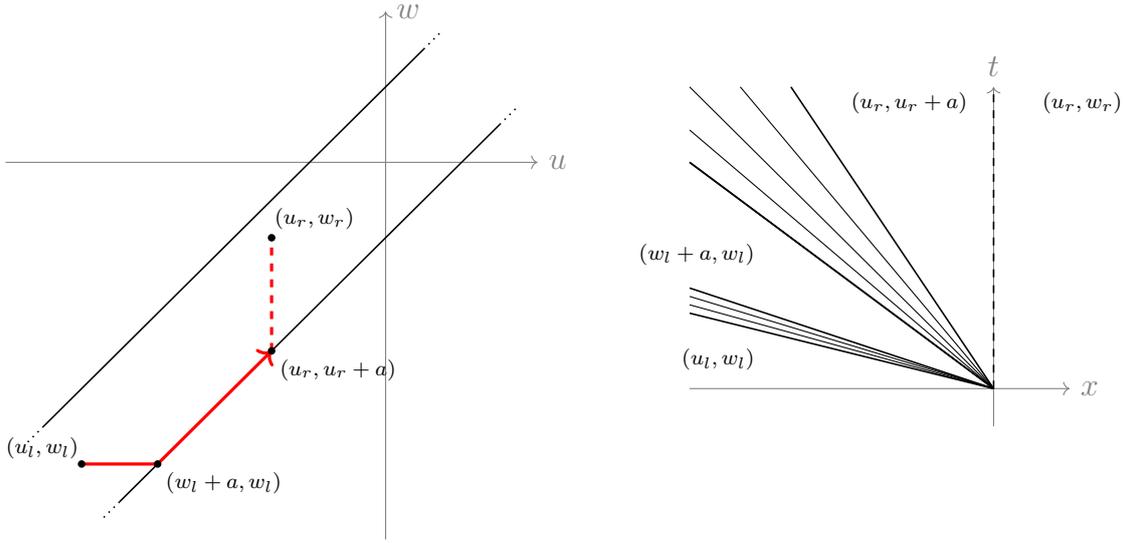
\begin{figure}
    \centering
    \begin{tikzpicture}
        \begin{scope}[yshift = 3cm]
            \draw[->, gray] (-5, 0) -- (2, 0) node[right] {$u$}; 
            \draw[->, gray] (0, -5) -- (0, 2) node[right] {$w$};
            
            \draw[-, semithick] (0.5, 1.5) -- (-4.5, -3.5);
            \draw[-, dotted, semithick] (0.5, 1.5) -- (0.75, 1.75);
            \draw[-, dotted, semithick] (-4.5, -3.5) -- (-4.75, -3.75);
            
            \draw[-, semithick] (1.5, 0.5) -- (-3.5, -4.5); 
            \draw[-, dotted, semithick] (1.5, 0.5) -- (1.75, 0.75); 
            \draw[-, dotted, semithick] (-3.5, -4.5) -- (-3.75, -4.75); 
            
            \draw[->, red, very thick] (-4,-4) -- (-3,-4)--(-1.5,-2.5);
            \draw[-, red, very thick, dashed] (-1.5,-2.5)--(-1.5,-1);

            \draw (-1.5,-1) node[circle, fill=black, inner sep=1pt] {} node[above right,xshift=-3pt, yshift=-1pt] {\scriptsize{$(u_r, w_r)$}};
            \draw (-4,-4) node[circle, fill=black, inner sep=1pt] {} node[above left, xshift=3pt, yshift=-2pt] {\scriptsize{$(u_l, w_l)$}};
            \draw (-3,-4) node[circle, fill=black, inner sep=1pt] {} node[below right, xshift=-1pt, yshift=1pt] {\scriptsize{$(w_l+a, w_l)$}};
            \draw (-1.5,-2.5) node[circle, fill=black, inner sep=1pt] {} node[below right,xshift=-1pt, yshift=1pt] {\scriptsize{$(u_r, u_r+a)$}};
        \end{scope}

        \begin{scope}[xshift=8cm]
            \draw[->, gray] (-4, 0) -- (1, 0) node[right] {$x$}; 
            \draw[->, gray] (0, -0.5) -- (0, 4) node[above] {$t$};
            
            \draw[semithick] (0,0) -- (-4,1.333333);
            \draw[thin] (0,0) -- (-4,1.111111);
            \draw[thin] (0,0) -- (-4,1.222222);
            \draw[semithick] (0,0) -- (-4,1);
        
            \draw[semithick] (0,0) -- (-2.6667,4);
            \draw[thin] (0,0) -- (-4,3.4285714);
            \draw[thin] (0,0) -- (-4,4);
            \draw[thin] (0,0) -- (-3.333333,4);
            \draw[semithick] (0,0) -- (-4,3);
            
            \draw[semithick] (0,0) -- (-4,3);

            \draw[semithick, dashed] (0,0)--(0,4);

            \node[above right] at (0.5,3.5) {\scriptsize{$(u_r, w_r)$}};
            \node[above left] at (-0.2,3.5) {\scriptsize{$(u_r, u_r+a)$}};
            \node[above left] at (-3,1.5) {\scriptsize{$(w_l+a, w_l)$}};
            \node[above left] at (-3,0.1) {\scriptsize{$(u_l, w_l)$}};
               
        \end{scope}
    \end{tikzpicture}
    \caption{An example of the case when $f'(u_r) \leq 0$. In this example, the flux function is $f(u) = \frac{1}{2} u^2$, the parameter $a = 1$, and the initial data $(u_l,w_l)=(-3,-3)$ and $(u_r,w_r)=(-1.5,-1)$.}
    \label{fig: ulminoreurminore0}
    \end{figure}

    \subsubsection{\texorpdfstring{$ f'(u_l)\leq 0 \leq  f'(u_r) $:}{ ul minore di zero minore di ur}}\label{sottocaso: ul<0<ur}
    This subcase is the combination of the previous two Subcases \ref{sottocaso: 0<ul<ur} and \ref{sottocaso: ul<ur<0}. In particular, looking at the direction of the waves, for $x>0$ we expect $u(x,\cdot)$ to decrease from $u_r$ to $u_*$ where $u_*\in[u_l,u_r]$ is the point of minimum of $f$ on $[u_l,u_r]$. Instead, for $x<0$, $t\mapsto u(x, t)$ increases from $u_l$ to $u_*$. The idea is then to consider the two Riemann problems with data \[u_0^1(x)=\begin{cases}
        u_* \quad& x<0,\\
        u_r \quad& x\geq0,\end{cases} \quad \text{and} \quad u_0^2(x)= \begin{cases}
        u_l \quad& x<0,\\
        u_* \quad& x\geq0,
    \end{cases} \]
    solve the first one as in Subcase \ref{sottocaso: 0<ul<ur} and consider its solution $(u_1,w_1)$ restricted to $x>0$, solve the second one as in Subcase \ref{sottocaso: ul<ur<0} and consider the solution $(u_2,w_2)$ on $x<0$. We finally obtain the solution $(u,w)$ to the original Riemann problem by gluing together these two restrictions. \par
    Notice that, since $f'(u_*)=0$, then $u(0-,t)=u_*=u(0+,t)$ so $u(x,t)$ is continuous at $x=0$ for each $t>0$. Our solution may have a discontinuity in $w$ at $x=0$, however, as we already pointed out, stationary discontinuities for $w$ are admissible as long as $u$ is continuous, see again \eqref{eq: hrh} and Figure \ref{fig: sottocaso3}. 

    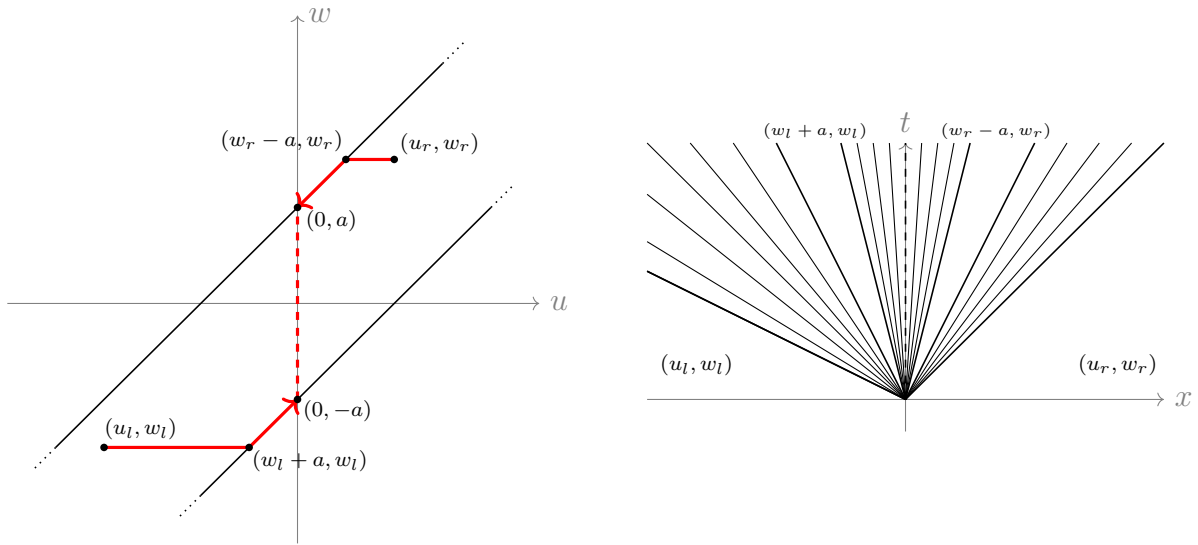
\begin{figure}
    \centering
    \begin{tikzpicture}
        \begin{scope}[scale=1.2727, yshift = 1cm]
            \draw[->, gray] (-3, 0) -- (2.5, 0) node[right] {$u$}; 
            \draw[->, gray] (0, -2.5) -- (0, 3) node[right] {$w$};
            
            \draw[-, semithick] (-2.5, -1.5) -- (1.5, 2.5);
            \draw[-, dotted, semithick] (-2.5, -1.5) -- (-2.75, -1.75);
            \draw[-, dotted, semithick] (1.5, 2.5) -- (1.75, 2.75);
            
            \draw[-, semithick] (-1, -2) -- (2, 1); 
            \draw[-, dotted, semithick] (-1, -2) -- (-1.25, -2.25); 
            \draw[-, dotted, semithick] (2, 1) -- (2.25, 1.25); 
            
            \draw[->, red, very thick] (1,1.5) -- (0.5,1.5)--(0,1);
            \draw[->, red, very thick] (-2,-1.5) -- (-0.5,-1.5)--(0,-1);
            \draw[-, red, very thick, dashed] (0,-1)--(0,1);

            \draw (1,1.5) node[circle, fill=black, inner sep=1pt] {} node[above right,xshift=-2pt, yshift=-2pt] {\scriptsize{$(u_r, w_r)$}};
            
            \draw (0.5,1.5) node[circle, fill=black, inner sep=1pt] {} node[above left,xshift=3pt, yshift=-1pt] {\scriptsize{$(w_r-a, w_r)$}};

            \draw (0,1) node[circle, fill=black, inner sep=1pt] {} node[below right,xshift=-1pt, yshift=3pt] {\scriptsize{$(0, a)$}};

            \draw (0,-1) node[circle, fill=black, inner sep=1pt] {} node[below right,xshift=-2pt, yshift=4pt] {\scriptsize{$(0, -a)$}};

            \draw (-0.5,-1.5) node[circle, fill=black, inner sep=1pt] {} node[below right,xshift=-3pt, yshift=3pt] {\scriptsize{$(w_l+a, w_l)$}};

            \draw (-2,-1.5) node[circle, fill=black, inner sep=1pt] {} node[above right,xshift=-4pt, yshift=-1pt] {\scriptsize{$(u_l, w_l)$}};

        \end{scope}

        \begin{scope}[xshift=8cm, scale = 0.85]
            \draw[->, gray] (-4, 0) -- (4, 0) node[right] {$x$}; 
            \draw[->, gray] (0, -0.5) -- (0, 4) node[above] {$t$};
            
            \draw[semithick] (0,0) -- (4,4);
            \draw[thin] (0,0) -- (3.5,4);
            \draw[thin] (0,0) -- (3,4);
            \draw[thin] (0,0) -- (2.5,4);
            \draw[semithick] (0,0) -- (2,4);

            \draw[semithick] (0,0) -- (1,4);
            \draw[thin] (0,0) -- (0.25,4);
            \draw[thin] (0,0) -- (0.5,4);
            \draw[thin] (0,0) -- (0.75,4);
            
            \draw[semithick] (0,0) -- (-1,4);
            \draw[thin] (0,0) -- (-0.75,4);
            \draw[thin] (0,0) -- (-0.5,4);
            \draw[thin] (0,0) -- (-0.25,4);

            \draw[semithick] (0,0) -- (-2,4);
            \draw[thin] (0,0) -- (-2.666667,4);
            \draw[thin] (0,0) -- (-3.333333,4);
            \draw[thin] (0,0) -- (-4,4);
            \draw[thin] (0,0) -- (-4,3.2);
            \draw[thin] (0,0) -- (-4,2.46153);
            \draw[semithick] (0,0) -- (-4,2);
            
            \draw[semithick] (0,0) -- (-4,2);

            \draw[semithick, dashed] (0,0)--(0,4);

            \node[above right] at (2.5,0.2) {\scriptsize{$(u_r, w_r)$}};
            \node[above] at (1.4,3.9) {\tiny{$(w_r-a, w_r)$}};
            \node[above left] at (-2.5,0.2) {\scriptsize{$(u_l, w_l)$}};
            \node[above] at (-1.4,3.9) {\tiny{$(w_l+a, w_l)$}};

        \end{scope}
    \end{tikzpicture}
    \caption{ This is an example of solution when $f'(u_l)\leq 0 \leq f'(u_r)$; in this example, the flux function is $f(u) = \frac{1}{2} u^2$, the parameters $a = 1$, and the initial data $(u_l,w_l)=(-2,-1.5)$ and $(u_r,w_r)=(1,1.5).$}
    \label{fig: sottocaso3}
    \end{figure}

    \subsection{ \texorpdfstring{$u_l>u_r$:}{ul maggiore ur} shock waves}\label{caso: ul>ur}
    Now we expect $u$ to develop shock discontinuities. Again, in order to study $u(x,\cdot)$, and hence understand $w(x,\cdot)$, we need to distinguish the cases when the shock waves propagate with null, positive or negative speed.
    
    \subsubsection{\texorpdfstring{$f(u_l)=f(u_r)$:}{f(ul) uguale f(ur)}}\label{sottocaso: ful=fur}
    We expect the formation of a shock for $u$ with $0$ speed, so, for each $x$, $u(x,\cdot)\equiv u_0(x)$, and, by imposing $w= [\F(u,w_0)]$, $w(x,\cdot)\equiv w_0(x)$. Then the couple $(u(x,t),w(x,t))(x,t)= (u_0(x),w_0(x))$, satisfying the \RH condition \eqref{eq: hrh}, is also an entropy weak solution to our problem. 
    
    \begin{rmk}\label{rmk: entropyshock1}
    Even if the case $f(u_l)=f(u_r)$ might seem trivial, we wanted to highlight it. Indeed, it may happen that $(u_l-u_r+w_l-w_r)=0$, with $u_l\not=u_r$, so the \RH condition \eqref{eq: hrh} does not give any information about the speed of the shock. Of course,  we must choose $\sigma'(t)\equiv0$  to have $w=[\F(u,w_0)]$, but, even without imposing $w=[\F(u,w_0)]$, we will show that the only entropy solution is the one with $\sigma'(t)\equiv 0$, see Proposition \ref{prop: entropia}.
    \end{rmk}
    
    \subsubsection{ \texorpdfstring{$f(u_l)>f(u_r)$:}{f(ul) maggiore f(ur)}}\label{sottocaso: ful>fur}
    In this subcase, we expect the discontinuity to propagate with positive speed hence for $x<0$, $u(x,\cdot)\equiv u_l$ so also $w(x,\cdot)\equiv w_l$. Instead, for each $x>0$, we expect $t\mapsto u(x,t)$ to increase from $u_r$ to $u_l.$ Similarly to Subcase \ref{sottocaso: 0<ul<ur} we consider the following flux 
    \[\hat{f}_{w_r} (u)= \begin{cases}
        f(u) \quad& u_r \leq u \leq w_r+a,\\
        \frac{1}{2}f(u) + \frac{1}{2} f(w_r+a) \quad &w_r+a \leq u,
    \end{cases}\] 
and again, at least formally, $u$ should satisfy 
\[
\partial_t u+\partial_x \hat{f}_{w_r}(u)=0.
\]
    When $u_r\leq u_l\leq w_r+a$, then $\hat{f}_{w_r} = f(u)$ in $[u_r,u_l]$ and the solution for $u$ consists of only one shock and again $w_t=0$. The other extreme case, that is when $u_r= w_r+a$, generates a solution with a single shock for both $u,w,$ connecting $(u_r,u_r-a)$ and $(u_l,u_l-a).$ In all these cases, the speed of the shock is given by the \RH condition \eqref{eq: hrh}. \par
    Instead, when $u_r<w_r+a<u_l$, in order to solve the Riemann problem we analyse the upper concave envelope of $\hat{f}_{w_r}$ between $u_r$ and $u_l$, since $u_l>u_r$ and $\hat{f}_{w_r}$ is certainly not concave. We then denote the following two quantities 
    \[\mu_r:= \frac{f(w_r+a)-f(u_r)}{(w_r+a)-u_r}, \quad \text{and} \quad \mu_l := \frac{1}{2}\frac{f(u_l)-f(w_r+a)}{u_l -(w_r+a)},\] 
    used to construct the upper concave envelope. Notice first of all that $\mu_l>0,$ as $f$ is convex, $u_r\leq w_r+a\leq u_l$ and $f(u_l)>f(u_r)$.
    
    \par  $\boldsymbol{\mu_r\geq \mu_l:}$
    Under this assumption, the upper concave envelope between $[u_r,u_l]$ of the graph of $\hat{f}_{w_r}$ consists of the two segments connecting the points $(u_r, \hat{f}_{w_r}(u_r)),(w_r+a, \hat{f}_{w_r}(w_r+a))$ and $(w_r+a, \hat{f}_{w_r}(w_r+a)), (u_l, \hat{f}_{w_r}(u_l))$, the first segment with slope $\mu_r$ and the second with slope $\mu_l$, see Figure \ref{fig: convessof}, left. As a result, when solving the Riemann problem we end up with two shock waves for $u$, one with velocity $\mu_r$ and the other $\mu_l$. By imposing the relationship $w=[\F(u,w_0)]$, we then observe that $w$ has only one discontinuity with speed $\mu_l$ connecting the state $w_r$ to $u_l-a$, see Figure \ref{fig: nofastshock}. Notice that the \RH condition \eqref{eq: hrh} holds for both shocks, consequently \eqref{eq: hweaksol} also holds. In fact, for the fastest one, only $u$ is discontinuous, and \RH gives $\mu_r$. For the second, instead, we have a discontinuity with the same amplitude for both $u$ and $w$, and \RH gives $\mu_l.$
    \par $\boldsymbol{\mu_r< \mu_l:} $ In this case, the upper concave envelope of $\hat{f}_{w_r}$ consist in just one segment with slope \[ \frac{\frac{1}{2} f(u_l)+\frac{1}{2}f(w_r+a)-f(u_r)}{u_l-u_r}= \frac{\mu_l I_l+\mu_r I_r}{I_r+I_l}=\tilde\mu,\] where we denoted $I_r=(w_r+a)-u_r$ and $I_l = u_l-(w_r+a)$, see Figure \ref{fig: convessof}, right. So the solution $u$ should consists of one jump discontinuity connecting $u_r$ and $u_l$ with speed given by the slope of that segment. By imposing the relationship $w= \F (u)$, it should hold that $w$ has a shock discontinuity connecting the states $w_r$ and $u_l-a$ with the same speed. However, the couple $(u,w)$ constructed in such a way is not a weak solution to the PDE. Indeed the slope $\tilde{\mu}$ does not satisfy the \RH condition which gives 
    \begin{equation}\label{eq: mu}
        \mu= \frac{f(u_l)-f(u_r)}{u_l-u_r+(u_l-a)-w_r}=\frac{2I_l\mu_l+I_r\mu_r}{I_r+2I_l}.
    \end{equation}
    Necessarily, since we only have one discontinuity both in $u$ and $w$, we must impose its velocity to be equal to $\mu$, so that $(u,w)$ is indeed an weak solution, see Figure \ref{fig: sottocaso4fs}. Notice that, by doing so, we still ensure that $w= [\F(u,w_0)]$. We will call this kind of shock a ``\textbf{fast shock}". 

    \begin{figure}
        \centering
        \begin{tikzpicture}
        \begin{scope}[scale = 3]   
           \draw[->, gray] (-0.15,0) -- (1.7,0) node[right] {$u$};
           \draw[->,gray] (0,-0.15) -- (0,1.7);

           \draw[domain= -0.15:1, smooth, samples=200, variable=\x, gray] plot ({\x}, {(\x)*(\x)});
           
           \draw[domain= 1:1.65, smooth, samples=200, variable=\x, gray]  plot ({\x}, {(1/2)*(\x)*(\x) + 0.5});

           \draw (0.5,0.25)-- (1,1)--(1.5,1.625);

           \draw[thin, dashed] (0.5,0.25)--(0.5,0) node[below] {\scriptsize{$u_r$}};
           \draw[thin, dashed] (1,1)--(1,0) node[below] {\scriptsize{$w_r+a$}};
           \draw[thin, dashed] (1.5,1.625)--(1.5,0)node[below] {\scriptsize{$u_l$}};

           \node[gray] at (1.7, 1.5) {$\hat{f}_{w_r}$};

        \end{scope}
        \begin{scope}[xshift =  10cm, scale= 3]
        \draw[->, gray] (-1.15,0) -- (1.65,0) node[right] {$u$};
           \draw[->,gray] (0,-0.15) -- (0,1.7);

           \draw[domain= -1.15:0, smooth, samples=200, variable=\x,  gray] plot ({\x}, {(\x)*(\x)});
           
           \draw[domain= 0:1.65, smooth, samples=200, variable=\x, gray]  plot ({\x}, {(1/2)*(\x)*(\x)});

           \draw[dash dot] (-1,1)-- (0,0)--(1.5,1.125);
           \draw (-1,1)--(1.5,1.125);

           \draw[thin, dashed] (-1,1)--(-1,0) node[below] {\scriptsize{$u_r$}};
           \draw[thin, dashed] (0,0)--(0,0) node[below] {\scriptsize{$w_r+a$}};
           \draw[thin, dashed] (1.5,1.125)--(1.5,0) node[below] {\scriptsize{$u_l$}};
           
           \node[gray] at (1.6, 0.9) {$\hat{f}_{w_r}$};
        \end{scope} 
        \end{tikzpicture}   
    \caption{Examples of $\hat{f}_{w_r}$. Left: if $\mu_r \geq \mu_l$, the upper concave envelope between $[u_r,u_l]$ of the graph of $\hat{f}_{w_r}$ consists of the two segments drawn in the picture. Right: if $\mu_r < \mu_l$,  the upper concave envelope consists only of one segment.}
    \label{fig: convessof}
\end{figure}
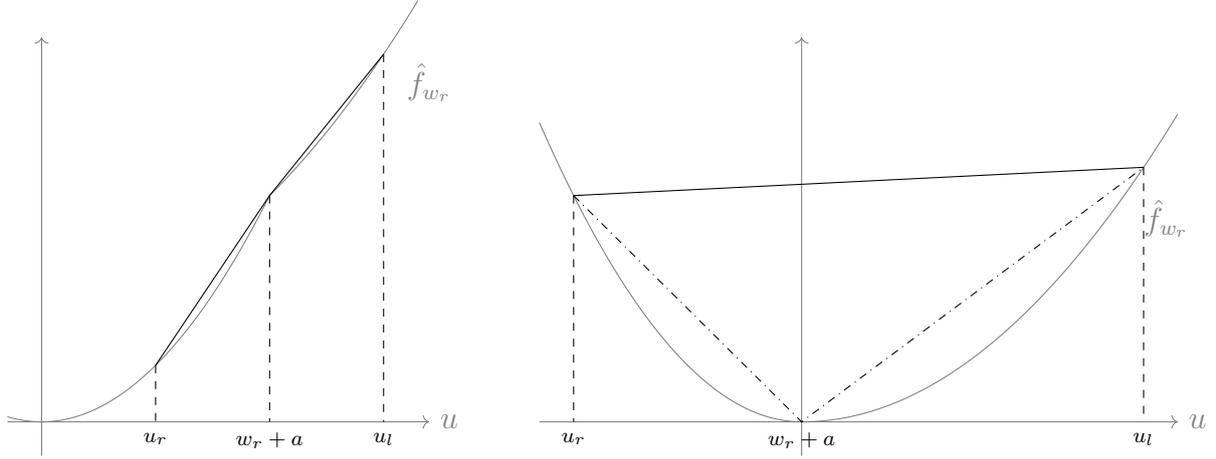

     \begin{rmk}\label{rmk: entropyshock2}
          In the case $\mu_r > \mu_l$ one could also considered the couple consisting of one ``fast shock" with slope given by $\mu$. Such couple $(u,w)$ still satisfies the \RH condition with also $w= [\F(u,w_0)]$ but, as we  will see in Section \ref{S6}, it does not satisfy the entropy weak formulation \eqref{eq: hweaksol}.
     \end{rmk}
      \begin{rmk}
         In the case when $\mu_r < \mu_l$, the heuristic idea is that, since the shock wave is too fast, the couple $(u,w)$ has no time to reach the intermediate state $(u_r,w_r+a)$ but it directly jumps to the state $(u_l,u_l-a)$. Moreover, the slope $\mu$, as defined in \eqref{eq: mu}, can be interpreted as a weighted average of the velocities of the two shocks, $\mu_l$ and $\mu_r$, with respective weights $I_r$ and $2 I_l$. The weight $I_r$ corresponds to the length of the interval in $u$ where $w$ remains constant, while $I_l$ represents the length of the interval in $u$ where both $u$ and $w$ change values, which is equivalent to the length of the corresponding interval in $w$. This justifies the factor $2$ in front of $I_l$.
     \end{rmk}
     \begin{rmk}\label{rmk: mu1=mu2}
         It can be easily seen that if $\mu_r=\mu_l$ then also $\mu=\mu_r=\mu_l$. Hence the two approaches coincide. 
     \end{rmk}
         
     \begin{figure}
    \centering
    \begin{tikzpicture}
        \begin{scope}[scale=1.2727, yshift = 1cm]
            \draw[->, gray] (-3, 0) -- (2.5, 0) node[right] {$u$}; 
            \draw[->, gray] (0, -2.5) -- (0, 3) node[right] {$w$};
            
            \draw[-, semithick] (-2.5, -1.5) -- (1.5, 2.5);
            \draw[-, dotted, semithick] (-2.5, -1.5) -- (-2.75, -1.75);
            \draw[-, dotted, semithick] (1.5, 2.5) -- (1.75, 2.75);
            
            \draw[-, semithick] (-1, -2) -- (2, 1); 
            \draw[-, dotted, semithick] (-1, -2) -- (-1.25, -2.25); 
            \draw[-, dotted, semithick] (2, 1) -- (2.25, 1.25); 
            
            \draw[->, red, very thick] (0.5,0) -- (1,0)--(1.5,0.5);
            \draw[-, red, very thick, dashed] (1.5,0.5)--(1.5,2);

            \draw (0.5,0) node[circle, fill=black, inner sep=1pt] {} node[above left,xshift=1pt, yshift=-2pt] {\scriptsize{$(u_r, w_r)$}};
            
            \draw (1,0) node[circle, fill=black, inner sep=1pt] {} node[below right,xshift=-3pt, yshift=2pt] {\scriptsize{$(w_r+a, w_r)$}};

            \draw (1.5,0.5) node[circle, fill=black, inner sep=1pt] {} node[below right,xshift=-1pt, yshift=3pt] {\scriptsize{$(u_l, u_l-a)$}};

            \draw (1.5,2) node[circle, fill=black, inner sep=1pt] {} node[below right,xshift=-1pt, yshift=3pt] {\scriptsize{$(u_l, u_l)$}};

        \end{scope}

        \begin{scope}[xshift=6.5cm]
            \draw[->, gray] (-1, 0) -- (4, 0) node[right] {$x$}; 
            \draw[->, gray] (0, -0.5) -- (0, 4) node[above] {$t$};
            
            \draw[semithick, dashed] (0,0) -- (3,4);
            \draw[semithick, dashed] (0,0) -- (2.5,4);
            \draw[semithick, dashed] (0,0)--(0,4);

            \node[above right] at (2.5,0.2) {\scriptsize{$(u_r, w_r)$}};
            
            \node[above] at (2.7,3.9) {\tiny{$(w_r+a, w_r)$}};
            
            \node[above] at (1,3.4) {\scriptsize{$(u_l, u_l-a)$}};
            
            \node[above] at (-1,3.4) {\scriptsize{$(u_l, w_l)$}};

        \end{scope}
    \end{tikzpicture}
    \caption{Example of solution for $u_l>u_r$ and $\mu_r> \mu_l$ with $f(u) = \frac{1}{2} u^2$, $a = 1$, and initial data $(u_l,w_l)=(1.5,2)$ and $(u_r,w_r)=(0.5,0).$}
    \label{fig: nofastshock}
    
    \end{figure}
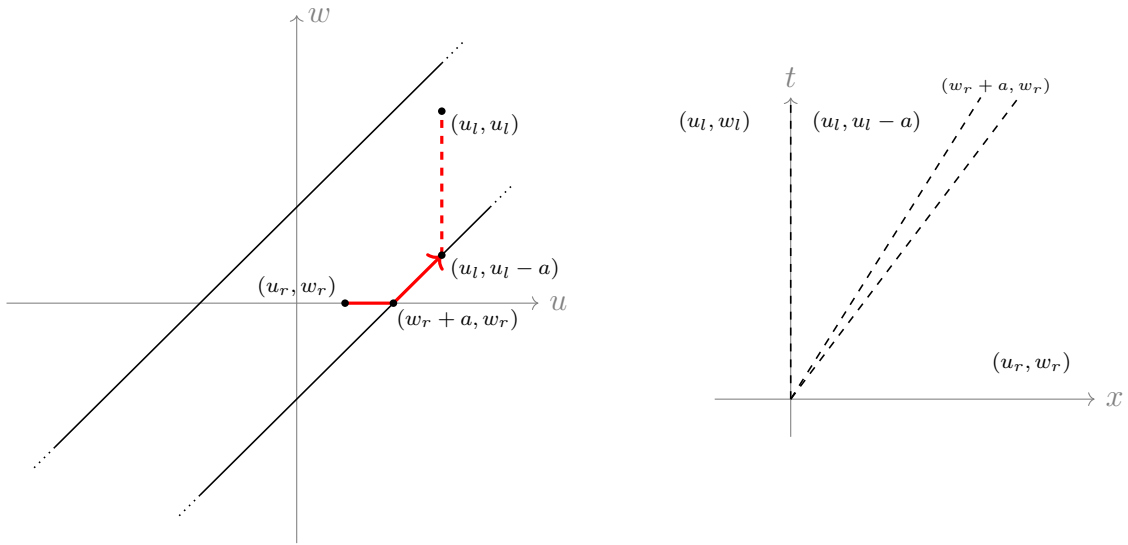

      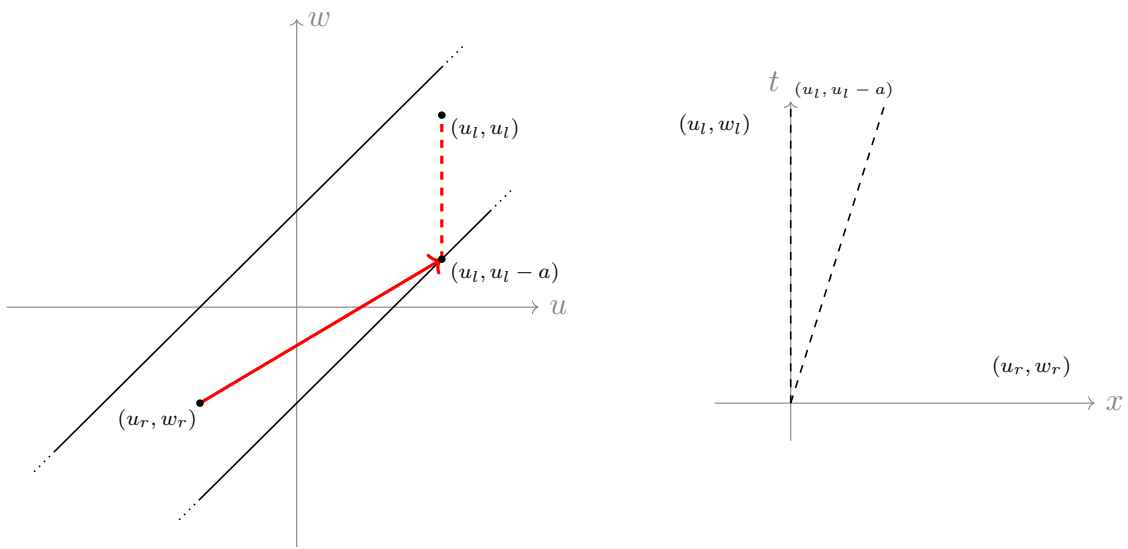
\begin{figure}
    \centering
    \begin{tikzpicture}
        \begin{scope}[scale=1.2727, yshift = 1cm]
            \draw[->, gray] (-3, 0) -- (2.5, 0) node[right] {$u$}; 
            \draw[->, gray] (0, -2.5) -- (0, 3) node[right] {$w$};
            
            \draw[-, semithick] (-2.5, -1.5) -- (1.5, 2.5);
            \draw[-, dotted, semithick] (-2.5, -1.5) -- (-2.75, -1.75);
            \draw[-, dotted, semithick] (1.5, 2.5) -- (1.75, 2.75);
            
            \draw[-, semithick] (-1, -2) -- (2, 1); 
            \draw[-, dotted, semithick] (-1, -2) -- (-1.25, -2.25); 
            \draw[-, dotted, semithick] (2, 1) -- (2.25, 1.25); 
            
            \draw[->, red, very thick] (-1,-1) --(1.5,0.5);
            \draw[-, red, very thick, dashed] (1.5,0.5)--(1.5,2);

            \draw (-1,-1) node[circle, fill=black, inner sep=1pt] {} node[below left,xshift=3pt, yshift=2pt] {\scriptsize{$(u_r, w_r)$}};

            \draw (1.5,0.5) node[circle, fill=black, inner sep=1pt] {} node[below right,xshift=-1pt, yshift=3pt] {\scriptsize{$(u_l, u_l-a)$}};

            \draw (1.5,2) node[circle, fill=black, inner sep=1pt] {} node[below right,xshift=-1pt, yshift=3pt] {\scriptsize{$(u_l, u_l)$}};
        \end{scope}

        \begin{scope}[xshift=6.5cm]
            \draw[->, gray] (-1, 0) -- (4, 0) node[right] {$x$}; 
            \draw[->, gray] (0, -0.5) -- (0, 4) node[above left] {$t$};
            
            \draw[semithick, dashed] (0,0) -- (1.25,4);
            \draw[semithick, dashed] (0,0)--(0,4);

            \node[above right] at (2.5,0.2) {\scriptsize{$(u_r, w_r)$}};
            
            \node[above] at (0.7,3.9) {\tiny{$(u_l, u_l-a)$}};
            
            \node[above] at (-1,3.4) {\scriptsize{$(u_l, w_l)$}};

        \end{scope}
    \end{tikzpicture}
    \caption{Example of solution for $u_l>u_r$ and formation of a ``fast shock" with $f(u) = \frac{1}{2} u^2$, $a = 1$ and  initial data $(u_l,w_l)=(3.5,2)$ and $(u_r,w_r)=(-1,-1).$}  
    \label{fig: sottocaso4fs}
    \end{figure}
    
     \subsubsection{ \texorpdfstring{$f(u_l)<f(u_r)$:}{f(ul) minore f(ur)}}\label{sottocaso: ful<fur}
     In this subcase instead we expect the discontinuity waves to propagate to the left. We then consider the flux 
     \[
     \bar{f}_{w_l}(u)= \begin{cases}
         \frac{1}{2} f(u) + \frac{1}{2} f(w_l-a) \quad& u<w_l-a,\\
         f(u) \quad & w_l-a \leq u \leq u_l,
     \end{cases}\]
     and define \[\nu_l:= \frac{f(u_l)-f(w_l-a)}{u_l -(w_l-a)}, \quad \text{and} \quad \nu_r:= \frac{1}{2} \frac{f(w_l-a)-f(u_r)}{(w_l-a)-u_r}.\] Reasoning as in Subcase \ref{sottocaso: ful>fur}, if $u_l<w_l+a<u_r$ and $\nu_l<\nu_r$ then we have two jump discontinuities, one only for $u$ with velocity $\nu_l$ and one for both $u$ and $w$ with velocity $\nu_r$. If instead $u_l<w_l+a<u_r$ and $\nu_l>\nu_r$, then the weak solution develops a ``fast shock" for both  $u$ and $w$ with slope 
     \[
     \nu  = \frac{f(u_l)-f(u_r)}{u_l-u_r+w_l-(u_r+a)} = \frac{J_l \nu_l+2J_r\nu_r}{J_l+2J_r},
     \]
     given by the \RH condition, where $J_l := u_l-(w_l-a)$ and $J_r = (w_l-a)-u_r$. Notice that this time necessary $\nu_r < 0,$ and the same previous  remarks, referring to this subcase, hold. We can also reason as before to deal with the cases when $u_r \leq w_l+a$ or $u_l=w_l+a$.

     \begin{rmk}\label{rmk: nofastshock}
        In this section, when constructing the solution, we dealt with all the possible cases to highlight the influence of the hysteresis term in the equation. Nevertheless, defining the flux 
         \begin{equation}\label{eq: ftildaw}
             \tilde{f}_w (u)=
             \begin{cases}
                 \frac{1}{2}f(u) + \frac{1}{2} f(w-a) \quad & u< w-a,\\
                 f(u) \quad & w-a \leq u \leq w+a,\\
                 \frac{1}{2} f(u) + \frac{1}{2} f(w+a) \quad & u \geq w + a,
             \end{cases}
         \end{equation}
         we can characterize the solution for $u$ in a more general way. Indeed, except for ``fast shocks",  the restriction of  $u$ on $x>0$ satisfies the PDE 
         \[
         \partial_t u + \partial_x \tilde{f}_{w_r}(u) =0,
         \]
         and its restriction on $x<0$ satisfies 
         \[
         \partial_t u + \partial_x \tilde{f}_{w_l}(u) =0.
         \]
         This general property will be exploited in the next section to define the numerical scheme by only distinguishing whether a ``fast shock" is formed or not, and not by going through all the above cases.
     \end{rmk}


\section{Existence of solutions}
\subsection{A Godunov-type finite volume scheme} \label{sec:godunov}
Let us fix $T>0$ and consider a mesh $\mathcal T$ of $\R$ given by an increasing sequence of points $(x_{i+1/2})_{i\in \mathbb Z}$, such that $\R = \cup_{i\in \mathbb{Z}}[x_{i-1/2},x_{i+1/2}[\,=: \cup_{i\in \mathbb{Z}} K_i $. For simplicity we suppose that $|x_{i-1/2} - x_{i+1/2}|=\Delta x$, for some $\Delta x>0,$ for all $i\in \mathbb{Z}$, and we denote by $x_i$ the middle point of $K_i$. We introduce the time step $\Delta t>0$ and we denote by $t^n:= n \Delta t$ and $K_i^n:=K_i \times [t^n,t^{n+1}[$. Given an initial datum $u_0,w_0$ we set $U_m := \inf_{\R} u_0$ and $U_M:=\sup_\R u_0$, so that $u_0(x)\in [U_m,U_M]$
for each $x$, and we also suppose the following Courant-Friedrichs-Lewy (CFL) condition, see e.g. \cite[Section 6.2]{TORO},  to hold 
\begin{equation}\label{eq: CFL}
    \Delta t \leq \frac{\Delta x}{2 L},
\end{equation}
where $L:=\max_{u\in[U_m,U_M]}|f'(u)|$ is the Lipschitz constant of $f$ on the interval $[U_m,U_M]$. Given piecewise constant data at time $t^n$, 
\[u(x, t^n)= \sum_{i\in \mathbb Z} u_i^n \1_{[x_{i-1/2},x_{i+1/2}[}\quad \text{and}\quad w(x, t^n)= \sum_{i\in \mathbb Z} w_i^n \1_{[x_{i-1/2},x_{i+1/2}[},\] 
we also denote by $\tilde{u}^n(x,t)$ and $\tilde{w}^n(x,t)$ the exact solution to \eqref{eq: intro2} defined on $\R \times [t^n,t^{n+1}[$ via the resolution of the Riemann problems centred at $x_{i+1/2}$, $i\in\mathbb{Z}$. 
Notice that  $\tilde{u}^n(x,t)$ and $\tilde{w}^n(x,t)$ are well defined on the whole time interval as the CFL condition \eqref{eq: CFL} ensures not only that no waves cross the cell interfaces $\{x_{i+1/2}\}\times [t^n, t^{n+1}[$, but also that there are no wave interactions in $K_i^n$ for any $n$ and $i$.

A finite volume scheme for the approximation of problem \eqref{eq: intro2}, with mesh $\mathcal T$ and time step $\Delta t$, writes as follows
\begin{equation} \label{eq: schemaTot}
\begin{cases}
    u_i^{n+1}+w_i^{n+1} = u_i^{n}+w_i^{n}-\frac{\Delta t}{\Delta x} \left(f^n_{i+1/2}-f^n_{i-1/2}\right), \quad &\forall i\in \mathbb{Z},\\[5pt]
    u_i^0+w_i^0 = \frac{1}{\Delta x} \int_{x_{i-1/2}}^{x_{i+1/2}} \left(u_0(x)+w_0(x)\right) \, dx, \quad&\forall i\in \mathbb{Z},
\end{cases}
\end{equation}
where $u_i^n,w_i^n$ are the averages of $u$ and $w$ respectively on $K_i \times \{t^n\}$ and $f^n_{i+1/2}$ denotes the numerical flux at the interface $x_{i+1/2}$, given by the Godunov flux  $f^n_{i+1/2}=f(\tilde{u}^n(x_{i+1/2},t))$, see e.g. \cite[Section 6.2]{TORO}. Notice that $\tilde{u}^n(x_{i+1/2},t)$ is constant in time as it is the solution to a Riemann problem centred at $x_{i+1/2}$, and that if $\tilde{u}^n(x_{i-1/2}-,t)\not=\tilde{u}^n(x_{i-1/2}+,t)$ then  $f(\tilde{u}^n(x_{i-1/2}-,t))=f(\tilde{u}^n(x_{i-1/2}+,t))$ by the \RH condition, so $f^n_{i+1/2}$ is well defined for every $i\in \mathbb{Z}$. Such numerical flux depends only on the Riemann data of $u_i^n,u_{i+1}^n$ (see Section \ref{S2}) and, as noticed in \cite[Section 13.5]{LEVEQUE}, it can be written as $f_{i+1/2}^n=g(u_{i}^n,u_{i+1}^n)$ where
\begin{equation}\label{eq: gGodunov}
    g(u_l,u_r) = 
    \begin{cases}
        \min\limits_{u\in[u_l,u_r]} f(u) \quad & u_l\leq u_r,\\
        \max\limits_{u\in[u_r,u_l]} f(u) \quad & u_l\geq u_r.
    \end{cases}
\end{equation}

The recursive scheme \eqref{eq: schemaTot} allows us to compute the approximate sum of the two variables $u+w$, but not the single values of the unknowns. To complete the algorithm, we need to compute $u_i^n$, $u_{i+1}^n$.
To this end, we analyse what happens to the exact solutions $\tilde{u}^n$ in a cell  $K_i^n$. 
This solution is given by the union of (the restrictions to the corresponding quarters of plane of) the solution to the Riemann problem centered at $(x_{i-1/2},t^n)$ with data $(u_{i-1}^n,w_{i-1}^n)$ and $(u_i^n,w_i^n)$, and the solution to the Riemann problem centred at $(x_{i+1/2},t^n)$ with data $(u_{i}^n,w_{i}^n)$ and $(u_{i+1}^n,w_{i+1}^n)$.

\begin{figure}
    \centering
    \begin{tikzpicture}[scale = 1.5]

      \draw[thick] (0,0) rectangle (4,2);
      \node at (2,-0.7) {\scriptsize{$(u_i^n, w_i^n)$}};
    
      \draw[-,thick] (-2,0) -- (0,0);
      \draw[dashed,thick] (-2.5,0) -- (-2,0);

      \draw[-,thick] (-2,2) -- (0,2);
      \draw[dashed,thick] (-2.5,2) -- (-2,2);
     
      \node at (-1.2,-0.3) {\scriptsize{$(u_{i-1}^n, w_{i-1}^n)$}};
    
      \draw[-,thick] (4,0) -- (6,0);
      \draw[dashed,thick] (6,0) -- (6.5,0);
      
      \draw[-,thick] (4,2) -- (6,2);
      \draw[dashed,thick] (6,2) -- (6.5,2);

      \node at (5.2,-0.3) {\scriptsize{$(u_{i+1}^n, w_{i+1}^n)$}};

      \draw[-,thick] (0,2) -- (0,2.5);
      \draw[dashed,thick] (0,2.5) -- (0,3);
      
      \draw[-,thick] (4,2) -- (4,2.5);
      \draw[dashed,thick] (4,2.5) -- (4,3);

      \draw[-,thick] (0,0) -- (0,-0.25);
      \draw[dashed,thick] (0,-0.25) -- (0,-0.5);
      
      \draw[-,thick] (4,0) -- (4,-0.25);
      \draw[dashed,thick] (4,-0.25) -- (4,-0.5);


      \draw[thin] (0,0) -- (1.5,2);
      \draw[thin] (0,0) -- (1.4,2);
      \draw[thin] (0,0) -- (1.3,2);
      \draw[thin] (0,0) -- (1.2,2);

      \draw[thin] (0,0) -- (0.1,2);
      \draw[thin] (0,0) -- (0.2,2);
      \draw[thin] (0,0) -- (0.3,2);
      \draw[thin] (0,0) -- (-0.1,2);
      \draw[thin] (0,0) -- (-0.2,2);

      \draw[thin] (0,0) -- (-1.1,2);
      \draw[thin] (0,0) -- (-1.2,2);
      \draw[thin] (0,0) -- (-1.3,2);
      \draw[thin] (0,0) -- (-1.4,2);

      \draw[thin, dashed] (4,0) -- (3.3,2);
      \draw[thin, dashed] (4,0) -- (2.4,2);
      
      \draw[dash dot] (2,-0.2)--(2,2.7);

      \node at (-2.7,0.2) {\scriptsize{$\{t=t^n\}$}};
      \node at (-2.7,2.2) {\scriptsize{$\{t=t^{n+1}\}$}};

      \node at (2,-0.37) {\scriptsize{$\{x=x_{i}\}$}};
      \node at (0,-0.7) {\scriptsize{$\{x=x_{i-1/2}\}$}};
      \node at (4,-0.7) {\scriptsize{$\{x=x_{i+1/2}\}$}};
      
      \node at (1,0.2) {\scriptsize{$K_{i}^{n,l}$}};
      \node at (3,0.2) {\scriptsize{$K_{i}^{n,r}$}};
    
    \end{tikzpicture}
    
    \caption{The cell $K_i^n$ split into $K_i^{n,l}$ and $K_i^{n,r}$}
    \label{fig: cella}
\end{figure}
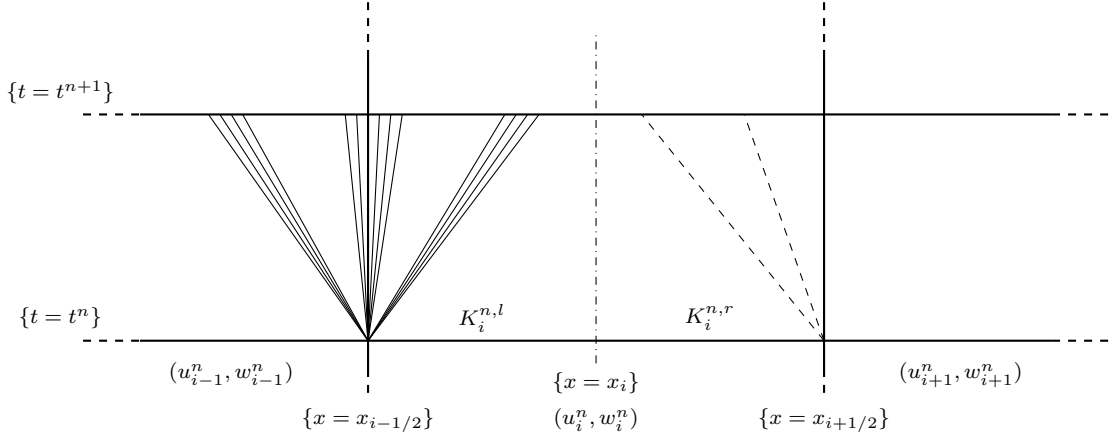

Suppose for now that in $K_i^n$ there are no ``fast shocks". Then, by Remark \ref{rmk: nofastshock}, $\tilde{u}^n$ satisfies weakly the following conservation law
\begin{equation}\label{eq: ftildewPDE}
    \partial_t u+ \partial_x \tilde{f}_{w^n_i}(u)=0,
\end{equation}
in the whole cell $K_i^n$, where $\tilde{f}_{w^n_i}$ is defined by \eqref{eq: ftildaw}. Hence, using the integral formulation of \eqref{eq: ftildewPDE}, we get the scheme for $u$ \begin{equation}\label{eq: nofastshock}
    u_i^{n+1}= u_i^n -\frac{\Delta t}{\Delta x } (\tilde{f}_{w_i^n}(\tilde{u}^n(x_{i+1/2}-,t))-\tilde{f}_{w_i^n}(\tilde{u}^n(x_{i-1/2}+,t))).
\end{equation}
If instead there is a ``fast shock" in $K_i^n$, then the identification of the flux is not straightforward in general .
We then subdivide $K_i^n$ as 
\[K_i^n = [x_{i-1/2},x_i[\,\times [t^n,t^{n+1}[\,\cup [x_i, x_{i+1/2}[\,\times [t^n,t^{n+1}[\,=:K_i^{n,l}\cup K_i^{n,r},\]
and we consider separately the averages of $\tilde{u}^n(x,t^{n+1})$ 
on $[x_{i-1/2},x_i[$ and $[x_i, x_{i+1/2}[$, which we denote respectively by ${u_{i}^{n+1,l}}$ and ${u_{i}^{n+1,r}}$. Indeed, thanks to the CFL condition \eqref{eq: CFL}, ${u_{i}^{n+1,l}}$ and ${u_{i}^{n+1,r}}$ depend respectively only on the solution to the Riemann problems centred at $(x_{i-1/2},t^n)$ and $(x_{i+1/2},t^n)$. 

If there is no ``fast shock" in $K_i^{n,l}$, then \eqref{eq: ftildewPDE} still holds in $K_i^{n,l}$. Hence, applying the discrete formulation of \eqref{eq: ftildewPDE} in $K_i^{n,l}$,
\[\begin{split}
 {u_{i}^{n+1,l}}&= u_i^n -2\frac{\Delta t}{\Delta x } \left(\tilde{f}_{w_i^n}(\tilde{u}^n(x_{i}^n,t))-\tilde{f}_{w_i^n}(\tilde{u}^n(x_{i-1/2}+,t))\right)\\
 &=  u_i^n -2\frac{\Delta t}{\Delta x } \left(f(u_i^n)-\tilde{f}_{w_i^n}(\tilde{u}^n(x_{i-1/2}+,t))\right).
 \end{split}\]
 Notice that the space interval here has length $\Delta x /2$ and that  $\tilde{u}^n(x_{i}^n,t)= u_i^n$ as ensured by the CFL condition \eqref{eq: CFL}.\\ Otherwise, if a ``fast shock" is generated, then we compute explicitly 
 \begin{equation}\label{eq: media1fastshock}
 {u_{i}^{n+1,l}} = u_i^n + 2\frac{\Delta t}{\Delta x} \mu ( u_{i-1}^n-u_{i}^n)
  \end{equation}
as the jump between $u_i^n$ and $u_{i-1}^n$ travels with speed $\mu$ given by (recall Subcase \ref{sottocaso: ful>fur})
\[
\mu=\mu(u_{i-1}^n,u_i^n,w_i^n) = \frac{f(u_{i-1}^n)-f(u_{i}^n)}{u_{i-1}^n-u_{i}^n+(u_{i-1}^n-a)-w_{i}^n}.
\]
Since $\tilde{u}^n(x_i,t)\equiv u_i^n$, then at least formally \eqref{eq: media1fastshock} means that the total flux crossing the interface $x_{i-1/2}\times [t^n,t^{n+1}[$ is $\mu (u_{i-1}^n-u_i^n) +f(u_i^n)$.

Similarly, if there is no ``fast shock" in $K_i^{n,r}$, then
\[
{u_{i}^{n+1,r}} = u_i^n - 2 \frac{\Delta t}{ \Delta x} \left(\tilde{f}_{w_i^n}(\tilde{u}^n(x_{i+1/2}-,t))-f(u_i^n)\right),
\]
while, if a ``fast shock" is present, then
\[
{u_{i}^{n+1,r}} = u_i^n - 2 \frac{\Delta x }{\Delta t} \nu (u_i^n-u_{i-1}^n),
\]
where (recall Subcase \ref{sottocaso: ful<fur})
\[
\nu=\nu(u_{i}^n,u_{i+1}^n,w_i^n) = \frac{f(u_{i}^n)-f(u_{i+1}^n)}{u_{i}^n-u_{i+1}^n+ w_{i}^n-(u_{i+1}^n+a)}.
\]
We then finally deduce that in general 
\begin{equation}\label{eq: h1piumenofinale}
u_i^{n+1}= u_i^n - \frac{\Delta t}{\Delta x} \left(h_1^-(u_{i}^n,u_{i+1}^n,w_i^n)-h_1^+(u_{i-1}^n,u_i^n,w_{i}^n)\right),
\end{equation}
where 
\begin{equation}\label{eq: h1meno}
    h_1^+(u_{i-1}^n,u_i^n,w_{i}^n) =
    \begin{cases}
        \tilde{f}_{w_i^n}(\tilde{u}^n(x_{i-1/2}+,t))  &\quad \text{if no ``fast shock" in $K_i^{n,l}$,}\\
        \mu (u_{i-1}^n-u_i^n) +f(u_i^n) &\quad \text{if ``fast shock" in $K_i^{n,l}$,}
    \end{cases}
\end{equation}
and
\begin{equation}\label{eq: h1piu}
     h_1^-(u_{i-1}^n,u_i^n,w_{i}^n) =
    \begin{cases}
        \tilde{f}_{w_i^n}(\tilde{u}^n(x_{i+1/2}-,t))  &\quad \text{if no ``fast shock" in $K_i^{n,r}$,}\\
        \nu (u_{i+1}^n-u_i^n) +f(u_i^n)&\quad \text{if ``fast shock" in $K_i^{n,r}$}.
    \end{cases}
\end{equation}
Once $u_i^{n+1}$ is computed by \eqref{eq: h1piumenofinale}, we use \eqref{eq: schemaTot} to also compute $w_i^{n+1}$. 

Finally, the complete scheme reads as
\begin{equation}\label{eq: schema}
    \begin{cases}
         u_i^{n+1} =  u_i^n- \frac{\Delta t}{\Delta x} \left(h_1^-(u_{i}^n,u_{i+1}^n,w_i^n)-h_1^+(u_{i-1}^n,u_i^n,w_{i}^n)\right),
         \quad& \forall i\in \mathbb Z,\\[5pt]
         w_i^{n+1} = u_i^{n}+w_i^{n}-\frac{\Delta t}{\Delta x} \left(f^n_{i+1/2}-f^n_{i-1/2}\right)-u_i^{n+1}, \quad &\forall i\in \mathbb{Z},\\[5pt]
         u_i^0= \frac{1}{\Delta x} \int_{x_{i-1/2}}^{x_{i+1/2}} u_0(x) \, dx, \quad w_i^0 = \frac{1}{\Delta x} \int_{x_{i-1/2}}^{x_{i+1/2}} w_0(x) \, dx,\quad&\forall i\in \mathbb{Z},\\[5pt]
         |u_0(x)-w_0(x)|\leq a, \quad & \forall x\in \R .
    \end{cases}
\end{equation}
We point out that the scheme is well-posed. Indeed, in order to be able to solve the Riemann problems to define $\tilde{u}^n(x,t)$ and $\tilde{w}^n(x,t)$, due to the hysteresis relationship, we always supposed that $|u_i^n-w_i^n|\leq a$. However, starting with compatible initial data, that is $|u_0(x)-w_0(x)| \leq a$, then 
\[|u_i^0-w_i^0| \leq \frac{1}{\Delta x}\int_{x_{i-1/2}}^{x_{i+1/2}} |u_0(x)-w_0(x)| \,dx \leq a, \quad \forall i\in \mathbb{Z}.\] 
Consequently, since the solution to the Riemann problem preserves the inequality $|u-w|\leq a$ for every $x$ and $t>0$, (see Section~\ref{S2}), it holds
\begin{equation}\label{eq: comp}
|u^{n+1}_i-w_i^{n+1}| \leq \frac{1}{\Delta x} \int_{x_{i-1/2}}^{x_{i+1/2}}|\tilde{u}^n(x,t^{n+1})-\tilde{w}^n(x,t^{n+1})| \, dx \leq a,\quad \forall i\in \mathbb{Z}, \forall n\in \mathbb{N}.
\end{equation}

To study the properties of the scheme, it will be convenient to rewrite \eqref{eq: schema} as 
\begin{equation}\label{eq: schema2}
    \begin{cases}
         u_i^{n+1} =  u_i^n- \frac{\Delta t}{\Delta x} (h_1^-(u_{i}^n,u_{i+1}^n,w_i^n)-h_1^+(u_{i-1}^n,u_i^n,w_{i}^n)),
         \quad& \forall i\in \mathbb Z,\\
         w_i^{n+1} = w_i^{n}-\frac{\Delta t}{\Delta x} (h_2^-(u_{i}^n,u_{i+1}^n,w_i^n)-h_2^+(u_{i-1}^n,u_i^n,w_{i}^n)), \quad &\forall i\in \mathbb{Z},\\
         u_i^0= \frac{1}{\Delta x} \int_{x_{i-1/2}}^{x_{i+1/2}} u_0(x) \, dx, \quad w_i^0 = \frac{1}{\Delta x} \int_{x_{i-1/2}}^{x_{i+1/2}} w_0(x) \, dx,\quad&\forall i\in \mathbb{Z},\\
         |u_0(x)-w_0(x)|\leq a, \quad & \forall x\in \R .
    \end{cases}
\end{equation}
where 
\[
 h_2^{\pm} =f_{i\pm 1/2}^n - h^\pm_1.
\]

We have the following monotonicity properties for $h_j^{\pm}$, $j=1,2.$
\begin{prop}\label{prop: monotonia}
    Given $\alpha, \beta, \gamma \in\R$, the following monotonicity properties hold: 
    \begin{itemize}
        \item if $|\beta - \gamma| \leq a$:
        \begin{itemize}
            \item $h_1^+(\alpha,\beta,\gamma)$ is non-decreasing in $\alpha$ and $\gamma$ and non-increasing in $\beta$; 
            \item $h_2^+(\alpha,\beta,\gamma)$ is non-decreasing in $\alpha$ and $\beta$ and non-increasing in $\gamma$;
        \end{itemize}
        \item if $|\alpha - \gamma| \leq a$:
        \begin{itemize}
            \item $h_1^-(\alpha,\beta,\gamma)$ is non-decreasing in $\alpha$ and non-increasing in $\beta$ and $\gamma$;
            \item $h_2^-(\alpha,\beta,\gamma)$ is non-decreasing in $\gamma$ and non-increasing in $\alpha$ and $\beta$.
        \end{itemize}
    \end{itemize}
\end{prop}

\begin{proof}
    We first suppose $|\beta - \gamma| \leq a$ and focus on $h_1^+$. If $\alpha \leq \beta$, see Section \ref{caso: ul<ur}, then no shock waves are formed so, recalling its definition \eqref{eq: h1meno}, $h_1^+(\alpha,\beta,\gamma) $ is the minimum in $[\alpha,  \beta]$ of $\tilde{f}_\gamma$ as for the standard Godunov flux. More explicitly,
    \begin{equation}\label{eq: h1menorare}
        h_1^+ (\alpha, \beta,\gamma)=
        \begin{cases}
            \min\limits_{u\in[\alpha,\beta]} f(u) \quad & \gamma - a \leq \alpha,\\
             \min\limits_{u\in[\alpha,\beta]} f(u) \quad &\gamma - a \geq \alpha \text{ and } f'(\gamma - a ) \leq 0,\\
             \min\limits_{u\in[\alpha,\beta]} \frac{1}{2}f(u) + \frac{1}{2} f(\gamma -a ) \quad &\gamma - a \geq \alpha \text{ and } f'(\gamma - a ) \geq 0.
        \end{cases}
    \end{equation}
    Indeed, if in the semi-cell $K_i^{n,l}$, $t\mapsto \tilde{u}^n(x,t)$ possibly decreases from $\beta$ to $\alpha$, hence when $\gamma - a \leq \alpha$ the couple $(\tilde{u}^n(x,\cdot),\tilde{w}^n(x,\cdot))$ will always be in the interior of the hysteresis strip, so $\partial_t\tilde{w}^n=0$ and the problem is solved as in the classical way; instead, when  $\gamma - a \geq \alpha \text{ and } f'(\gamma - a ) \leq 0$, we consider $\tilde{f}_\gamma(u)$ and its minimum is in $[\gamma -a,\beta]$, hence when $\tilde{f}_\gamma(u)=f(u)$; otherwise, if $\gamma - a \geq \alpha \text{ and } f'(\gamma - a ) \geq 0$, then the minimum is in $[\alpha, \gamma -a].$ \\
    By looking directly at \eqref{eq: h1menorare}, it can be easily checked that $h_1^+$ is continuous in every variable when the other two are fixed, hence to prove the monotonicity we just need to prove it in every different possible case of \eqref{eq: h1menorare}, ignoring the conditions on the right-hand side of \eqref{eq: h1menorare}. It is straightforward 
    then that $\alpha \mapsto h_1^+(\alpha, \beta, \gamma)$ is non-decreasing and that $\beta \mapsto h_1^+(\alpha, \beta ,\gamma)$ is non-increasing. Also $\gamma \mapsto h_1^+(\alpha, \beta, \gamma)$ is non-decreasing as either $h_1^+$ is constant in $\gamma$ or non-decreasing under the condition $f(\gamma -a )'\geq0$.

    When $\alpha \geq \beta$, see Section \ref{caso: ul>ur}, paying attention to the case of ``fast shocks", we have 
    \begin{equation}\label{eq: h1menoshock}
        h_1^+ (\alpha, \beta,\gamma)=
        \begin{cases}
            \max\limits_{u\in[\beta,\alpha]} f(u) \quad & \gamma + a \geq \alpha,\\
             \max\limits_{u\in[\beta,\alpha]} f(u)  \quad &\gamma + a \leq \alpha \text{ and } f(\alpha) \leq f(\beta) ,\\
             \mu (\alpha- \beta) + f(\beta) \quad & \gamma + a \leq \alpha, f(\alpha) \geq f(\beta) \text{ and } \mu_r \leq \mu_l, \\
             \max\limits_{u\in[\beta,\alpha]} \frac{1}{2}f(u) + \frac{1}{2} f(\gamma +a ) \quad & \gamma + a \leq \alpha, f(\alpha) > f(\beta) \text{ and } \mu_r \geq \mu_l,
        \end{cases}
    \end{equation}
    where $\mu = \mu(\alpha, \beta, \gamma)$. Again it can be checked that $ h_1^+ $ is continuous in each variable when we fix the another two. In particular, notice that, when $f(\alpha)=f(\beta)$, then second and third definitions coincide and, when $\mu_r=\mu_l$, then the last two coincide (see Remark~\ref{rmk: mu1=mu2}). 
    Then again we check the monotonicity for every possible case in \eqref{eq: h1menoshock}. 
    The first two are constant in $\gamma$  and trivially non-decreasing in $\alpha$ and non-increasing in $\beta$. For the third one instead, using the notation introduced in Subcase \ref{sottocaso: ful>fur} with $u_l=\alpha, u_r = \beta$ and $w_r = \gamma,$ we have that since $\partial_\alpha I_r = 0,\partial_\alpha I_l= 1,\partial_\beta I_r = -1,\partial_\beta I_l= 0, \partial_\gamma I_r = 1, \partial_\gamma I_l= -1,$ and $\alpha-\beta = I_r+I_l$ then 
    \begin{align}
         \begin{split}
            \frac{\partial\left(\mu (\alpha- \beta) + f(\beta)\right)}{\partial \alpha}&=  \frac{f'(\alpha)(I_r+2I_l)-2(f(\alpha)-f(\beta))}{(I_r+2I_l)^2} (\alpha-\beta) + \mu \\
             &= \frac{I_r+I_l}{I_r+2I_l} ( f'(\alpha)-2\mu) + \mu, 
            \end{split}
            \label{eq: dh1piudalpha}\\
            \begin{split}
         \frac{\partial\left(\mu (\alpha- \beta) + f(\beta)\right)}{\partial \beta} &= \frac{-f'(\beta)(I_r+2I_l)+(f(\alpha)-f(\beta))}{(I_r+2I_l)^2}(\alpha-\beta)-\mu + f'(\beta) \\
         &= \frac{I_r+I_l}{I_r+2I_l} (\mu - f'(\beta)) - (\mu - f'(\beta))\\
         &= \frac{I_l}{I_r+2I_l} ( f'(\beta)-\mu), 
         \end{split}
         \label{eq: dh1piubeta}\\
         \begin{split}
        \frac{\partial\left(\mu (\alpha- \beta) + f(\beta)\right)}{\partial \gamma} &= \frac{f(\alpha)-f(\beta)}{(I_r+2I_l)^2}(\alpha-\beta)\\
        &= \frac{I_r+I_l}{I_r+2I_l} \mu. 
        \end{split}\label{eq: dh1piudgamma}
    \end{align}
    We observe that \eqref{eq: dh1piudalpha} is positive because $f'(\alpha)-2\mu \geq f'(\alpha)-2\mu_l \geq 0 $ by $\mu_r \leq \mu \leq \mu_l$ and the convexity of $f$, and because $\mu \geq 0 $ since $f(\alpha) \geq f(\beta)$; \eqref{eq: dh1piubeta} is negative since for the same previous reasons $ f'(\beta)-\mu\leq f'(\beta)-\mu_r \leq 0$; \eqref{eq: dh1piudgamma} positive since again $\mu\geq 0$. Regarding the fourth case in \eqref{eq: h1menoshock}, it is non-decreasing in $\alpha$ and non-increasing in $\beta$; it is also non-decreasing in $\gamma$ since $\mu_r \geq \mu_l$ implies $\mu_r \geq0$, hence $f'(\gamma +a)\geq0$ by convexity of $f$.

    Regarding, $h_2^+$, as $h_2^+=f_{i-1/2}^n-h_1^+$ and by
    recalling that $f_{i-1/2}^n=g(\alpha,\beta)$, see \eqref{eq: gGodunov}, we have that when $\alpha \leq \beta$
     \begin{equation}\label{eq: h2menorare}
        h_2^+ (\alpha, \beta,\gamma)=
        \begin{cases}
             0 \quad & \gamma - a \leq \alpha,\\
             0 \quad &\gamma - a \geq \alpha \text{ and } f'(\gamma - a ) \leq 0,\\
             \min\limits_{u\in[\alpha,\beta]} \frac{1}{2}f(u) - \frac{1}{2} f(\gamma -a ) \quad &\gamma - a \geq \alpha \text{ and } f'(\gamma - a ) \geq 0.
        \end{cases}
    \end{equation}
    which, when the other two variables are fixed, is again non-decreasing in $\alpha$, non-increasing in $\gamma$ and constant in $\beta$. Indeed, notice that when $\gamma - a \geq \alpha \text{ and } f'(\gamma - a ) \geq 0$, then the minimum of $1/2 f $ is obtained in $[\alpha, \gamma-a]$ since $f$ is convex and $\beta \geq \gamma-a$, so in this case $h_2^+$ is constant in $\beta.$ Similarly, when $\alpha \geq \beta$,
    \begin{equation}\label{eq: h2menoshock}
        h_2^+ (\alpha, \beta,\gamma)=
        \begin{cases}
             0 \quad & \gamma + a \geq \alpha,\\
             0 \quad &\gamma + a \leq \alpha \text{ and } f(\alpha) \leq f(\beta) ,\\
             I_l\mu \quad & \gamma + a \leq \alpha, f(\alpha) \geq f(\beta) \text{ and } \mu_r \leq \mu_l, \\
             \max\limits_{u\in[\beta,\alpha]} \frac{1}{2}f(u) - \frac{1}{2} f(\gamma + a ) \quad & \gamma + a \leq \alpha, f(\alpha) > f(\beta) \text{ and } \mu_r \geq \mu_l,
        \end{cases}
    \end{equation}
    as for the case $\gamma + a \leq \alpha, f(\alpha) \geq f(\beta) \text{ and } \mu_r \leq \mu_l$ we have $f_{i-1/2}^n=f(\alpha)$, hence
    \[
    f(\alpha)-(\mu(\alpha-\beta)+f(\beta)) = f(\alpha)-f(\beta)-\frac{I_r+I_l}{I_r+2I_l}(f(\alpha)-f(\beta))=I_l\mu.
    \]
    We infer the same conclusion as 
    \begin{align}
       \begin{split}
       \frac{\partial\left( I_l \mu\right)}{\partial \alpha} & = \mu + \frac{f'(\alpha)(I_r+2I_l)-2(f(\alpha)-f(\beta))}{(I_r+2I_l)^2} I_l\\
       &= \mu+\frac{I_l}{I_r+2I_l} (f'(\alpha)-2\mu) \ge 0,
       \end{split}
       \\
       \begin{split}    
       \frac{\partial\left(I_l \mu\right)}{\partial \beta} &= \frac{-f'(\beta)(I_r+2I_l)+(f(\alpha)-f(\beta))}{(I_r+2I_l)^2}I_l\\
       &= -\frac{I_l}{I_r+2I_l} (f'(\beta)-\mu)  \geq 0 ,
       \end{split}
       \\
       \begin{split}
       \frac{\partial\left(I_l \mu\right)}{\partial \gamma} &= -\mu + \frac{f(\alpha)-f(\beta)}{(I_r+2I_l)^2}I_l=-\mu +\frac{I_l}{I_r+2I_l} \mu\\ 
       &=-\frac{I_r+I_l}{I_r+2I_l} \mu \leq 0,
        \end{split}
    \end{align}
    and in the case when $\gamma + a \leq \alpha, f(\alpha) > f(\beta) \text{ and } \mu_r \geq \mu_l$, the maximum of $1/2 f(u)$ equals to $1/2 f(\alpha)$ which is constant in $\beta.$ \par 
    
    Briefly, if we now consider $h_1^-(\alpha, \beta, \gamma)$ with $|\alpha - \gamma| \leq a $ then if $\alpha \leq \beta$
    \begin{equation}\label{eq: h1piurare}
        h_1^- (\alpha, \beta,\gamma)=
        \begin{cases}
            \min\limits_{u\in[\alpha,\beta]} f(u) \quad & \gamma + a \geq \beta,\\
             \min\limits_{u\in[\alpha,\beta]} f(u) \quad &\gamma + a \geq \beta \text{ and } f'(\gamma + a ) \geq 0,\\
             \min\limits_{u\in[\alpha,\beta]} \frac{1}{2}f(u) + \frac{1}{2} f(\gamma +a ) \quad &\gamma + a \geq \beta \text{ and } f'(\gamma + a ) \leq 0.
        \end{cases}
    \end{equation}
    which again is non-decreasing in $\alpha$, non-increasing in $\beta$ and this time also non-increasing in $\gamma.$ The same conclusion holds, for $\alpha \geq \beta $ when
    \begin{equation}\label{eq: h1piushock}
        h_1^+ (\alpha, \beta,\gamma)=
        \begin{cases}
            \max\limits_{u\in[\beta,\alpha]} f(u) \quad & \gamma - a \leq \beta,\\
             \max\limits_{u\in[\beta,\alpha]} f(u) \quad &\gamma - a \geq \beta \text{ and } f(\alpha) \geq f(\beta) ,\\
             -\nu (\alpha- \beta) + f(\alpha) \quad & \gamma - a \geq \beta, f(\alpha) \leq f(\beta) \text{ and } \nu_l \geq \nu_r, \\
             \max\limits_{u\in[\beta,\alpha]} \frac{1}{2}f(u) + \frac{1}{2} f(\gamma -a ) \quad & \gamma - a \geq \beta, f(\alpha) < f(\beta) \text{ and } \nu_l \leq \nu_r.
        \end{cases}
    \end{equation}
    We just point out that 
    \begin{align}
       &\frac{\partial\left( -\nu (\alpha- \beta) + f(\alpha)\right)}{\partial \alpha} = \frac{J_r}{J_l+2J_r} (f'(\alpha)-\nu) \ge 0,\\
       &\frac{\partial\left( -\nu (\alpha- \beta) + f(\alpha)\right)}{\partial \beta} = \frac{J_l+J_r}{J_l+2J_r} (f'(\beta)-2\nu) +\nu \le 0,\\
       &\frac{\partial\left( -\nu (\alpha- \beta) + f(\alpha)\right)}{\partial \gamma} = \frac{J_l+J_r}{J_l+2J_r} \nu \le 0.
    \end{align}
    Regarding $h_2^-$ a similar analysis as for $h_2^+$ can be done, implying the desired conclusions.
\end{proof}

\begin{prop}\label{prop: lip}
     $h_j^{\pm}$ are $L-$Lipschitz continuous with respect to the first two arguments for $j=1,2$. That is 
     \[
     |h_j^{\pm}(\alpha_1,\beta,\gamma)-h_j^{\pm}(\alpha_1,\beta,\gamma)| \leq L | \alpha_1 - \alpha_2|
     \]
     and 
     \[
     |h_j^{\pm}(\alpha,\beta_1,\gamma)-h_j^{\pm}(\alpha,\beta_2,\gamma)| \leq L | \beta_1 - \beta_2|
     \] 
     for $j=1,2$ and for all $\alpha, \alpha_1,\alpha_2$ and $\beta, \beta_1,\beta_2$ compatible with the fixed $\gamma.$
\end{prop}

\begin{proof}
    Since $h_{j}^{\pm}(\alpha,\beta,\gamma)$ are continuous in $\alpha$ and $\beta$, then it is sufficient to check the Lipschitz continuity in each different case of definition. In particular, as $\gamma$ is fixed, it is straightforward to see that $\min_{[\alpha,\beta]} f(u)$ and $\max_{[\beta,\alpha]} f(u)$ are L-Lipschitz continuous with respect to $\alpha$ and $\beta$, so we just prove explicitly the statement in the case of ``fast shocks". Let us now focus on this case for $h_1^+$; then by looking at \eqref{eq: dh1piudalpha} and \eqref{eq: dh1piubeta} we have that 
    \begin{align*}
    &0\leq\frac{\partial\left(\mu (\alpha- \beta) + f(\beta)\right)}{\partial \alpha} \leq \sup\limits_{[U_m,U_M]} |f'(u)| = L,\\
    &0\geq\frac{\partial\left(\mu (\alpha- \beta) + f(\beta)\right)}{\partial \beta} \geq - \frac{1}{2}\left( \sup\limits_{[U_m,U_M]} |f'(u)|+\mu \right)\geq - L,
    \end{align*}
    showing the Lipschitz continuity. Similar estimates also holds in the other cases for all $h_j^{\pm}$. 

\end{proof}
A direct consequence of Proposition \ref{prop: monotonia} and Proposition \ref{prop: lip} is the following corollary.

\begin{cor}\label{cor: monotonia}
    Under the CFL condition \eqref{eq: CFL}, we can write 
    \begin{align}
        u_i^{n+1} = H_1 (u_{i-1}^n,u_i^n,u_{i+1}^n,w_i^n),\\
        w_i^{n+1} = H_2 (u_{i-1}^n,u_i^n,u_{i+1}^n,w_i^n),
    \end{align}
    with $H_1,H_2$ non-decreasing in each of their arguments and such that 
    \begin{equation}\label{eq: H1piuH2}
        \begin{split}
            H_1 (u_{i-1}^n,u_i^n,u_{i+1}^n,w_i^n) + H_2 (u_{i-1}^n,u_i^n,u_{i+1}^n,w_i^n)&= u_{i}^n+w_i^n - \frac{\Delta t}{ \Delta x} \left(f_{i+1/2}^n-f_{i-1/2}^n\right) \\ &= u_{i}^n+w_i^n - \frac{\Delta t}{ \Delta x} \left(g(u_i^n, u_{i+1}^n)-g(u_i^n,u_{i+1}^n)\right).
        \end{split}
    \end{equation}
\end{cor}

\begin{proof}
    By Proposition \ref{prop: lip}, it is clear that
    \[
     H_1 (u_{i-1}^n,u_i^n,u_{i+1}^n,w_i^n) = u_i^n - \frac{ \Delta t}{ \Delta x} \left ( h_1^-(u_{i}^n,u_{i+1}^n,w_i^n)-h_1^+(u_{i-1}^n,u_i^n,w_{i}^n)\right),
    \] 
    is non-decreasing in $u_{i-1}^n,u_{i+1}^n,w_i^n$. Fix now  $u_{i-1}^n,u_{i+1}^n,w_i^n$ and consider $\alpha_1 \leq \alpha_2$, then, by Proposition \ref{prop: lip} and under the CFL condition \eqref{eq: CFL}
    
    \[
    \begin{split}
        H_1 (u_{i-1}^n,\alpha_2,u_{i+1}^n,w_i^n)&-H_1 (u_{i-1}^n,\alpha_1,u_{i+1}^n,w_i^n) =\\
        &= \alpha_2-\alpha_1 - \frac{\Delta t}{\Delta x}\left ( h_1^-(\alpha_2,u_{i+1}^n,w_i^n)-h_1^-(\alpha_1,u_{i+1}^n,w_{i}^n)\right)
        \\&\quad+ \frac{\Delta t}{ \Delta x}\left ( h_1^+(u_{i-1}^n,\alpha_2,w_i^n)-h_1^+(u_{i-1}^n,\alpha_1,w_{i}^n)\right) \\
        & \geq \alpha_2-\alpha_1- 2\frac{\Delta t}{ \Delta x } L (\alpha_2-\alpha_1) \geq 0.
    \end{split}
    \]
    Similarly,
     \[
     H_2 (u_{i-1}^n,u_i^n,u_{i+1}^n,w_i^n) = w_i^n - \frac{ \Delta x}{ \Delta t} \left ( h_2^-(u_{i}^n,u_{i+1}^n,w_i^n)-h_2^+(u_{i-1}^n,u_i^n,w_{i}^n)\right),
    \] 
    is non-decreasing $u_{i-1}^n,u_i^n,u_{i+1}^n $ as a direct consequence of Proposition \ref{prop: monotonia}, and, reasoning in the same way as for $u$, it can be proved that it is non-decreasing also in $w_i^n$ by Proposition \ref{prop: lip} and the CFL condition \eqref{eq: CFL}.
\end{proof}

\subsection{Compactness estimates}
    Given $u_{i-1}^n,u_i^n,u_{i+1}^n$ and $w_i^n$ we define \begin{equation}
        a_{i-1/2}^n: = \begin{cases}
        \frac{\Delta t}{ \Delta x}\frac{h_1^+(u_{i-1}^n,u_i^n,w_i^n)-f(u_i^n)}{u_{i-1}^n-u_{i}^n} \quad& \text{if } u_i^n \not= u_{i-1}^n,\\
        0 \quad& \text{if } u_i^n = u_{i-1}^n,
        \end{cases}
    \end{equation}
    and 
    \begin{equation}
        b_{i+1/2}^n: = \begin{cases} 
        \frac{\Delta t }{\Delta x}\frac{h_1^-(u_{i}^n,u_{i+1}^n,w_i^n)-f(u_i^n)}{u_{i}^n-u_{i+1}^n} \quad& \text{if } u_i^n \not= u_{i+1}^n,\\
        0 \quad& \text{if } u_i^n = u_{i+1}^n.
        \end{cases}
    \end{equation}
 So, we can rewrite the scheme for $u_i^{n+1}$ in the two following ways
 \begin{subequations}\label{eq: umi}
     \begin{align}
          u_i^{n+1}&= u_i^n +b_{i+1/2}^n (u_{i+1}^n-u_i^n)+  a_{i-1/2}^n(u_{i-1}^n-u_{i}^n) \label{eq: umia}\\
            &=u_i^n \left(1-b_{i+1/2}^n-a_{i-1/2}^n\right)+ b_{i+1/2}^n u_{i+1}^n+  a_{i-1/2}^n u_{i-1}^n, \label{eq: umib}
     \end{align}
 \end{subequations}
as $h_1^+(u_{i-1}^n,u_i^n,w_i^n)=f(u_i^n)$ when $u_{i-1}^n=u_i^n$ and $h_1^-(u_{i}^n,u_{i+1}^n,w_i^n)=f(u_i^n)$ when $u_{i}^n=u_{i+1}^n$.
Similarly, define 
\begin{equation}
        c_{i-1/2}^n: = \begin{cases}
        \frac{\Delta t}{\Delta x}\frac{h_2^+(u_{i-1}^n,u_i^n,w_i^n)}{w_{i-1}^n-w_{i}^n} \quad& \text{if } w_i^n \not= w_{i-1}^n,\\
        0 \quad& \text{if } w_i^n = w_{i-1}^n,
        \end{cases}
    \end{equation}
    and 
    \begin{equation}
        d_{i+1/2}^{\,n}: = \begin{cases}
        \frac{\Delta t}{\Delta x}\frac{h_2^-(u_{i}^n,u_{i+1}^n,w_i^n)}{w_{i}^n-w_{i+1}^n} \quad& \text{if } w_i^n \not= w_{i+1}^n,\\
        0 \quad& \text{if } w_i^n = w_{i+1}^n.
        \end{cases}
    \end{equation}
    Notice now that if $w_{i-1}^n = w_{i}^n$ then $|w_i^n-u_{i-1}^n|<a$, so when solving the Riemann problem at $x_{i-1/2}$, it can be checked that $w$ remains constant in $K_i^{n,l}$, since the couple $(\tilde{u}^n,\tilde{w}^n)$ remains into the hysteresis region $\mathcal{L}$. This implies then that $h_1^+(u_{i-1}^n,u_i^n,w_i^n)=f_{i-1/2}^n$ and $h_2^+(u_{i-1}^n,u_i^n,w_i^n)=0$. Similarly, when $w_i^n=w_{i+1}^{n}$ then $h_2^-(u_{i}^n,u_{i+1}^n,w_i^n)=0$. With this in mind, the scheme for $w$ can be rewritten as 
    \begin{subequations}\label{eq: wmi}
         \begin{align}
          w_i^{n+1}&= w_i^n + d_{i+1/2}^n (w_{i+1}^n-w_i^n)+ c_{i-1/2}^n(w_{i-1}^n-w_{i}^n) \label{eq: wmia}\\
        &=w_i^n \left(1-d_{i+1/2}^n+c_{i-1/2}^n\right)+ d_{i+1/2}^n w_{i+1}^n+ c_{i-1/2}^n w_{i-1}^n. \label{eq: wmib}
         \end{align}
     \end{subequations}
\begin{lem}\label{lemma: miki}
    Under the CFL condition \eqref{eq: CFL}, for all $i\in \mathbb{Z}$ and $n\in \N,$ it holds that $0\leq a_{i-1/2}^n,b_{i+1/2}^n \leq 1/2$ and $0\leq c_{i-1/2}^n,d_{i+1/2}^n \leq 1/2$.
\end{lem}
\begin{proof} 
    Positivity of $a_{i-1/2}^n$ and $b_{i+1/2}^n$ follows from the monotonicity properties of $h_1^{\pm}$, see Proposition \ref{prop: monotonia}, and from the fact that $h_1^{\pm}(u_i^n,u_i^n,w_i^n)=f(u_i^n)$. The estimate $a_{i-1/2}^n,b_{i+1/2}^n \leq 1/2$ follows from the Lipschitz continuity of $h_1^{\pm}$, see Proposition \ref{prop: lip}, together with the CFL condition \eqref{eq: CFL}. \par
    
    Regarding $c_{i-1/2}^n$ it is easy to see that if $|u_{i-1}^n-w_i^n|\leq a$ then $h_2^+=0$, see indeed \eqref{eq: h2menorare} and \eqref{eq: h2menoshock}. Suppose then for example that $u_{i-1}^n < w_i^n-a$, which in particular implies that $0\leq w_i^n-w_{i-1}^n \leq u_i^n-u_{i-1}^n$ as $w_{i-1}^n -a \leq  u_{i-1}^n < w_i^n-a  \leq u_{i}^n$.
    Hence, as $h_2^+(u_{i}^n,u_i^n,w_i^n)=0 $, then we can use the monotonicity properties of Proposition \ref{prop: monotonia}, the Lipschitz continuity stated in Proposition \ref{prop: lip} and the CFL condition \eqref{eq: CFL} to conclude that 
    \[
    0 \leq c_{i-1}^n = \frac{\Delta t}{\Delta x} \frac{h_2^+(u_{i-1}^n,u_i^n,w_i^n)-h_2^+(u_{i}^n,u_i^n,w_i^n)}{u_{i-1}^n-u_i^n} \frac{u_{i-1}^n-u_i^n}{w_{i-1}^n-w_i^n} \leq \frac{1}{2}.
    \]
    With similar reasoning we can prove the estimate when $u_{i-1}^n > w_i^n+a$, noticing that is such case $0 \leq w_{i-1}^n-w_{i}^n \leq u_{i-1}^n-u_{i}^n.$ Also the estimates for $d^n_{i+1/2}$ can be deduced in the same way as the ones for $c^n_{i-1/2}$.
\end{proof}

The discrete $\BV$ estimate in space then follows.
\begin{lem}\label{lemma: BVdiscr}
    Assume, $u_0,w_0\in \BV(\R)$, $|u_0(x)-w_0(x)|\leq a$ and let $\Delta x$, $\Delta t>0 $ such that the $CFL$ condition \eqref{eq: CFL} holds. Then the following inequalities hold \begin{equation}\label{eq: BVdiscru}
        \sum_{i\in \mathbb Z} |u_{i+1}^{n+1}-u_i^{n+1}| \leq \sum_{i\in \mathbb Z} | u_{i+1}^n-u_i^n|, \quad \forall n\in \N,
    \end{equation} 
    and \begin{equation}\label{eq: BVdiscrw}
        \sum_{i\in \mathbb Z} |w_{i+1}^{n+1}-w_i^{n+1}| \leq \sum_{i\in \mathbb Z} | w_{i+1}^n-w_i^n|, \quad \forall n\in \N.
    \end{equation} 
\end{lem}
\begin{proof}
    By using \eqref{eq: umia} we can write 
    \[
    \begin{split}
    u_{i+1}^{n+1}-u_i^{n+1} &= (u_{i+1}^n-u_i^n)\left(1- b_{i+1/2}^n- a_{i+1/2}^n\right) + b_{i+3/2}^n (u_{i+2}^n-u_{i+1}^n)+ a_{i-1/2}^n (u_{i}^n-u_{i-1}^n)
    \end{split}
    \]
    Then by Lemma \ref{lemma: miki}
    \[
    b_{i+3/2}^n ,\,  a_{i-1/2}^n, \,\left(1- b_{i+1/2}^n- a_{i+1/2}^n\right) \geq 0 
    \] 
    so
    \[ 
    \begin{split}
        |u_{i+1}^{n+1}-u_i^{n+1}| &\leq |u_{i+1}^n-u_i^n|\left(1- b_{i+1/2}^n- a_{i+1/2}^n\right) + \ b_{i+3/2}^n |u_{i+1}^n-u_{i+1}^n|+ a_{i-1/2}^n |u_{i}^n-u_{i-1}^n|.
     \end{split}
     \]
    Summing this last inequality over $i\in \mathbb Z$ we get \eqref{eq: BVdiscru}.\par
    Starting instead from \eqref{eq: wmia} and by following the same previous reasoning we infer \eqref{eq: BVdiscrw}.
\end{proof}
Notice that the assumption $u_0,w_0\in \BV(\R)$ is not really needed, however it ensures that the sums in the inequalities \eqref{eq: BVdiscru} and \eqref{eq: BVdiscrw} are finite. 
\par Let us now denote as $u^{\Delta}$ and $w^{\Delta}$ the approximate finite volume solution defined by \begin{equation}\label{eq: approxsol}
    u^{\Delta}(x,t):= u_{i}^n,\quad w^{\Delta}(x,t):= w_i^n \quad \text{for } x\in [x_{i-1/2},x_{i+1/2}[\,,\, t\in [t^n, t^{n+1}[\,. 
\end{equation} 
We then recover the following bounds on the $\L{\infty}$ norm and on the total variation in space-time for both $u^{\Delta}$ and $w^{\Delta}$, as a consequence of the previous lemmas. The proof of the following results are quite standard for the case of a single unknown $u$, see e.g.~\cite[Section 5.3]{FVM}, and can be directly adapted to the variable $w$ as it shares the same properties of $u$.

\begin{prop}\label{prop: linfty}
    Let $u_0, w_0\in L^\infty (\R)$, such that $u_0(x)\in [U_m, U_M]$, $w_0(x) \in [W_m, W_M]$ and $|u_0(x)-w_0(x)|\leq a$ for almost every $x\in \R$. Fix also $\Delta x, \Delta t$ such that the CFL condition \eqref{eq: CFL} holds and consider the approximate solutions $u^\Delta$ and  $w^\Delta$ defined by \eqref{eq: schema2}, \eqref{eq: approxsol}. Then $u^{\Delta}$ and $w^{\Delta}$ satisfy \begin{equation}\label{eq: uwinfty}
    u^{\Delta}(x,t) \in [U_m,U_M]\quad \text{and} \quad w^{\Delta}(x,t)\in [W_m,W_M],\end{equation}
    for a.e. $(x,t)\in \R\times [0,T].$
\end{prop}

\begin{prop}\label{prop: stimaBV}
    Let $u_0,w_0 \in \BV(\R)$ with $|u_0(x)-w_0(x)|\leq a $ for almost all $ x\in \R,$ and fix $\Delta x, \Delta t>0$ such that the CFL condition \eqref{eq: CFL} holds. Consider the approximate solutions $u^{\Delta}$ and $w^{\Delta}$  defined by \eqref{eq: schema2} and \eqref{eq: approxsol}. Then, for any $T>0,$ there exist $C_1= C_1(u_0, T)>0$ and $C_2 = C_2(w_0,T)>0$ such that \begin{equation}\label{eq: BVspaziotempo}
        |u^{\Delta}|_{\BV(\R \times [0,T[)} \leq C_1 \quad \text{and} \quad  |w^{\Delta}|_{\BV(\R \times [0,T[)} \leq C_2.
    \end{equation}
\end{prop}
In particular, the proof of Proposition~\ref{prop: stimaBV} relies on the following $\L1$ continuity in time.

\begin{lem}\label{cor: contl1}
    Given $w_0, u_0 \in \BV(\R)$, we have that \begin{equation}\label{eq: lipu}
        \int_\R |u^{\Delta}(x,s)-u^{\Delta}(x,t)|\,dx \leq  2 L\, |u_0|_{\BV(\R)} (|t-s|+ \Delta t)
    \end{equation} and
    \begin{equation}\label{eq: lipw}
        \int_\R |w^{\Delta}(x,s)-w^{\Delta}(x,t)|\,dx \leq 2 L\, |w_0|_{\BV(\R)} (|t-s|+ \Delta t)
    \end{equation} for any $s,t\in [0,T[.$
\end{lem}
\begin{proof}
    If $s,t \in [t^n, t^{n+1}[$ for  some $n\in \N$, then $u^{\Delta}(x,s)\equiv u^{\Delta}(x,t)$ and $w^{\Delta}(x,s)\equiv w^{\Delta}(x,t)$ so the two inequalities are satisfied trivially. \par 
    Let us assume that $s\in [t^{n_1},t^{n_1+1}[$, $t\in [t^{n_2},t^{n_2+1}[$ with e.g. $n_1 < n_2$. From \eqref{eq: umia} and the estimates of Lemma \ref{lemma: miki}, under the CFL condition \eqref{eq: CFL}, it can be shown that 
    \begin{equation*}
    \Delta x\sum_{i\in \mathbb{ Z}} |u_i^{n+1}-u_i^n|\leq  2 L\, \Delta t|u_0|_{\BV(\R)},
    \end{equation*}
    for any $n\in\N$ (for the details see \cite[proof of Corollary 5.1]{FVM}).
    Consequently, 
    \[\begin{split}
     \int_\R |u^{\Delta}(x,s)-u^{\Delta}(x&,t)|\,dx \leq \sum_{n=n_1}^{n_2-1} \int_\R |u^{\Delta}(x,t^{n+1})-u^{\Delta}(x,t^{n})|\,dx\\ 
     & = \Delta x \sum_{n=n_1}^{n_2-1} \sum_{i\in \mathbb{Z}} |u_i^{n+1}-u_i^n| \leq 2 L\sum_{n=n_1}^{n_2-1} \Delta t \,|u_0|_{\BV(\R)} \\ 
     &= 2 L\,(n_2-n_1) \Delta t |u_0|_{\BV(\R)} \leq 2 L\,(|s-t| +\Delta t)|u_0|_{\BV(\R)}.
     \end{split}\] 
     By the same reasoning, we get \eqref{eq: lipw}.
\end{proof}

\subsection{Discrete entropy condition}

Recalling the definition of $g$, that is \eqref{eq: gGodunov} and  using the monotonicity properties of Corollary~\ref{cor: monotonia}, we can prove the following discrete entropy condition.

\begin{prop}\label{prop: entropiadiscr}
    For each $i\in \mathbb Z, n \in \N, (k,\hat{k})\in \LM,$ it holds
    \begin{multline} \label{eq: discrEntropy}
        |u_i^{n+1}-k|-|u_i^{n}-k|+ |w_i^{n+1}-\hat{k}|-|w_i^{n}-\hat{k}|+\\ \frac{\Delta t}{\Delta x} \left(g(u_{i}^n\top k,u_{i+1}^n\top k)-g(u_{i}^n \perp k, u_{i+1}^n\perp k) \right. \\ 
        \left.- g(u_{i-1}^n\top k,u_{i}^n\top k)+g(u_{i-1}^n \perp k, u_{i}^n\perp k) \right) \leq 0,
    \end{multline}
    where $a\top b:= \max\{a,b\}$ and $a\perp b :=\min \{a,b\}.$
\end{prop}

\begin{proof}
    It is easy to check that, since $|u_i^n-w_i^n| \leq a$ and $|k-\hat{k}|\leq a$, then $|u_i^n \top k- w_i^n\top \hat{k}|\leq a$ and $|u_i^n \perp k- w_i^n\perp \hat{k}|\leq a$. Then by Corollary \ref{cor: monotonia} we have that 
    \begin{equation}\label{eq: disU}
     u_i^{n+1} \leq H_1(u_{i-1}^n\top k, u_i^n \top k, u_{i+1}^n \top k,w_i^n\top \hat{k}), \quad w_i^{n+1} \leq H_2(u_{i-1}^n\top k, u_i^n \top k, u_{i+1}^n \top k,w_i^n\top \hat{k}),
    \end{equation}
    and
    \begin{equation}\label{eq: disK}
      k \leq H_1(u_{i-1}^n\top k, u_i^n \top k, u_{i+1}^n \top k,w_i^n\top \hat{k}),\quad \hat{k} \leq H_2(u_{i-1}^n\top k, u_i^n \top k, u_{i+1}^n \top k,w_i^n\top \hat{k}),
    \end{equation}
    as $ k = H_1(k,k,k,\hat{k})$ and $ \hat{k} = H_2(k, k, k, \hat{k})$. Hence, adding $u_i^{n+1} \top k$ and $w_i^{n+1} \top \hat{k}$, thanks to \eqref{eq: H1piuH2}, \eqref{eq: disU} and \eqref{eq: disK}, we get the following estimate
    \begin{equation}\label{eq: disTop}
        \begin{split}
            u_i^{n+1} \top k + w_i^{n+1} \top \hat{k} \leq u_i^n \top k+w_i^n\top \hat{k} -\frac{\Delta t}{\Delta x} \left(g(u_i^n \top k,u_{i+1}^n \top k)-g(u_{i-1}^n \top k,u_i^n \top k)\right).
        \end{split}
    \end{equation}
    Similarly,
    \begin{equation}\label{eq: disPerp}
        u_i^{n+1} \perp k + w_i^{n+1} \perp \hat{k} \geq u_i^n \perp k+w_i^n\perp \hat{k} -\frac{\Delta t}{\Delta x} \left(g(u_i^n \perp k,u_{i+1}^n \perp k)-g(u_{i-1}^n \perp k,u_i^n \perp k)\right),
    \end{equation}
    so subtracting \eqref{eq: disPerp} to \eqref{eq: disTop} we get \eqref{eq: discrEntropy}.
\end{proof}

\subsection{Weak Hysteresis Relationship}

Our goal is then to show that the approximate solutions constructed via the numerical scheme presented in Section~\ref{sec:godunov} satisfy the weak hysteresis relationship, so by starting from \eqref{eq: hismis} we would like to prove \eqref{eq: genweakhis}. As shock waves solution lack in regularity, we cannot do this directly. Instead, we will exploit the following property. 

\begin{lem}\label{lemma: weakhisshock}
    Consider the scheme \eqref{eq: schema2}, with $\Delta x,\Delta t>0$ such that the CFL condition \eqref{eq: CFL} holds. Moreover, define 
    \[
    G(u):= \int_0^u \xi f'(\xi)\,d\xi = u f(u) - \int_0^u f(\xi)\,d\xi
    \] 
    and denote by $\tilde{u}^n$ and $\tilde{w}^n$ the exact solutions defined on $\R \times [t^n,t^{n+1}[$ such that $\tilde{u}^n(x,t^n)=u_i^n$ and $\tilde{w}^n(x,t^{n})=w_i^n$ for $x\in K_i$. Then for $t\in [t^n,t^{n+1}[$, if there is an entropic shock in the half cell $K_i^{n,l}$, it holds 
    \begin{equation}\label{eq: hisshockK1}
    \begin{split}
       - \Delta &t\left( G(u_i^n)-G(\tilde{u}(x_{i-1/2}+,t)) \right) - \frac{1}{2} \int_{x_{i-1/2}}^{x_{i}} \left(\tilde{u}^n(x, t^{n+1})^2- \tilde{u}^n(x, t^{n})^2\right)\,dx
       \\ & -\ \frac{1}{2} \int_{x_{i-1/2}}^{x_i} \left(\tilde{w}^n(x, t^{n+1})^2- \tilde{w}^n(x, t^{n})^2\right) \, dx\geq   a \left|\partial_t \tilde{w}^n\right|\left( K_i^{n,l}\right);
       \end{split}
    \end{equation} 
    instead, if an entropic shock is present in $K_i^{n,r}$, it holds
\begin{equation}\label{eq: hisshockK2}
    \begin{split}
       - \Delta &t\left( G(\tilde{u}(x_{i+1/2}-,t)) - G(u_i^n)\right) - \frac{1}{2}\int_{x_{i}}^{x_{i+1/2}} \left(\tilde{u}^n(x, t^{n+1})^2- \tilde{u}^n(x, t^{n})^2\right)\,dx\\
       &-\frac{1}2{} \int_{x_{i}}^{x_{i+1/2}}\left(\tilde{w}^n(x, t^{n+1})^2- \tilde{w}^n(x, t^{n})^2\right) \, dx \geq   a \left|\partial_t \tilde{w}^n\right|\left( K_i^{n,r}\right).
       \end{split}
    \end{equation}
\end{lem}
\begin{proof}
    We detail the proof for \eqref{eq: hisshockK1},  the other case being similar. Three configuration can occur: a single shock in $\tilde{u}^n$ with $\tilde{w}^n$ constant; two shocks, one both in $\tilde{u}^n$ and $\tilde{w}^n$ and one only in $\tilde{u}^n$; a ``fast shock" in both $\tilde{u}^n$ and $\tilde{w}^n$; see Case \ref{caso: ul>ur} and in particular Subcase \ref{sottocaso: ful>fur}  in Section \ref{S2}.
\begin{itemize}
    \item If there is only a shock in $u$, then the right hand side of \eqref{eq: hisshockK1} is $0$.  Since the shock joins  the states $u_{i-1}^n$ and $u_i^n$ with speed given by the \RH condition, see Figure \ref{fig: shockcaso1}, the left hand side reads
    \begin{align} - \Delta &t\left( G(u_i^n)-G(\tilde{u}(x_{i-1/2}+,t)) \right) - \frac{1}{2} \int_{x_{i-1/2}}^{x_{i}} \left(\tilde{u}^n(x, t^{n+1})^2- \tilde{u}^n(x, t^{n})^2\right)\,dx
       \nonumber\\ & -\ \frac{1}{2} \int_{x_{i-1/2}}^{x_i} \left(\tilde{w}^n(x, t^{n+1})^2- \tilde{w}^n(x, t^{n})^2\right) \, dx\nonumber\\
      =\  &-\Delta t (G(u_i^n)-G(u_{i-1}^n))-\frac{1}{2}\frac{f(u_{i-1}^n)-f(u_i^n)}{u_{i-1}^n-u_i^n} \Delta t \, ((u_{i-1}^n)^2-(u_i^n)^2) \nonumber\\
    \label{eq: caso1eq1}
    =\ &-\Delta t \left[(G(u_i^n)-G(u_{i-1}^n))+\frac{1}{2}(f(u_{i-1}^n)-f(u_i^n)) \, (u_{i-1}^n + u_i^n)\right].
    \end{align}
    
    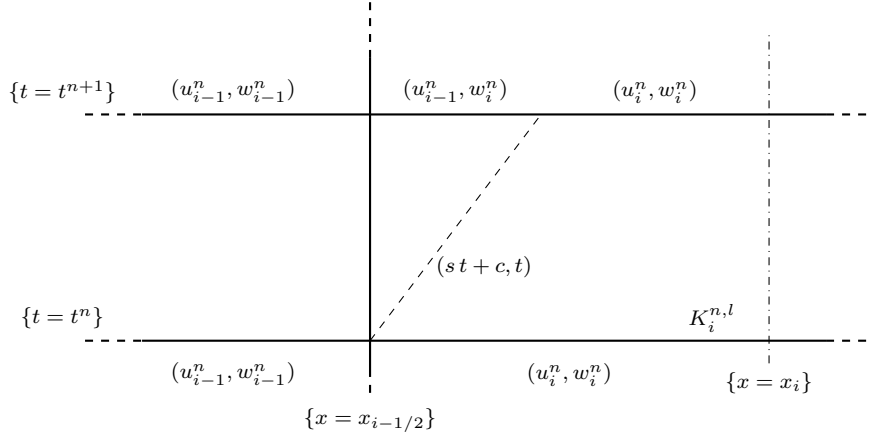
\begin{figure}
    \centering
    \begin{tikzpicture}[scale = 1.5]

      \draw[-,thick] (0,0) -- (4,0);
      \draw[dashed,thick] (4,0) -- (4.5,0);
      
      \draw[-,thick] (0,2) -- (4,2);
      \draw[dashed,thick] (4,2) -- (4.5,2);

      \draw[-,thick] (0,0)--(0,2);

      \draw[-,thick] (-2,0) -- (0,0);
      \draw[dashed,thick] (-2.5,0) -- (-2,0);

      \draw[-,thick] (-2,2) -- (0,2);
      \draw[dashed,thick] (-2.5,2) -- (-2,2);

     
      \node at (-1.2,-0.3) {\scriptsize{$(u_{i-1}^n, w_{i-1}^n)$}};
      \node at (1.75,-0.3) {\scriptsize{$(u_{i}^n, w_{i}^n)$}};

      \node at (-1.2,2.2) {\scriptsize{$(u_{i-1}^n, w_{i-1}^n)$}};
      \node at (0.75,2.2) {\scriptsize{$(u_{i-1}^n, w_{i}^n)$}};
      \node at (2.5,2.2) {\scriptsize{$(u_{i}^n, w_{i}^n)$}};

      \draw[-,thick] (0,2) -- (0,2.5);
      \draw[dashed,thick] (0,2.5) -- (0,3);

      \draw[-,thick] (0,0) -- (0,-0.25);
      \draw[dashed,thick] (0,-0.25) -- (0,-0.5);
      
      \draw[dash dot] (3.5,-0.2)--(3.5,2.7);

      \draw[dashed] (0,0)--(1.5,2);
      \node at (1,0.65) {\scriptsize{$(s\,t+c,t)$}};
      
      \node at (-2.7,0.2) {\scriptsize{$\{t=t^n\}$}};
      \node at (-2.7,2.2) {\scriptsize{$\{t=t^{n+1}\}$}};

      \node at (3.5,-0.37) {\scriptsize{$\{x=x_{i}\}$}};
      \node at (0,-0.7) {\scriptsize{$\{x=x_{i-1/2}\}$}};
      
      \node at (3,0.2) {\scriptsize{$K_{i}^{n,l}$}};
    
    \end{tikzpicture}
    
    \caption{The exact solution in the half-cell $K_i^{n,l}$ in the case only a shock in $\tilde{u}^n$ is present (here represented by the dashed line) and it has speed $s$ given by the  \RH condition.  }
    \label{fig: shockcaso1}
    \end{figure}

    Replacing $G$ by its definition and doing some computations we get 
    \begin{align*}
    [\eqref{eq: caso1eq1}] &=
     -\Delta t \left[ u_i^n f(u_i^n)-u_{i-1}^n f(u_{i-1}^n) + \int_{u_{i}^n}^{u_{i-1}^n} f(\xi)\,d\xi+\frac{1}{2}(f(u_{i-1}^n)-f(u_i^n)) \, (u_{i-1}^n + u_i^n)\right]\\
     &= \Delta t \left[\frac{1}{2}(f(u_{i-1}^n)+f(u_i^n)) \, (u_{i-1}^n - u_i^n) - \int_{u_{i}^n}^{u_{i-1}^n} f(\xi)\,d\xi\right]
     \geq 0
     \end{align*}
    since $u_{i-1}^n> u_{i}^n$ (entropic shock) and $f$ is convex. 
    
\item    If there are two shocks, one in $\tilde{u}^n$ only, connecting $u_i^n$ to $w_i^n+a$, and the other in both $\tilde{u}^n$ and $\tilde{w}^n$, which connects respectively $w_i^n+a$ to $u_i^n$ and $w_i^n$ to $u_{i-1}^n-a$, both with slope given by the \RH condition (see Figure \ref{fig: shockcaso2}), then
    \begin{equation}\label{eq: caso2eq1} 
    \begin{split}
     \int_{x_{i-1/2}}^{x_{i}} &\left(\tilde{u}^n(x, t^{n+1})^2- \tilde{u}^n(x, t^{n})^2\right)\,dx\\
    &= \Delta x_1 \left((u_{i-1}^n)^2-(u_i^n)^2\right) + (\Delta x_2-\Delta x_1) \left((w_i^n+a)^2-(u_i^n)^2\right) \\ 
    &=  \Delta x_1 \left((u_{i-1}^n)^2-(w_i^n+a)^2\right)+  \Delta x_2 \left((w_i^n+a)^2-(u_i^n)^2\right)\\
    &= \frac{1}{2} \Delta t \, \frac{f(u_{i-1}^n)-f(w_i^n+a)}{u_{i-1}^n-(w_i^n+a)} \left((u_{i-1}^n)^2-(w_i^n+a)^2\right) \\
    &~~~~~+\Delta t\, \frac{f(w_i^n+a)-f(u_i^n)}{(w_{i}^n+a)-u_i^n}  \left((w_i^n+a)^2-(u_i^n)^2\right)\\
    & =\frac{1}{2}\Delta t \left(f(u_{i-1}^n)-f(w_i^n+a)\right) (u_{i-1}^n+w_i^n+a) 
    \\&~~~~~+\Delta t \left(f(w_i^n+a)-f(u_i^n)\right) (w_i^n+a+u_i^n)
    \end{split}
    \end{equation}
    and 
    \begin{align}
     \int_{x_{i-1/2}}^{x_{i}} &\left(\tilde{w}^n(x, t^{n+1})^2- \tilde{w}^n(x, t^{n})^2\right)\,dx \notag \\
     &= \frac{1}{2} \Delta t \, \frac{f(u_{i-1}^n)-f(w_i^n+a)}{(u_{i-1}^n)-(w_i^n+a)}\left((u_{i-1}^n-a)^2-(w_i^n)^2\right) \notag\\
     &=  \frac{1}{2} \Delta t  \left(f(u_{i-1}^n)-f(w_i^n+a)\right)(u_{i-1}^n-a+w_i^n) \notag \\
     &=  \frac{1}{2}\Delta t \left(f(u_{i-1}^n)-f(w_i^n+a)\right)  (u_{i-1}^n+a+w_i^n) -a\Delta t\left(f(u_{i-1}^n)-f(w_i^n+a)\right) \label{eq: caso2eq2} 
    \end{align}
    and
    \begin{equation}\label{eq: caso2eq3}
    \begin{split}
    a \left|\partial_t \tilde{w}^n\right|\left( K_i^{n,l}\right) &= a\Delta t \,\frac{1}{2}\frac{f(u_{i-1}^n)-f(w_i^n+a)}{u_{i-1}^n-(w_i^n+a)}  \left|(u_{i-1}^n-a)-w_i^n\right|\\
    & =\, \frac{1}{2} a\Delta t \left(f(u_{i-1}^n)-f(w_i^n+a)\right) .
    \end{split}
    \end{equation}

    \begin{figure}
    \centering
    \begin{tikzpicture}[scale = 2]

      \draw[-,thick] (0,0) -- (4,0);
      \draw[dashed,thick] (4,0) -- (4.5,0);
      
      \draw[-,thick] (0,2) -- (4,2);
      \draw[dashed,thick] (4,2) -- (4.5,2);

      \draw[-,thick] (0,0)--(0,2);

      \draw[-,thick] (-2,0) -- (0,0);
      \draw[dashed,thick] (-2.5,0) -- (-2,0);

      \draw[-,thick] (-2,2) -- (0,2);
      \draw[dashed,thick] (-2.5,2) -- (-2,2);

     
      \node at (-1.2,-0.3) {\scriptsize{$(u_{i-1}^n, w_{i-1}^n)$}};
      \node at (1.75,-0.3) {\scriptsize{$(u_{i}^n, w_{i}^n)$}};

      \node at (-1.2,2.2) {\scriptsize{$(u_{i-1}^n, w_{i-1}^n)$}};
      \node at (0.6,2.2) {\scriptsize{$(u_{i-1}^n, u_{i-1}^n-a)$}};
      \node at (2,2.2) {\scriptsize{$(w_{i}^n+a, w_{i}^n)$}};
      \node at (3.1,2.2) {\scriptsize{$(u_{i}^n, w_{i}^n)$}};

      \draw[-,thick] (0,2) -- (0,2.5);
      \draw[dashed,thick] (0,2.5) -- (0,3);

      \draw[-,thick] (0,0) -- (0,-0.25);
      \draw[dashed,thick] (0,-0.25) -- (0,-0.5);
      
      \draw[dash dot] (3.5,-0.2)--(3.5,2.7);

      \draw[dashed] (0,0)--(1.4,2);
      
      \draw[dashed] (0,0)--(2.7,2);


      \draw[<->, thin] (0,1.9)--(1.4,1.9) node[below, midway] {\scriptsize{$\Delta x_1$}};
      \draw[<->, thin] (0,1.6)--(2.7,1.6) node[below, midway] {\scriptsize{$\Delta x_2$}};

      \draw[thin] (1.4,2)--(1.4,1.9);
      \draw[thin] (2.7,2)--(2.7,1.6);

      \node at (-2.7,0.2) {\scriptsize{$\{t=t^n\}$}};
      \node at (-2.7,2.2) {\scriptsize{$\{t=t^{n+1}\}$}};

      \node at (3.5,-0.37) {\scriptsize{$\{x=x_{i}\}$}};
      \node at (0,-0.7) {\scriptsize{$\{x=x_{i-1/2}\}$}};
      
      \node at (3,0.2) {\scriptsize{$K_{i}^{n,l}$}};
    
    \end{tikzpicture}
    
    \caption{The exact solution in the semi-cell $K_i^{n,l}$ in the case when we have two shocks.}
    \label{fig: shockcaso2}
    \end{figure}
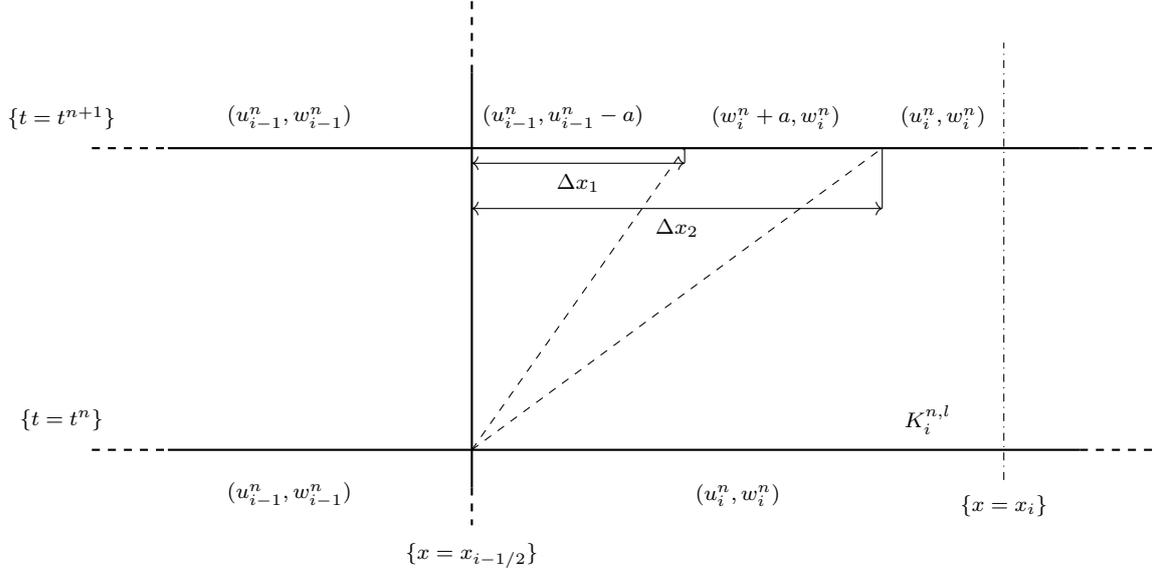

    Replacing \eqref{eq: caso2eq1}, \eqref{eq: caso2eq2} and \eqref{eq: caso2eq3} in \eqref{eq: hisshockK1}, it becomes 
    \[
    \begin{split}
    -\Delta t \bigg[&(G(u_i^n)-G(u_{i-1}^n))+\frac{1}{2}\left(f(u_{i-1}^n)-f(w_i^n+a)\right)  (u_{i-1}^n + w_i^n+a) \\  & +\frac{1}{2}\left(f(w_i^n+a)-f(u_i^n)\right) (w_i^n+a+u_i^n) \bigg]\geq 0,
    \end{split}
    \] 
    which can be rewritten as 
    \[
    \begin{split}
    -\Delta t \bigg[&\int_{w_i^n+a}^{u^n_{i-1}} f(\xi)\,d\xi-\frac{1}{2}\left(f(u_{i-1}^n)+f(w_i^n+a)\right) \left(u_{i-1}^n - (w_i^n+a)\right)\\ 
    &+ \int_{u_{i}^n}^{w_i^n+a} f(\xi)\,d\xi-\frac{1}{2}\left(f(w_i^n+a)+f(u_i^n)\right) (w_i^n+a - u_i^n)\bigg] \geq 0.
     \end{split}
    \]
    As before, since $u_{i-1}^n\geq w_i^n+a \geq u_i^n$ and $f$ is convex, this last inequality is satisfied, inferring the desired conclusion.

\item    If a ``fast shock" is present, connecting $(u_i^n,w_i^n)$ to $(u_{i-1}^n,u_{i-1}^n-a)$ with slope given by \RH condition, we get
    \begin{equation}\label{eq: fsu}
    \begin{split}
      \int_{x_{i-1/2}}^{x_{i}} &\left(\tilde{u}^n(x, t^{n+1})^2- \tilde{u}^n(x, t^{n})^2\right) \, dx\\
    & =\Delta t \,\frac{f(u_{i-1}^n)-f(u_i^n)}{u_{i-1}^n-u_i^n+(u_{i-1}^n-a)-w_i^{n}}  \left((u_{i-1}^n)^2-(u_i^n)^2\right)\\
     & = \Delta t \, \frac{f(u_{i-1}^n)-f(u_i^n)}{u_{i-1}^n-u_i^n+(u_{i-1}^n-a)-w_i^{n}}  \left((u_{i-1}^n)^2-(w_i^n+a)^2+(w_i^n+a)^2-(u_i^n)^2\right)
     \\
     & = \Delta t \left(f(u_{i-1}^n)-f(u_i^n)\right) (u_{i-1}^n+w_i^n+a) \frac{u_{i-1}^n-(w_i^n+a)}{u_{i-1}^n-u_i^n+(u_{i-1}^n-a)-w_i^{n}} 
     \\
     &~~~~~+\Delta t \left(f(u_{i-1}^n)-f(u_i^n)\right) (w_i^n+a+u_i^n)\frac{(w_i^n+a)-u_i^n}{u_{i-1}^n-u_i^n+(u_{i-1}^n-a)-w_i^{n}} 
    \\
     & = \Delta t \left(f(u_{i-1}^n)-f(u_i^n)\right) (u_{i-1}^n+w_i^n+a) \frac{I_l}{I_r+2I_l} 
     \\& ~~~~~+ \Delta t \left(f(u_{i-1}^n)-f(u_i^n)\right) (w_i^n+a+u_i^n) \frac{I_r}{I_r+2I_l} 
    \end{split}
    \end{equation} 
    and 
    \begin{equation}\label{eq: fsw}
    \begin{split}
    \int_{x_{i-1/2}}^{x_{i}} &\left(\tilde{w}^n(x, t^{n+1})^2- \tilde{w}^n(x, t^{n})^2\right)\,dx\\
     & =\Delta t \,\frac{f(u_{i-1}^n)-f(u_i^n)}{u_{i-1}^n-u_i^n+(u_{i-1}^n-a)-w_i^{n}}  \left((u_{i-1}^n-a)^2-(w_i^n)^2\right)\\
    & = \Delta t \, \left(f(u_{i-1}^n)-f(u_i^n)\right) (u_{i-1}^n-a+w_i^n) \frac{(u_{i-1}^n-a)-w_i^n}{u_{i-1}^n-u_i^n+(u_{i-1}^n-a)-w_i^{n}}\\
     &= \Delta t \left(f(u_{i-1}^n)-f(u_i^n)\right) (u_{i-1}^n+w_i^n+a) \frac{I_l}{I_r+2I_l} 
     \\ &~~~~~-a \Delta t \left(f(u_{i-1}^n)-f(u_i^n)\right)  \frac{(u_{i-1}^n-a)-w_i^n}{u_{i-1}^n-u_i^n+(u_{i-1}^n-a)-w_i^{n}}
    \end{split}
    \end{equation} 
    and
    \begin{equation}\label{eq: fswar}
    \begin{split}
    a \left|\partial_t \tilde{w}^n\right|\left( K_i^{n,l}\right) &= a\, \Delta t \,\frac{f(u_{i-1}^n)-f(u_i^n)}{u_{i-1}^n-u_i^n+(u_{i-1}^n-a)-w_i^{n}} \left|u_{i-1}^n-a-w_i^n\right|\\
    & =a\, \Delta t \left(f(u_{i-1}^n)-f(u_i^n)\right)  \frac{(u_{i-1}^n-a)-w_i^n}{u_{i-1}^n-u_i^n+(u_{i-1}^n-a)-w_i^{n}}.
    \end{split}
    \end{equation}
    Above, we use the same notation $I_r=(w_r+a)-u_r$, $I_l = u_l-(w_r+a)$ introduced in Subcase \ref{sottocaso: ful>fur}. We also define 
    \[
    c_1 = \frac{I_r}{I_r+2I_l} \quad c_2= \frac{2I_l}{I_r+2I_l}
    \] 
    and notice that $c_1+c_2=1.$\par
    By\eqref{eq: fsu}, \eqref{eq: fsw} and \eqref{eq: fswar}, we get
    \begin{align}
       &- \frac{1}{2} \int_{x_{i-1/2}}^{x_{i}} \left(\tilde{u}^n(x, t^{n+1})^2- \tilde{u}^n(x, t^{n})^2\right)\,dx \nonumber\\
        &-\ \frac{1}{2} \int_{x_{i-1/2}}^{x_i} \left(\tilde{w}^n(x, t^{n+1})^2- \tilde{w}^n(x, t^{n})^2\right) \, dx -  a \left|\partial_t \tilde{w}^n\right|\left( K_i^{n,l}\right) \nonumber\\ 
   =\,&- \frac{1}{2} \Delta t \left(f(u_{i-1}^n)-f(u_i^n)\right) (u_{i-1}^n+w_i^n+a) c_2 \nonumber\\&- \frac{1}{2} \Delta t \left(f(u_{i-1}^n)-f(u_i^n)\right) (w_i^n+a+u_{i}^n) c_1 \nonumber\\
   =\,& - \frac{1}{2} \Delta t \left(f(u_{i-1}^n)-f(w_i^n+a)\right) (u_{i-1}^n+w_i^n+a) c_2 \nonumber\\&- \frac{1}{2} \Delta t \left(f(w_i^n+a)-f(u_i^n)\right) (u_{i-1}^n+w_i^n+a) c_2
         \nonumber\\
         &- \frac{1}{2} \Delta t \left(f(u_{i-1}^n)-f(w_i^n+a)\right) (w_i^n+a+u_i^n) c_1 \nonumber\\&- \frac{1}{2} \Delta t \left(f(w_i^n+a)-f(u_i^n)\right) (w_i^n+a+u_i^n) c_1\,.
         \label{eq: c1c2}
    \end{align}
Besides, we can develop
    \begin{equation}\label{eq: GG} 
    \begin{split}   
     - &\Delta t \left[G(u_i^n)-G(u_{i-1}^n)\right] \\&= - \Delta t \left[f(u_{i}^n)u_i^n-f(u_{i-1}^n)u_{i-1}^n+\int_{u_i^n}^{u_{i-1}^n} f(\xi)\, d\xi\right] 
     \\ &= - \Delta t \left[c_1 f(u_{i}^n)u_i^n+ c_2f(u_{i}^n)u_i^n -c_1f(u_{i-1}^n)u_{i-1}^n-c_2f(u_{i-1}^n)u_{i-1}^n+\int_{u_i^n}^{u_{i-1}^n} f(\xi)\, d\xi\right]
    \end{split}
    \end{equation}
    Now, summing the terms involving $c_2$ in the expressions above, we get
    \[
    \begin{split}
        &- \frac{c_2}{2} \left(f(u_{i-1}^n)-f(w_i^n+a)\right) (u_{i-1}^n+w_i^n+a) \\&- \frac{c_2}{2} \left(f(w_i^n+a)-f(u_i^n)\right) (u_{i-1}^n+w_i^n+a) 
        \\&- c_2f(u_{i}^n)u_i^n +c_2f(u_{i-1}^n)u_{i-1}^n
        \\ =\ &  \frac{c_2}{2} \left(f(u_{i-1}^n)+f(w_i^n+a)\right)\left(u_{i-1}^n - (w_i^n+a)\right) \\&+ \frac{c_2}{2} \left(f(u_{i}^n)+f(w_i^n+a)\right)\left( (w_i^n+a)-u_{i}^n\right) 
        \\&- \frac{c_2}{2}\left(f(w_i^n+a)-f(u_i^n)\right) (u_{i-1}^n-u_i^n)
    \end{split}
    \]
    and,  summing the terms involving $c_1$,
    \[
    \begin{split}
        &- \frac{c_1}{2} \left(f(u_{i-1}^n)-f(w_i^n+a)\right) (w_i^n+a+u_i^n) \\&- \frac{c_1}{2} \left(f(w_i^n+a)-f(u_i^n)\right) (w_i^n+a+u_i^n)  \\& - c_1f(u_{i}^n)u_i^n +c_1f(u_{i-1}^n)u_{i-1}^n
        \\ =\ & \frac{c_1}{2} \left(f(u_{i-1}^n)+f(w_i^n+a)\right)\left(u_{i-1}^n - (w_i^n+a)\right) \\&+ \frac{c_1}{2} \left(f(u_{i}^n)+f(w_i^n+a)\right)\left( (w_i^n+a)-u_{i}^n\right) 
        \\ & + \frac{c_1}{2}\left(f(u_{i-1}^n)-f(w_i^n+a)\right) (u_{i-1}^n-u_i^n)\,.
    \end{split}
    \]
    Therefore, summing \eqref{eq: c1c2} with \eqref{eq: GG} and recalling that $c_1+c_2=1,$ we obtain $\mathcal{I}_1+\mathcal{I}_2+\mathcal{I}_3$ with
    \[\mathcal{I}_1 :=
   \Delta t \left[ \frac{1}{2} \left(f(u_{i-1}^n)+f(w_i^n+a)\right)\left(u_{i-1}^n - (w_i^n+a)\right) - \int_{w_i^n+a}^{u_{i-1}^n }f(\xi)\,d\xi\right] \geq 0
    \]
    and 
    \[ \mathcal{I}_2 :=
     \Delta t\left[\frac{1}{2} \left(f(u_{i}^n)+f(w_i^n+a)\right)\left( (w_i^n+a)-u_{i}^n\right)-\int^{w_i^n+a}_{u_{i}^n }f(\xi)\,d\xi\right] \geq 0
    \]
   by convexity of $f$. The third term is
    \begin{align*}
    \mathcal{I}_3 := \ &
        - \frac{c_2}{2}\left(f(w_i^n+a)-f(u_i^n)\right) (u_{i-1}^n-u_i^n)+ \frac{c_1}{2}\left(f(u_{i-1}^n)-f(w_i^n+a)\right) (u_{i-1}^n-u_i^n)\\
       =\ & \frac{I_rI_l}{I_r+2I_l} (u_{i-1}^n-u_i^n)(\mu_l-\mu_r)\geq 0
    \end{align*}
    since $u_{i-1}^n>u_i^n$ and $\mu_l>\mu_r$ as required by the entropy condition. 
    \end{itemize}
    This concludes the proof of \eqref{eq: hisshockK1}.
\end{proof}

The previous Lemma allows to prove the following result.
\begin{prop}\label{prop: whisdiscr}
    Let $u_0,w_0\in \L1(\R)\cap \BV(\R)$, such that $|u_0(x)-w_0(x)| \leq a$, and $\Delta x, \Delta t>0$ satisfying the CFL condition \eqref{eq: CFL}. Let $(u^\Delta,w^\Delta )$ be the approximate solutions defined by \eqref{eq: schema2}, \eqref{eq: approxsol}, then this couple satisfies the weak hysteresis relationship \eqref{eq: genweakhis} for almost every $t\in [0,T[$.
\end{prop}

\begin{proof}
    Consider a cell $K_i^n$ and the exact solutions $\tilde{u}^n(x,t), \tilde{w}^n(x,t)$ defined on $K_i^n$ with data at time $t^n$ equal to $u_i^n,w_i^n$. We know that, for each $x\in\, ]x_{i-1/2},x_{i+1/2}[$, the relationship $\tilde{w}^n(x,t)= \F[\tilde{u}(x,\cdot),w_i^n](t)$ holds, so, by integrating it on $K_i$, we get, by Proposition \ref{prop: whis}
    \begin{equation}\label{eq: hiscella}
        \int_{x_{i-1/2}}^{x_{i+1/2}}\int_{t^n}^{t^{n+1}} \left(\tilde{u}^n(x,t+)- \tilde{w}^n(x,t+)\right) \, d \left(\partial_t \tilde{w}^n(x,\cdot)\right)\, dx \geq a \left|\partial_t \tilde{w}^n\right|( K_i^n).
    \end{equation}
    As usual, we split $K_i^n$ into the two subcells $K_i^{n,l}$ and $K_i^{n,r}$. If in $K_i^{n,l}$ there are no jump discontinuities, then $\tilde{u}^n$ and $\tilde{w}^n$ are Lipschitz continuous and they satisfy the PDE strongly almost everywhere in $K_i^{n,l}$. In particular, $\partial_t \tilde{w}^n(x,\cdot)$ is not only a measure but a function for a. e. $x$, so using the strong formulation of the PDE, the left hand side of \eqref{eq: hiscella} restricted to $K_i^{n,l}$ writes as follow
    \begin{equation}\label{eq: rarek1}
    \begin{split}
       &\int_{x_{i-1/2}}^{x_{i}}\int_{t^n}^{t^{n+1}} \left(\tilde{u}^n(x,t+)- \tilde{w}^n(x,t+)\right) \, d \left(\partial_t \tilde{w}^n(x,\cdot)\right)\, dx \\
       &=  \int_{x_{i-1/2}}^{x_{i}}\int_{t^n}^{t^{n+1}} \left(\tilde{u}^n(x,t)- \tilde{w}^n(x,t)\right)  \, \partial_t \tilde{w}^n \, dt\, dx\\
        &= \int_{x_{i-1/2}}^{x_{i}}
        \int_{t^n}^{t^{n+1}} \left( -\tilde{u}^n \, \partial_t \tilde{u}^n -\tilde{u}^n \, \partial_x f(\tilde{u}^n) \right)\, dt \, dx - \frac{1}{2}\int_{x_{i-1/2}}^{x_{i}} \left(\tilde{w}^n(x, t^{n+1}-)^2- \tilde{w}^n(x, t^{n}+)^2 \right)\, dx \\ 
        &= - \Delta t\left( G(u_i^n)-G(\tilde{u}^n(x_{i-1/2}+,t)) \right) - \frac{1}{2}\int_{x_{i-1/2}}^{x_{i}} \left(\tilde{u}^n(x, t^{n+1}-)^2- \tilde{u}^n(x, t^{n}+)^2)\right)\,dx\,
        \\ & \hspace{1.5cm}- \frac{1}{2}\int_{x_{i-1/2}}^{x_{i}}\left((\tilde{w}^n(x, t^{n+1}-)^2- \tilde{w}^n(x, t^{n}+)^2\right) \, dx,
    \end{split}
    \end{equation}
    where we recall $G(u):=\int_0^u \xi f'(\xi)\,d\xi= u f(u)-\int_0^u f(\xi)\,d\xi$. \\
    Similarly, if there are no shocks in $K_i^{n,r}$, 
    \begin{equation}\label{eq: rarek2}
    \begin{split}
       &\int_{x_{i}}^{x_{i+1/2}}\int_{t^n}^{t^{n+1}} \left(\tilde{u}^n(x,t+)- \tilde{w}^n(x,t+)\right) \, d \left(\partial_t \tilde{w}^n(x,\cdot)\right)\, dx  \\
       &=- \Delta t\left( G(\tilde{u}^n(x_{i+1/2}-,t)) - G(u_i^n)\right) -    \frac{1}{2} \int_{x_{i}}^{x_{i+1/2}} \left(\tilde{u}^n(x, t^{n+1}-)^2- \tilde{u}^n(x, t^{n}+)^2\right)\,dx
       \\ & \hspace{1.5 cm} -\frac{1}{2} \int_{x_{i}}^{x_{i+1/2}}\left(\tilde{w}^n(x, t^{n+1}-)^2- \tilde{w}^n(x, t^{n}+)^2\right) \, dx.
    \end{split}
    \end{equation}
    Combining \eqref{eq: hiscella}, \eqref{eq: rarek1} and \eqref{eq: rarek2} to handle rarefactions, together with the results of Lemma \ref{lemma: weakhisshock} to account for shocks, we generally obtain in $K_i^n$ the following inequality
    \begin{equation}
        \begin{split}
         - \Delta & t\left( G(\tilde{u}^n(x_{i+1/2}-,t))-G(\tilde{u}^n(x_{i-1/2}+,t)) \right) -\frac{1}{2}\int_{x_{i-1/2}}^{x_{i+1/2}} \left(\tilde{u}^n(x, t^{n+1}-)^2- \tilde{u}^n(x, t^{n}+)^2\right)\, dx
         \\ & -\frac{1}{2} \int_{x_{i-1/2}}^{x_{i+1/2}} \left(\tilde{w}^n(x, t^{n+1}-)^2- \tilde{w}^n(x, t^{n}+)^2\right) \, dx
         \geq a \left| \partial_t \tilde{w}^n\right|(K_i^n)
         \end{split}
    \end{equation}
    As usual, recalling that $\tilde{u}^n(x, t^{n})\equiv u_i^n$, $\tilde{w}^n(x, t^{n})\equiv w_i^n$ and by Jensen's inequality, we can infer that 
    \begin{equation}\label{eq: discr1}
        \begin{split}
         - \Delta & t\left( G(\tilde{u}^n(x_{i+1/2}-,t))-G(\tilde{u}^n(x_{i-1/2}+,t)) \right)
         \\& -\frac{1}{2}\int_{x_{i-1/2}}^{x_{i+1/2}} \left((u_i^{n+1})^2- (u_i^n)^2+ (w_i^{n+1})^2- (w_i^n)^2 \right)\, dx
         \geq a \left| \partial_t \tilde{w}^n\right|(K_i^n).
         \end{split}
    \end{equation}
    Moreover, notice that $\tilde{w}^n_i(x,\cdot)$ is monotone in $t$ for $x \in K_i$ and it may only change monotonicity across $\{x_i\} \times [t^n, t^{n+1}[$, which has measure $0$ with respect to ${\partial_t \tilde{w^n}}$. So, by using Jensen's inequality once again, it holds that 
    \[
    \begin{split}
         a \left| \partial_t \tilde{w}^n\right|(K_i) &= a \int_{x_{i-1/2}}^{x_{i+1/2}} |\tilde{w}^n(x,t^{n+1})-\tilde{w}^n(x,t^{n})| \,dx
         \\ & = a \int_{x_{i-1/2}}^{x_{i+1/2}} |\tilde{w}^n(x,t^{n+1})-w_i^n| \,dx
         \\ & \geq  a \int_{x_{i-1/2}}^{x_{i+1/2}} |w_{i}^{n+1}-w_i^n| \,dx.
    \end{split}
    \]
    Using this last inequality and summing \eqref{eq: discr1} over $i\in \mathbb{Z}$ and recalling \eqref{eq: approxsol}, we finally get
    \begin{equation}\label{eq: discr2}
        \begin{split}
       - \Delta t & \sum_{i\in \mathbb{Z}}
         \left( G(\tilde{u}^n(x_{i+1/2}-,t))-G(\tilde{u}^n(x_{i-1/2}+,t)) \right) -\frac{1}{2}\int_{\R} \left(u^\Delta (x,t^{n+1})^2- u^\Delta (x,t^{n})^2\right)\,dx
         \\ &-\frac{1}{2} \int_{\R}\left(w^\Delta )(x,t^{n+1})^2- w^\Delta (x,t^{n})^2 \right)\, dx \geq a \int_\R |w^\Delta (x,t^{n+1})- w^\Delta (x,t^{n})|\,dx.
         \end{split}
    \end{equation}
    Now, if $\tilde{u}^n
    (x_{i+1/2}-,t) = \tilde{u}^n(x_{i+1/2}+,t)$ for some $i\in \mathbb{Z}$, then they cancel out in the sum. If instead $\tilde{u}^n
    (x_{i+1/2}-,t) \not= \tilde{u}^n(x_{i+1/2}+,t)$, then there is a stationary shock in $u$ at $x_{i+1/2}$, i.e.
    $\tilde{u}^n(x_{i+1/2}-,t) = u_i^n$ and $\tilde{u}^n(x_{i+1/2}+,t) = u_{i+1}^n$ with $f(u_i^n)=f(u_{i+1}^n)$ and $u_i^n>u_{i+1}^n$.
    (see Subcase \ref{sottocaso: ful=fur}).  In such case 
    \[
    \begin{split}
    -\left(G(\tilde{u}^n(x_{i+1/2}-,t))-G(\tilde{u}^n(x_{i+1/2}+,t))\right) & = -u_i^nf(u_i^n)+u_{i+1}^nf(u_{i+1}^n)+\int_{u_{i+1}^n}^{u_i^n} f(\xi)\,d\xi \\
    &= \int_{u_{i+1}^n}^{u_i^n} f(\xi)\,d\xi  -(u_i^n-u_{i+1}^n)f(u_i^n) \leq 0 
    \end{split}
    \]
    where the last inequality is due to the fact that $f$ is convex and $u_{i}^n> u_{i+1}^n$. \par
    Hence,  $- \Delta t  \sum_{i\in \mathbb{Z}}
         \left( G(\tilde{u}^n(x_{i+1/2}-,t))-G(\tilde{u}^n(x_{i-1/2}+,t)) \right) \leq 0$, so from \eqref{eq: discr2} we finally get 
    \begin{equation}\label{eq: discr3}
        \begin{split}
       -\frac{1}{2}\int_{\R} &\left(u^\Delta (x,t^{n+1})^2- u^\Delta (x,t^{n})^2\right)\,dx
         -\frac{1}{2} \int_{\R}\left(w^\Delta )(x,t^{n+1})^2- w^\Delta (x,t^{n})^2 \right)\, dx \\ &\geq a \int_\R |w^\Delta (x,t^{n+1})- w^\Delta (x,t^{n})|\,dx.
         \end{split}
    \end{equation}
    Let now consider $t\in\, ]0,T[$ such that $t\in [t^N,t^{N+1}[$ for some $N\in \N$. Then, summing \eqref{eq: discr3} over $n=0,\dots, N-1$, we get 
    \begin{equation}
        \begin{split}
        -\frac{1}{2}\int_{\R}& \left(u^\Delta(x,t)^2- u_0(x)^2+ w^\Delta(x,t)^2- w_0(x)^2\right) \, dx
        \\ & \geq a \sum_{n=0}^{N-1} \int_\R \left|w^\Delta (x,t^{n+1})- w^\Delta (x,t^{n})\right|\,dx =  a \left| \partial_t w^\Delta  \right|(\R\times\, ]0,t[).
         \end{split}
    \end{equation}
\end{proof}

\subsection{Convergence of approximate solutions: the existence theorem}
\begin{teo}\label{teo: esistenza}
     Let $T>0,$ $u_0,w_0 \in \BV(\R), |u_0(x)-w_0(x)| \leq a,$ and consider a sequence of meshes $\mathcal T_m$ and time steps $\Delta t_m$ satisfying the CFL condition \eqref{eq: CFL} for every $m\in \N$ and such that $\Delta x_m, \Delta t_m \to 0 $ as $m\to +\infty$. Then the sequence of approximate solutions $\{u_m\}_{m\in{\N}}:=\{u^{\Delta_m}\}_{m\in{\N}}$ and $\{w_m\}_{m\in{\N}}:=\{w^{\Delta_m}\}_{m\in{\N}}$ defined by the scheme \eqref{eq: schema2} and by \eqref{eq: approxsol}  converges in $\Lloc1(\R \times [0,T[)$ to $u$ and $w$ respectively, where $(u,w)$ is an entropy weak solution to \eqref{eq: intro2} according to Definition \ref{def: hweaksol}.
 \end{teo}
 \begin{proof} 
 It can be shown that $u_m,w_m$ converge in $\Lloc{1}(\R \times[0,T[)$ to a couple $u,w\in \C{0} ( [0,T[; \Lloc{1}(\R))$ and that this couple satisfies \eqref{eq: hweaksol}, by following classical arguments; which can be found in Appendix~\ref{appendice}. To conclude, we then have to show inequalities \eqref{eq: dishis} and \eqref{eq: genweakhis}.

Condition \eqref{eq: dishis} is true since it holds for $u_m,w_m$, see \eqref{eq: comp}, and the $\Lloc{1}(\R\times [0,T[)$ convergence implies (up to a subsequence) almost everywhere pointwise convergence. 
 
 Finally, also the weak hysteresis relationship \eqref{eq: genweakhis}, which holds for each approximate solution (see Proposition \ref{prop: whisdiscr}), passes to the limit, for almost every $t\in [0,T]$, due to the $\Lloc{1}(\R \times [0,T[)$ convergence and to the lower semicontinuity of the mass of measures (see e.g. \cite{AFP}). Moreover, since $u_0, w_0 \in \L{1}(\R) \cap \BV(\R)$, it can be checked that $u, w \in \L{1}(\R \times [0,T[)$, and not only locally. In addition, as $u, w \in \L{\infty}(\R \times [0,T[)$, we also have $u, w \in \L{2}(\R \times [0,T])$, which is indeed required by \eqref{eq: genweakhis}.

 \end{proof}

\begin{rmk}\label{rmk: linear} {\bf (The linear case)} The work
\cite{BFMS} considers problem \eqref{eq: intro2} with linear flux $f(u)=u$. The main result of such paper was the existence and uniqueness of a solution for the Cauchy problem with  $\BV$ initial data constructed via the wave-front tracking method. However, no investigation about  numerical schemes was carried out. 

 Reasoning as we did in this paper, it can be shown that in a cell $K_i^n$, the solution  $\tilde{u}^n$ of the linear equation satisfies the conservation law 
 \begin{equation}\label{eq: PDEwi}
    \partial_t u+ \partial_x\bar{f}_{w_i^{n}} (u)=0,
    \end{equation}
    where
    \begin{equation}\label{eq: ftilde}
        \Bar{f}_{w_i^n}= \begin{cases}
            \frac{1}{2} u+\frac{1}{2}(w_i^n-a) \quad& u\in (-\infty, w_i^n -a],\\
            u \quad &u\in [w_i^n-a, w_i^n+a],\\
            \frac{1}{2} u+\frac{1}{2}(w_i^n+a) \quad &u\in [ w_i^n +a,+\infty).
        \end{cases}
    \end{equation}
    Then, since waves have positive speed, the solution to the Riemann problem at $x=0$ coincides with the left datum $u_l$, so 
    \begin{equation}\label{eq: schemaLineare}
    \begin{cases}
         u_i^{n+1} = u_i^{n}-\frac{\Delta t}{\Delta x} \left(f_{w_i^n}(u_i^n)-f_{w_i^n}(u_{i-1}^n)\right), \quad& \forall i\in \mathbb Z,\\
         w_i^{n+1} = u_i^{n}+w_i^{n}-\frac{\Delta t}{\Delta x} (u^n_{i}-u^n_{i-1})-u_i^{n+1}, \quad &\forall i\in \mathbb{Z},\\
         u_i^0= \frac{1}{\Delta x} \int_{x_{i-1/2}}^{x_{i+1/2}} u_0(x) \, dx, \quad w_i^0 = \frac{1}{\Delta x} \int_{x_{i-1/2}}^{x_{i+1/2}} w_0(x) \, dx,\quad&\forall i\in \mathbb{Z},\\
         |u_0(x)-w_0(x)|\leq a, \quad & \forall x\in \R .
        \end{cases}
    \end{equation}
    The numerical solution associated to scheme \eqref{eq: schemaLineare} for the linear flux has the same properties of the more general one considered in this paper. In particular, the corresponding approximate solution converges to the theoretical solution constructed in \cite{BFMS} via a wave-front tracking algorithm, as the size of the mesh goes to $0$, due to the uniqueness of the entropy weak solution.

\end{rmk}

    \subsection{Some numerical examples}

    We present some examples of numerical solutions generated via the scheme \eqref{eq: schema2} to the Riemann problems \eqref{eq: intro2}, \eqref{eq: datiRP}. In particular, Figure~\ref{fig: rarefazioniNum} shows the three possible configurations of rarefaction wave solutions described in Section~\ref{caso: ul<ur}, while Figure~\ref{fig: shocksNum} shows cases involving shock waves, see Section~\ref{caso: ul>ur}. In all cases, we consider the hysteresis parameter $a=1$, the flux function $f(u)= \tfrac{1}{2} u^2$, whereas the other data and parameters are summarized in Table \ref{tab: datiSimulazioni} and $\Delta t$ is such that the CFL condition \eqref{eq: CFL} holds as an equality.
   The analytical solutions of some of these problems are illustrated in Figures \ref{fig: 0minoreulminoreur}--\ref{fig: sottocaso4fs}, Section \ref{S2}.
   In all cases, the numerical scheme captures well the exact profiles for both $u$ and $w$.

\begin{table}[H]
\centering
\begin{tabular}{|l||c c c c|}
\hline
\textbf{Figure}& $(u_l,w_l)$ & $(u_r,w_r)$ & $\Delta x$ & $t$\\
\hline
Figure \ref{fig: rarefazioniNum} (left)  & (1,0.5) & (3,3) & $10^{-3}$ & 0.25 \\
Figure \ref{fig: rarefazioniNum} (middle)  & (-3,-3) & (-1.5,-1)  & $10^{-3}$ & 0.25 \\
Figure \ref{fig: rarefazioniNum} (right) & (-2,-1.5) & (1,1.5) & $10^{-3}$ & 0.4 \\
\cline{1-1}
Figure \ref{fig: shocksNum} (left)  & (1.5,2) & (0.5,0) & $10^{-3}$ & 0.5 \\
Figure \ref{fig: shocksNum} (middle) & (-0.5,0) & (-1.5,-2) & $10^{-3}$ & 0.5 \\
Figure \ref{fig: shocksNum} (right) & (1.5,2) & (-1,-1) & $10^{-3}$ & 0.5 \\
\hline

\end{tabular}
\caption{Parameters used for the simulations.}
\label{tab: datiSimulazioni}
\end{table}

\begin{figure}
\centering
  \begin{minipage}{0.33\linewidth}
     \centering
     \includegraphics[trim=1.4cm 7.5cm 1.5cm 8cm, clip, width=1\linewidth,]{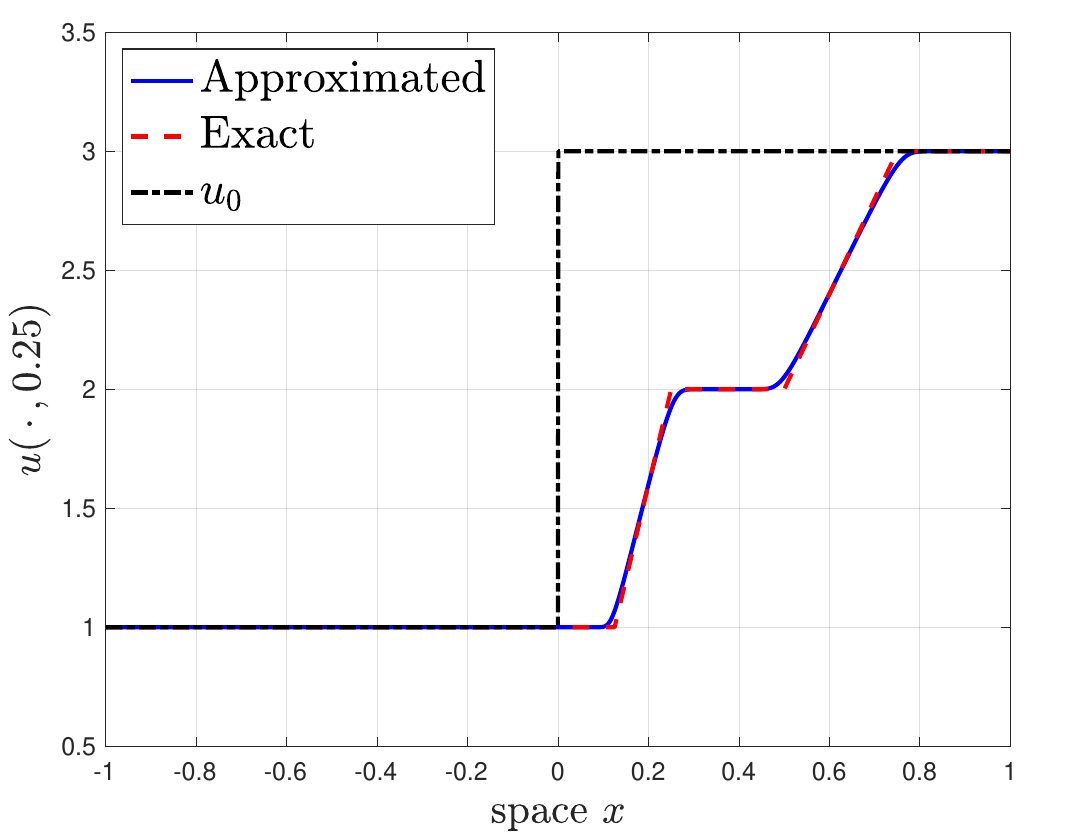}
  \end{minipage}\hfill
  \begin{minipage}{0.33\linewidth}
    \centering
    \includegraphics[trim=1.4cm 7.5cm 1.5cm 8cm, clip, width=1\linewidth]{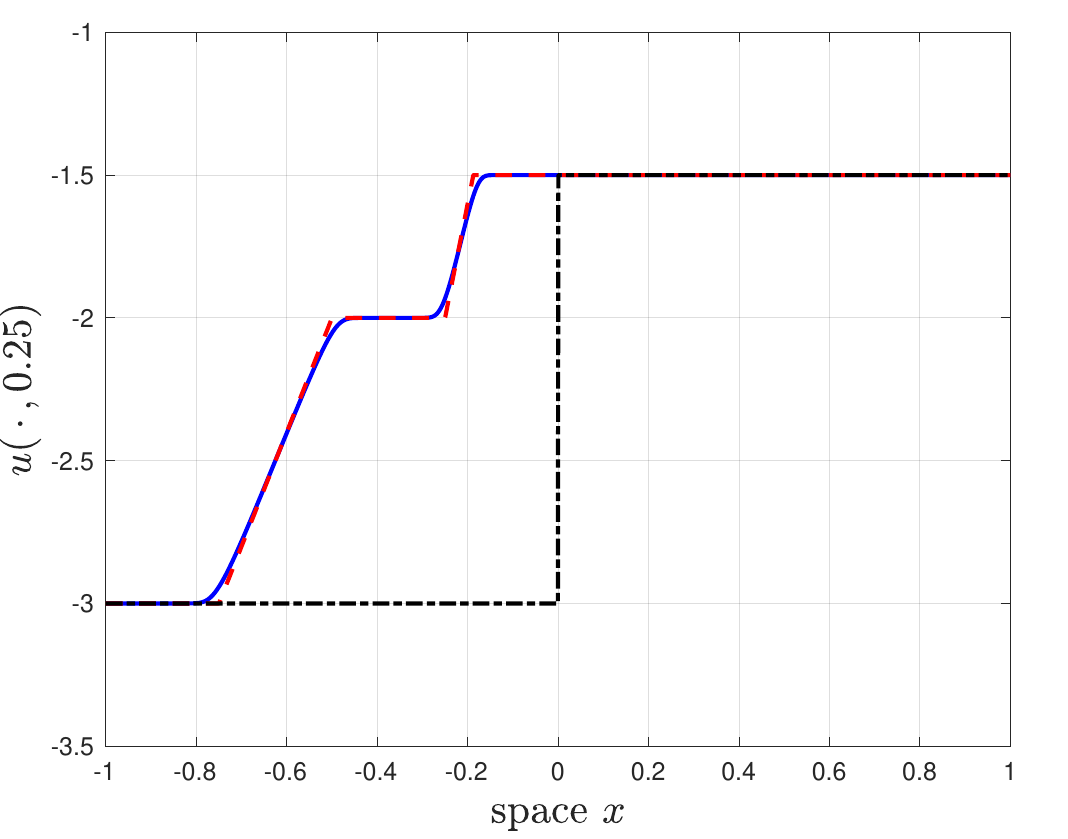}
  \end{minipage}\hfill
  \begin{minipage}{0.33\linewidth}
    \centering
    \includegraphics[trim=1.4cm 7.5cm 1.5cm 8cm, clip, width=1\linewidth]{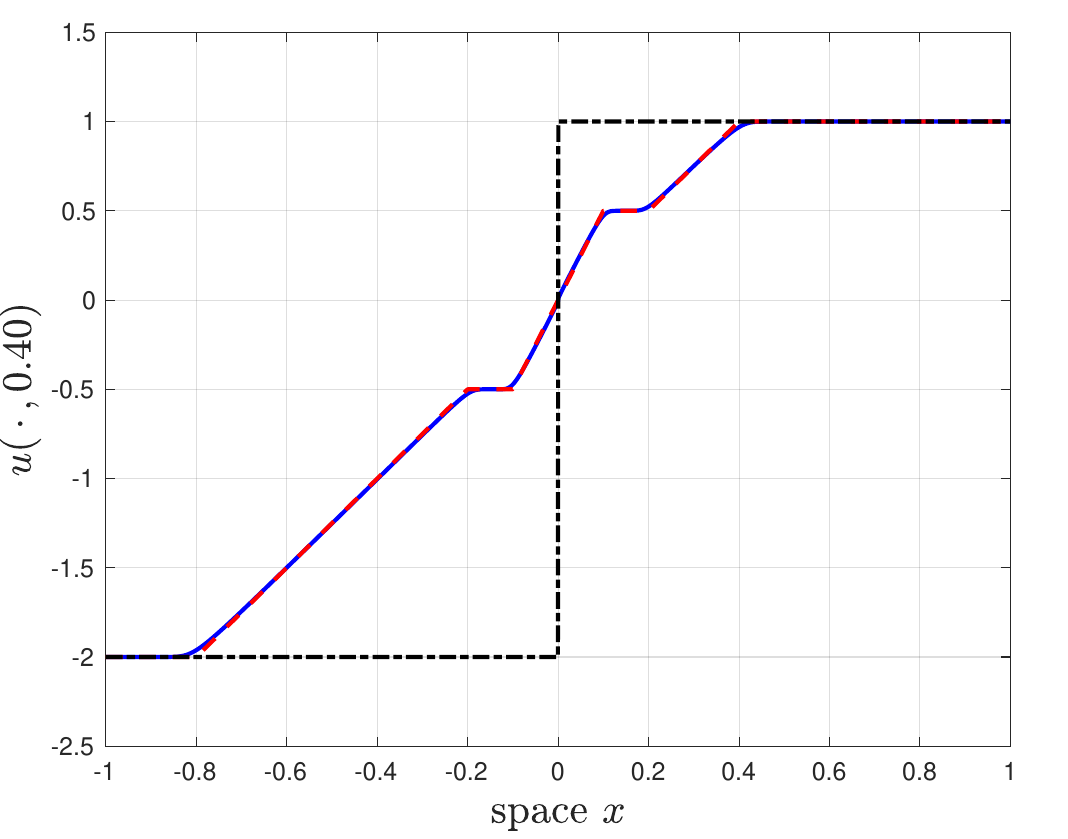}
  \end{minipage}

  \vspace{0 em} 

  \begin{minipage}{0.33\linewidth}
    \centering
    \includegraphics[trim=1.4cm 7.5cm 1.5cm 8cm, clip, width=1\linewidth]{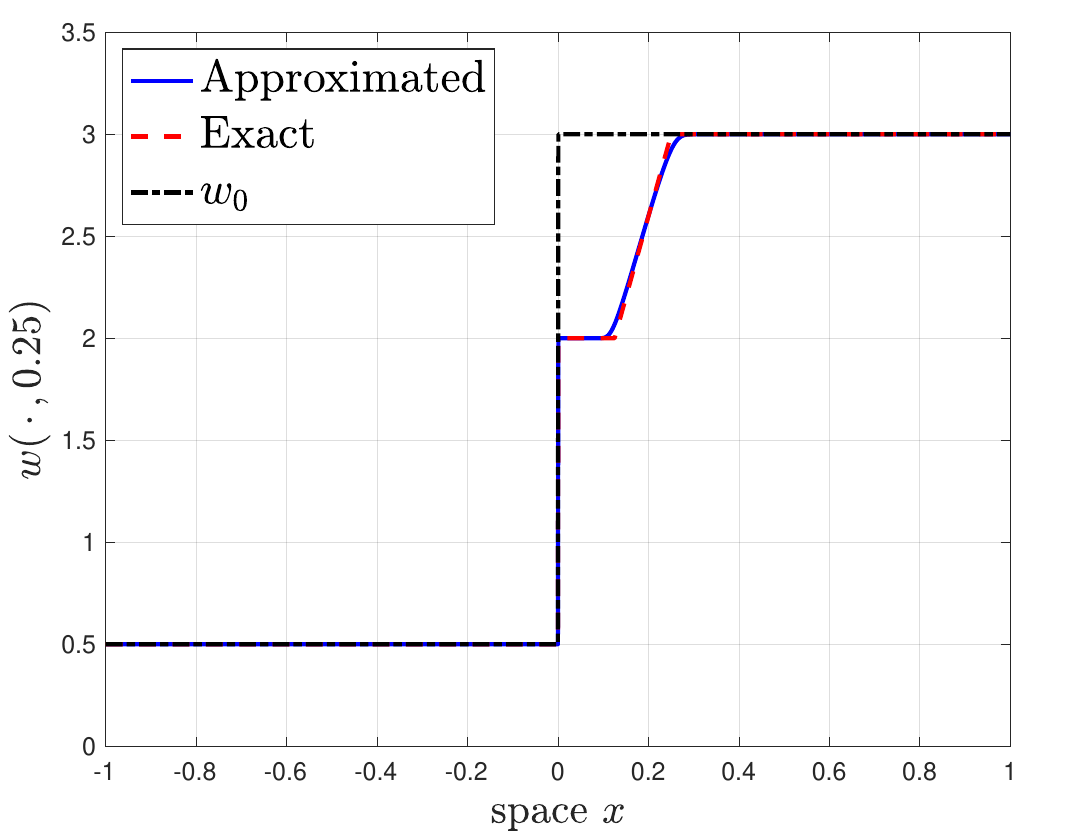}
  \end{minipage}\hfill
  \begin{minipage}{0.33\linewidth}
    \centering
    \includegraphics[trim=1.4cm 7.5cm 1.5cm 8cm, clip, width=1\linewidth]{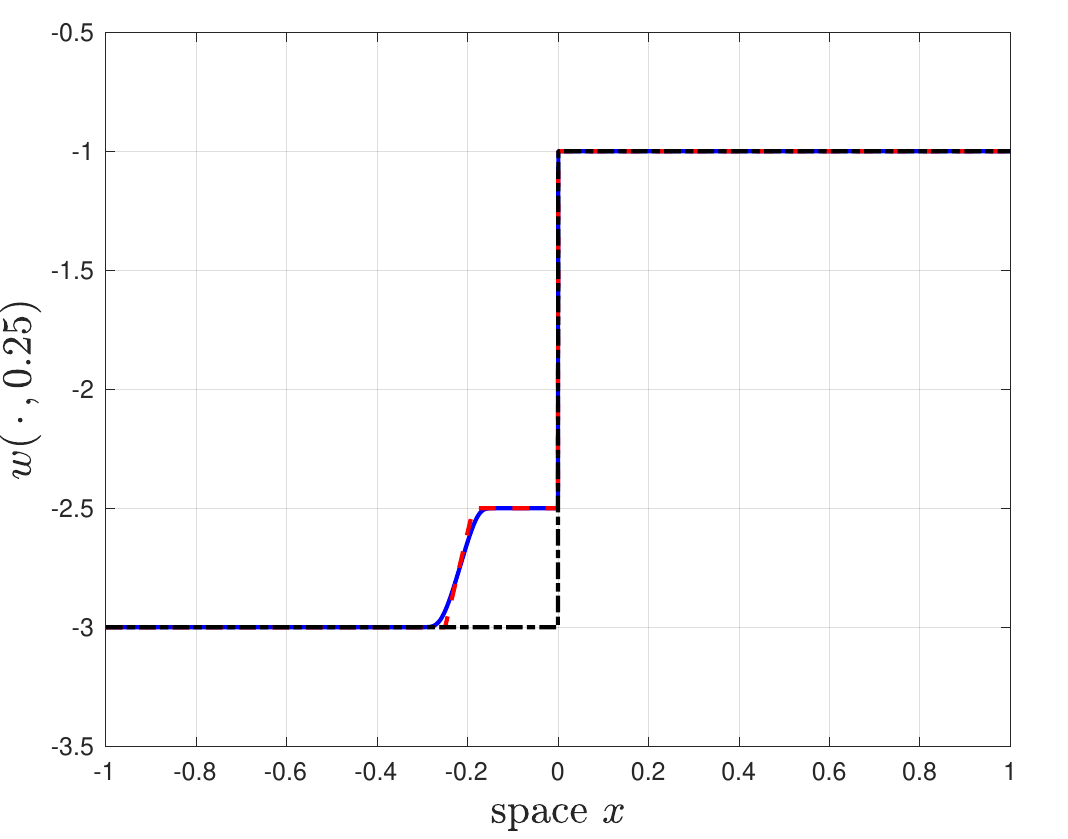}
  \end{minipage}\hfill
  \begin{minipage}{0.33\linewidth}
    \centering
    \includegraphics[trim=1.4cm 7.5cm 1.5cm 8cm, clip, width=1\linewidth]{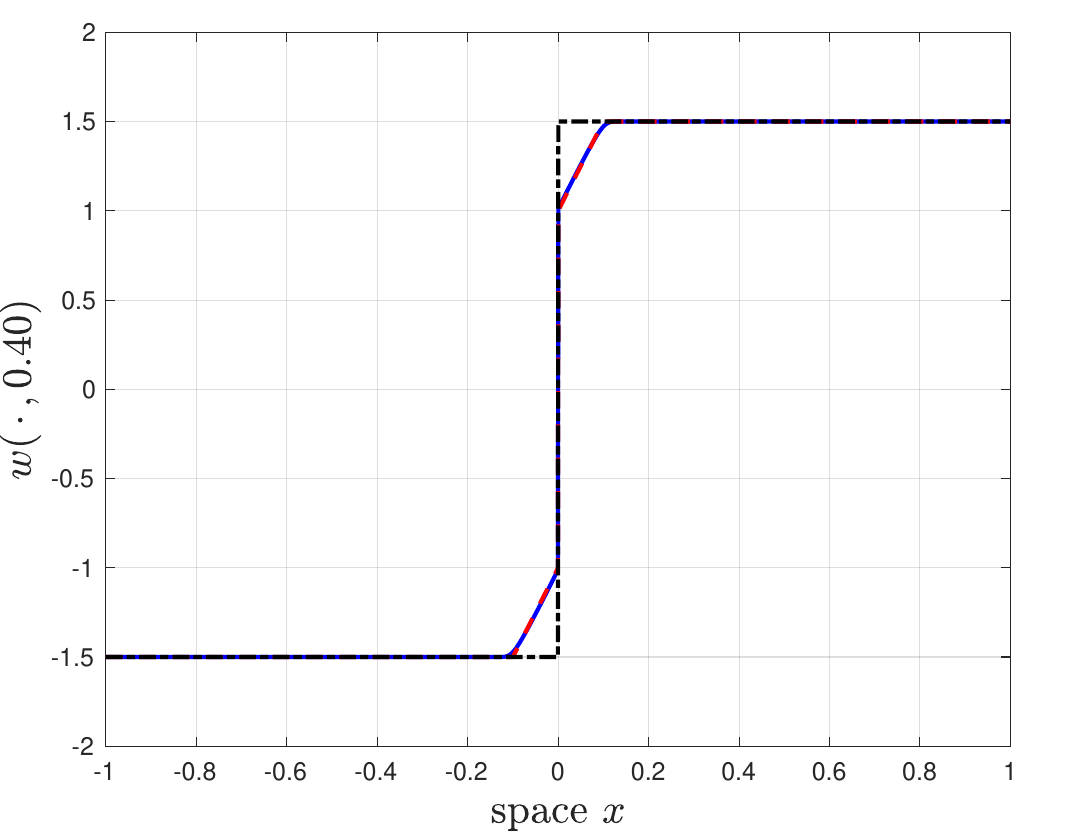}
  \end{minipage}
  \caption{Left: double rarefaction wave for $u$ propagating to the right, see Subcase \ref{sottocaso: 0<ul<ur}. Middle: double rarefaction wave propagating to the left, see Subcase \ref{sottocaso: ul<ur<0}. Right: centered rarefaction, see \ref{sottocaso: ul<0<ur}. }
  \label{fig: rarefazioniNum}
\end{figure}

\begin{figure}
\centering
  \begin{minipage}{0.33\linewidth}
     \centering
     \includegraphics[trim=1.4cm 7.5cm 1.5cm 8cm, clip, width=1\linewidth,]{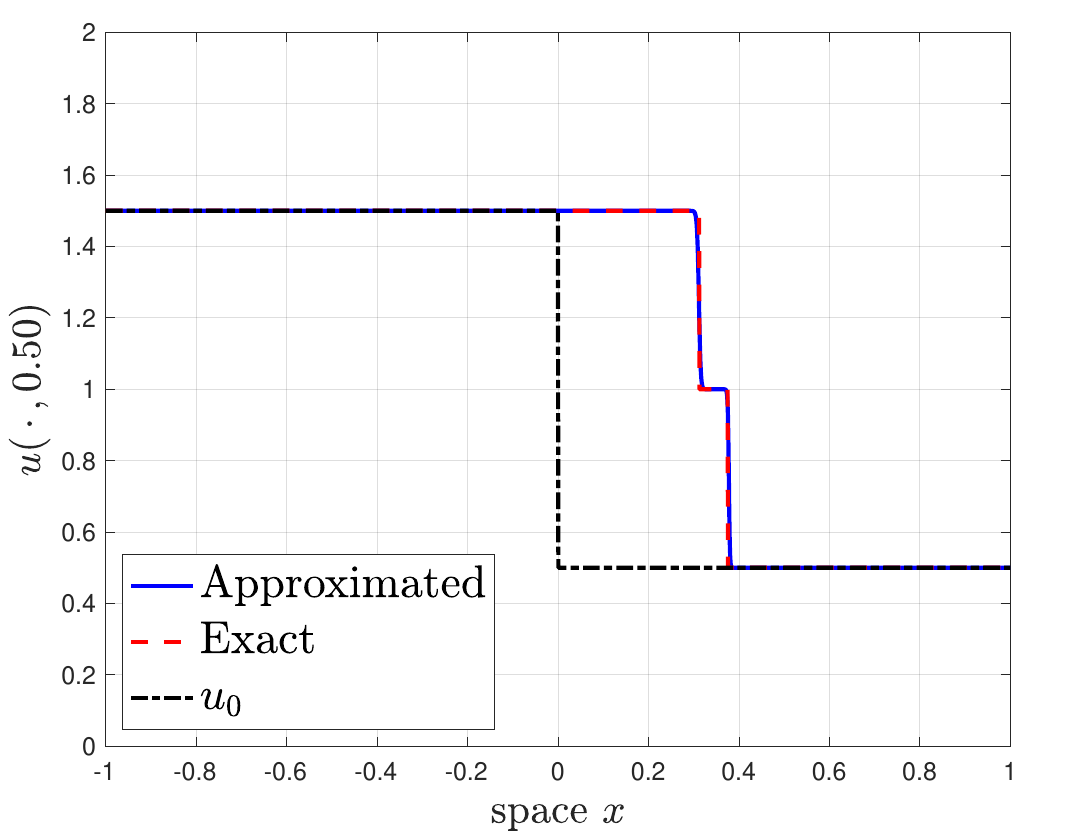}
  \end{minipage}\hfill
  \begin{minipage}{0.33\linewidth}
    \centering
    \includegraphics[trim=1.4cm 7.5cm 1.5cm 8cm, clip, width=1\linewidth]{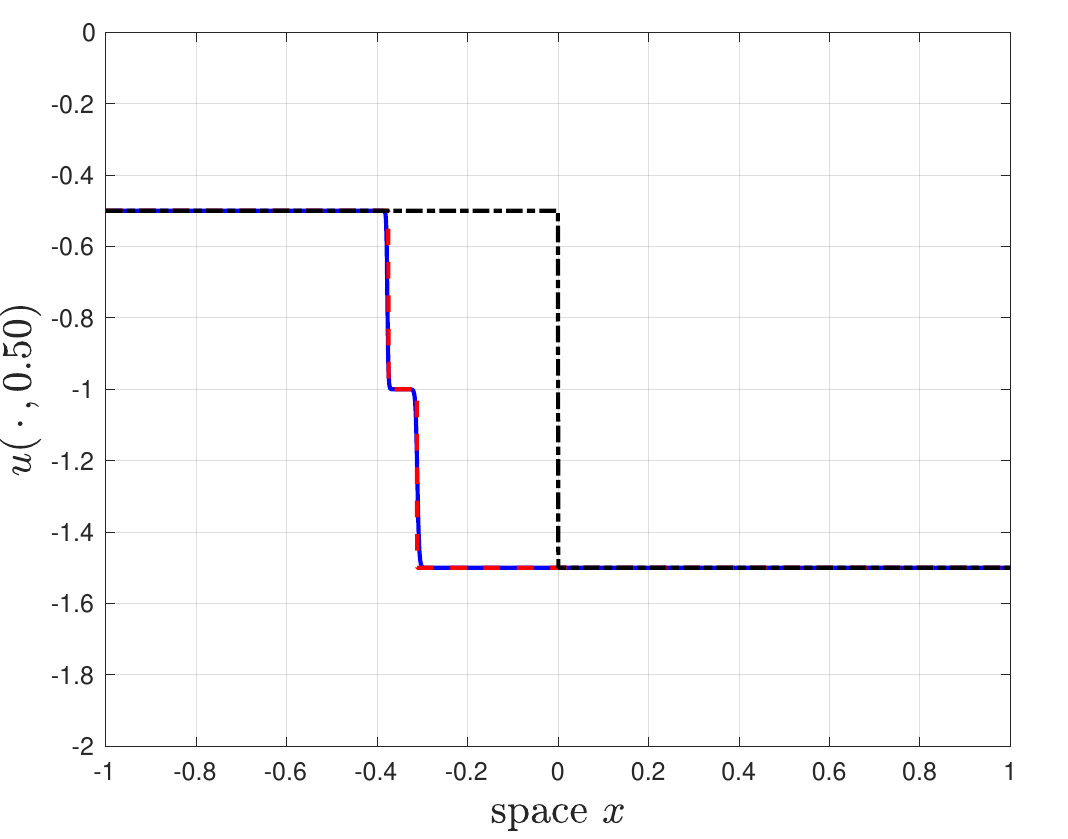}
  \end{minipage}\hfill
  \begin{minipage}{0.33\linewidth}
    \centering
    \includegraphics[trim=1.4cm 7.5cm 1.5cm 8cm, clip, width=1\linewidth]{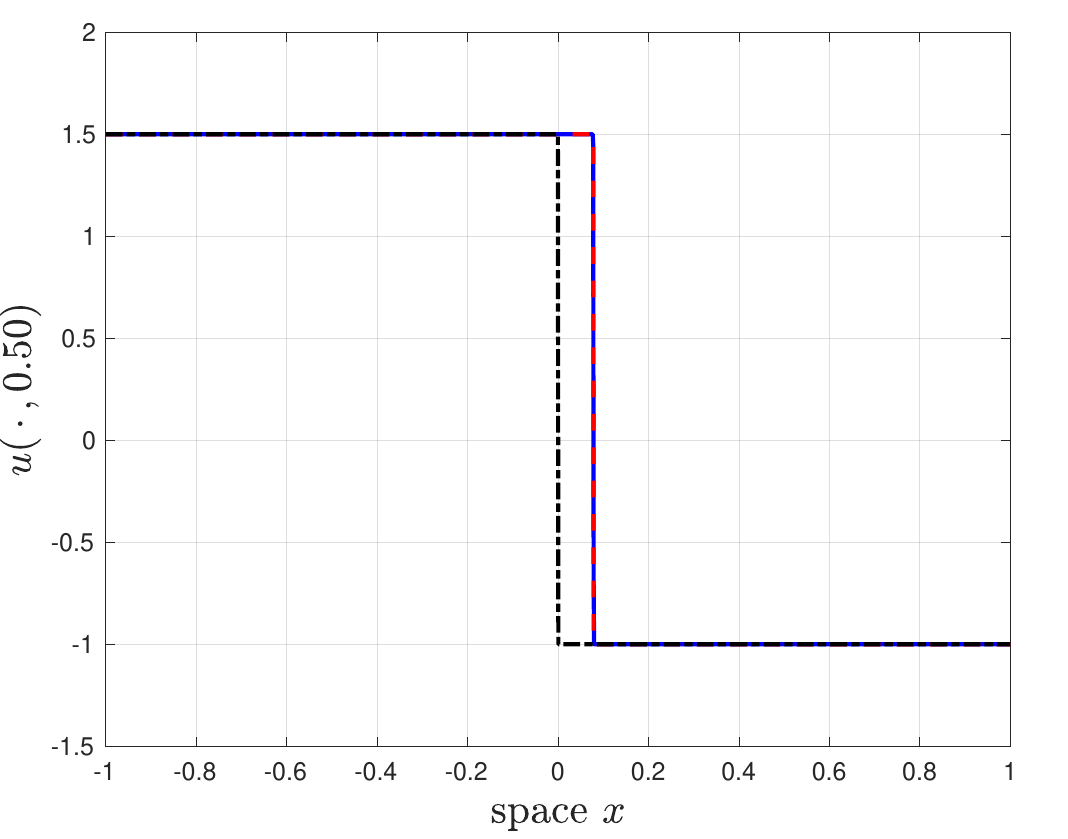}
  \end{minipage}

  \vspace{0 em} 

  \begin{minipage}{0.33\linewidth}
    \centering
    \includegraphics[trim=1.4cm 7.5cm 1.5cm 8cm, clip, width=1\linewidth]{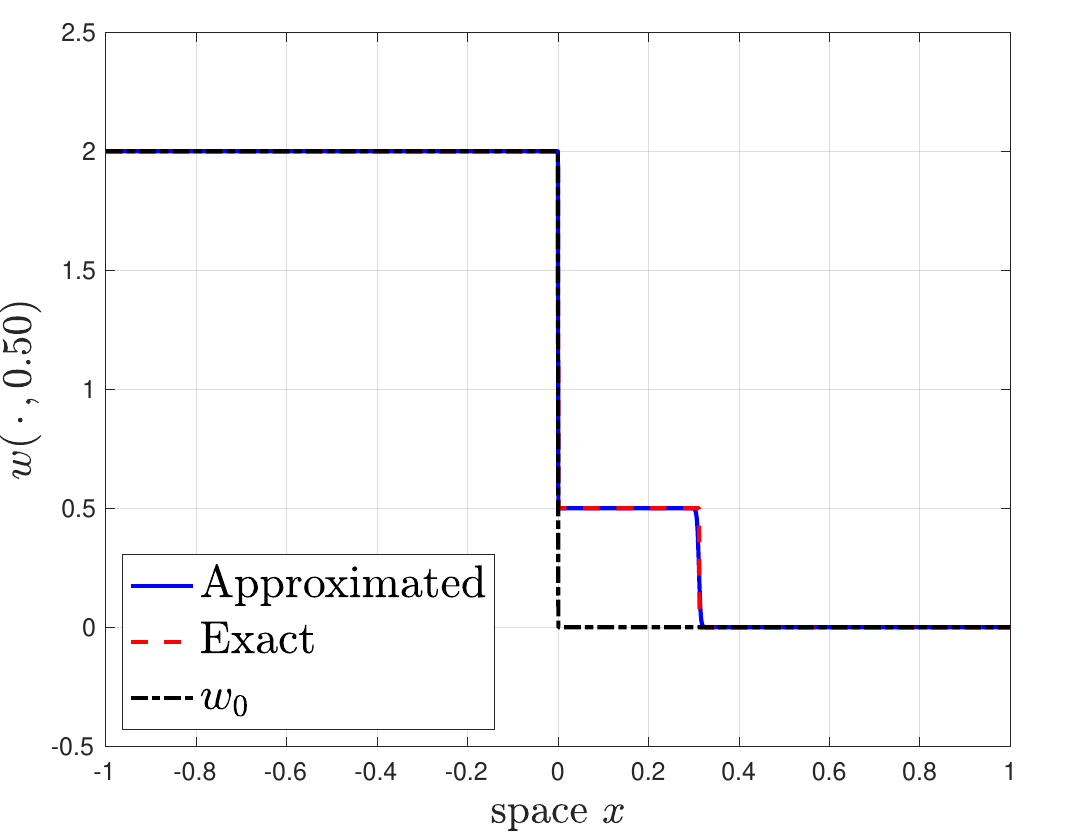}
  \end{minipage}\hfill
  \begin{minipage}{0.33\linewidth}
    \centering
    \includegraphics[trim=1.4cm 7.5cm 1.5cm 8cm, clip, width=1\linewidth]{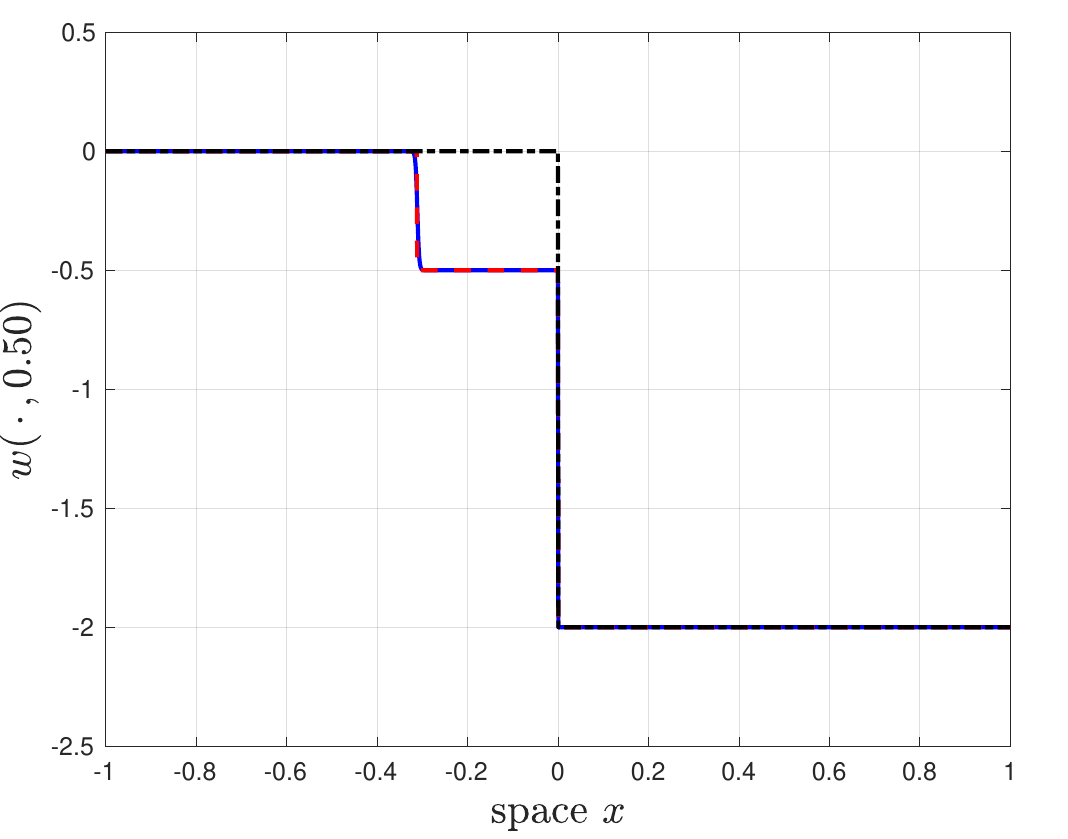}
  \end{minipage}\hfill
  \begin{minipage}{0.33\linewidth}
    \centering
    \includegraphics[trim=1.4cm 7.5cm 1.5cm 8cm, clip, width=1\linewidth]{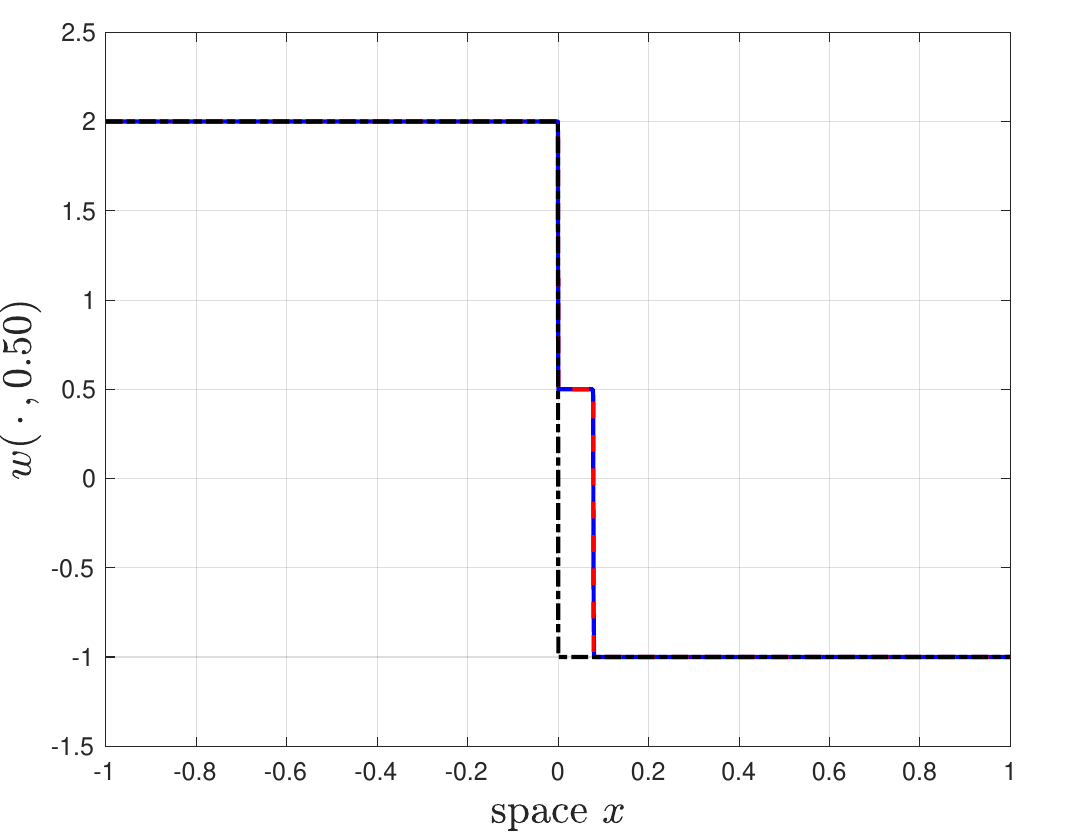}
  \end{minipage}
  \caption{Left: double shock wave for $u$ propagating to the right with $\mu_r \geq \mu_l$, see Subcase \ref{sottocaso: ful>fur}. Middle: double shock wave propagating to the left with $\nu_r \geq \nu_l$, see Subcase \ref{sottocaso: ful<fur}. Right: ``fast shock'' with with $\mu_r < \mu_l$, see again Subcase \ref{sottocaso: ful>fur}.} 
  \label{fig: shocksNum}
\end{figure}

We also compute the numerical solution $(u^\Delta, w^\Delta)$ of the Cauchy problem with a more general initial datum, namely $u_0 = e^{-x^2/2},$  with $u_0 = w_0$, the flux function $f(u)=\tfrac{1}{2}u^2$, $a=1$, $\Delta x = 10^{-3}$ and $\Delta t$ such that the CFL condition \eqref{eq: CFL} holds as an equality.  Figure \ref{fig: cauchyNum} shows the evolution of $u^\Delta(\cdot,t)$ and $w^\Delta(\cdot,t)$ at increasing time steps $t$.

In the same plots, $u^\Delta$ is compared with the solutions to the two equations without hysteresis 
\[
\partial_t u + \partial_x f(u) =0 \quad \text{and} \quad \partial_t u + \frac{1}{2} \partial_x f(u)=0
\]
which are depicted by the dashed red line and by the dot-dashed blue line, respectively. To compute such solutions we still used the scheme \eqref{eq: schema2}, with flux either $f(u)= u^2/2$ or $f(u) =u^2/4$, where, by setting the hysteresis parameter $a$ large enough and by also setting $u_0=w_0$, we ensure $\partial_t w=0$ for all $(x,t).$ In this way, by considering only the unknown $u,$ we are approximating the solutions of the two equations without hysteresis.

Referring again to Figure \ref{fig: cauchyNum}, we observe that the speed with which $u^\Delta$ propagates to the right is intermediate between those of the solutions of the two equations without hysteresis. This behaviour is expected as the equation with hysteresis is a combination of the non-hysteretic two, recall Remark \ref{rmk: nofastshock}. Moreover, the presence of hysteresis lowers the peak of $u^\Delta$, as large variations in $u(x,\cdot)$ lead to $(u(x,\cdot),w(x,\cdot))$ following the hysteresis boundary, where $\partial_t w \neq 0$. From the relationship $\partial_t u + \partial_x f(u) = -\partial_t w$, we see that the term $-\partial_t w$ acts as a source term damping large variations of $u$, since when $(u,w)$ is on the boundary of hysteresis regions $\sgn(\partial_t u) = - \sgn(-\partial_t w)$. 

In Table~\ref{tab: sommaNormeL2}, we also highlight the sum of the $\L{2}$ norms squared in space of $u^\Delta$ and $w^\Delta$, which decreases in time accordingly with \eqref{eq: genweakhis}.

\begin{figure}
\centering
  \begin{minipage}{0.25\linewidth}
     \centering
     \includegraphics[trim=1.4cm 7.5cm 1.5cm 8cm, clip, width=1\linewidth,]{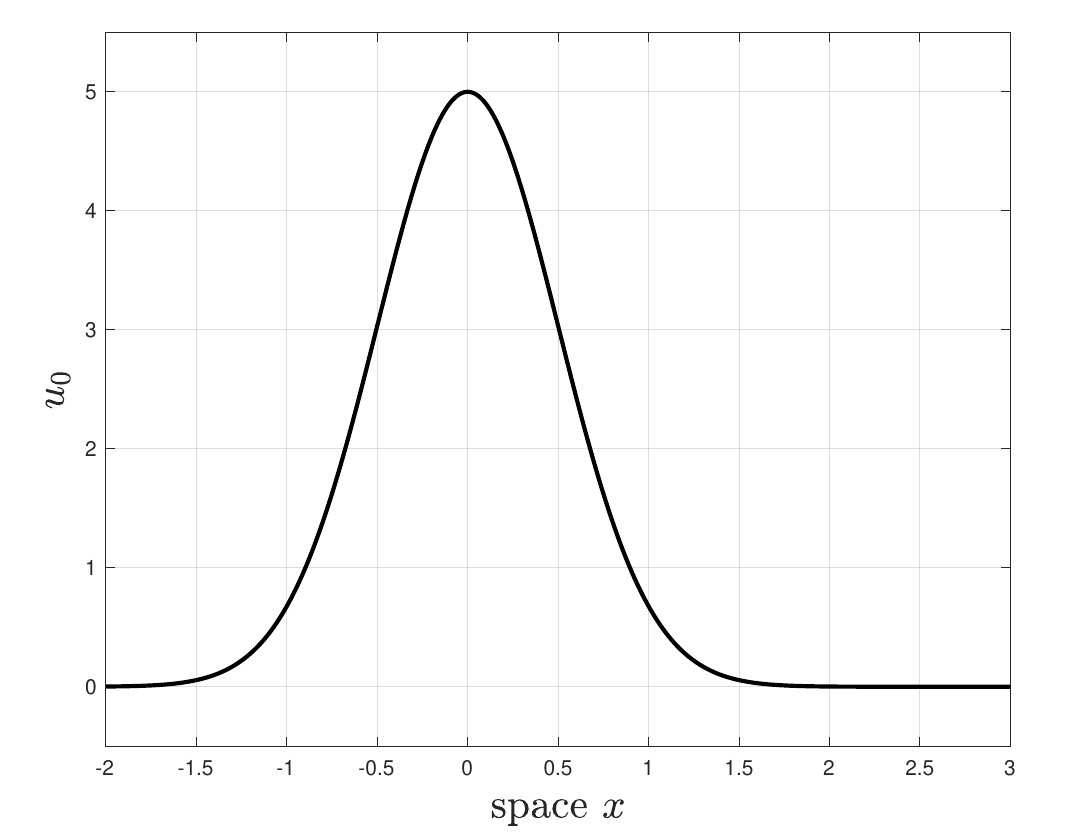}
  \end{minipage}\hfill
  \begin{minipage}{0.25\linewidth}
    \centering
    \includegraphics[trim=1.4cm 7.5cm 1.5cm 8cm, clip, width=1\linewidth]{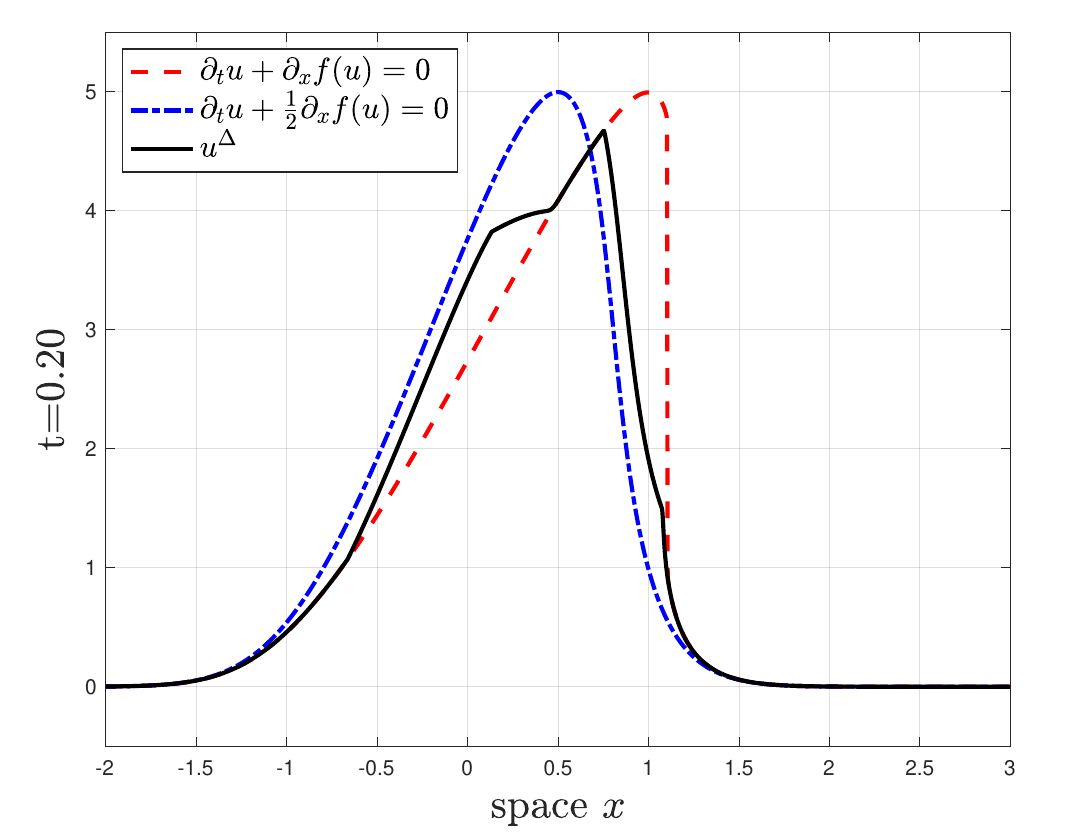}
  \end{minipage}\hfill
  \begin{minipage}{0.25\linewidth}
    \centering
    \includegraphics[trim=1.4cm 7.5cm 1.5cm 8cm, clip, width=1\linewidth]{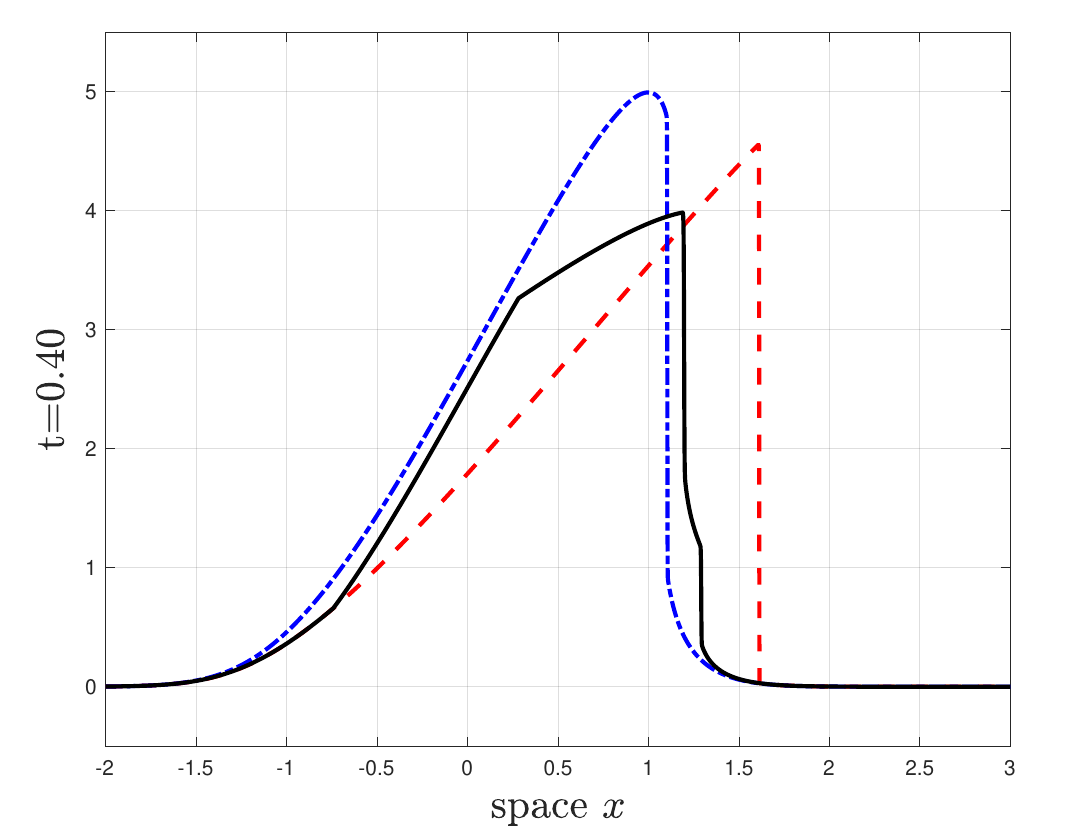}
  \end{minipage}\hfill
  \begin{minipage}{0.25\linewidth}
    \centering
    \includegraphics[trim=1.4cm 7.5cm 1.5cm 8cm, clip, width=1\linewidth]{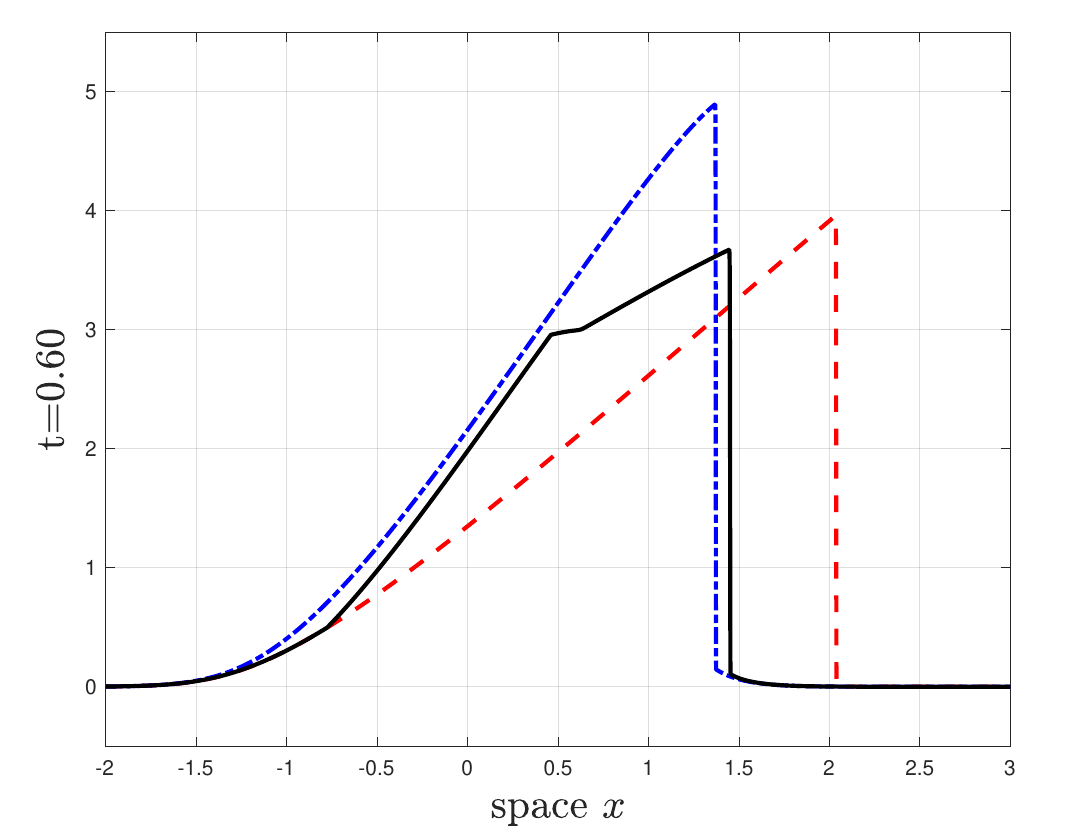}
  \end{minipage}

  \vspace{0 em} 
  
  \begin{minipage}{0.25\linewidth}
     \centering
     \includegraphics[trim=1.4cm 7.5cm 1.5cm 8cm, clip, width=1\linewidth,]{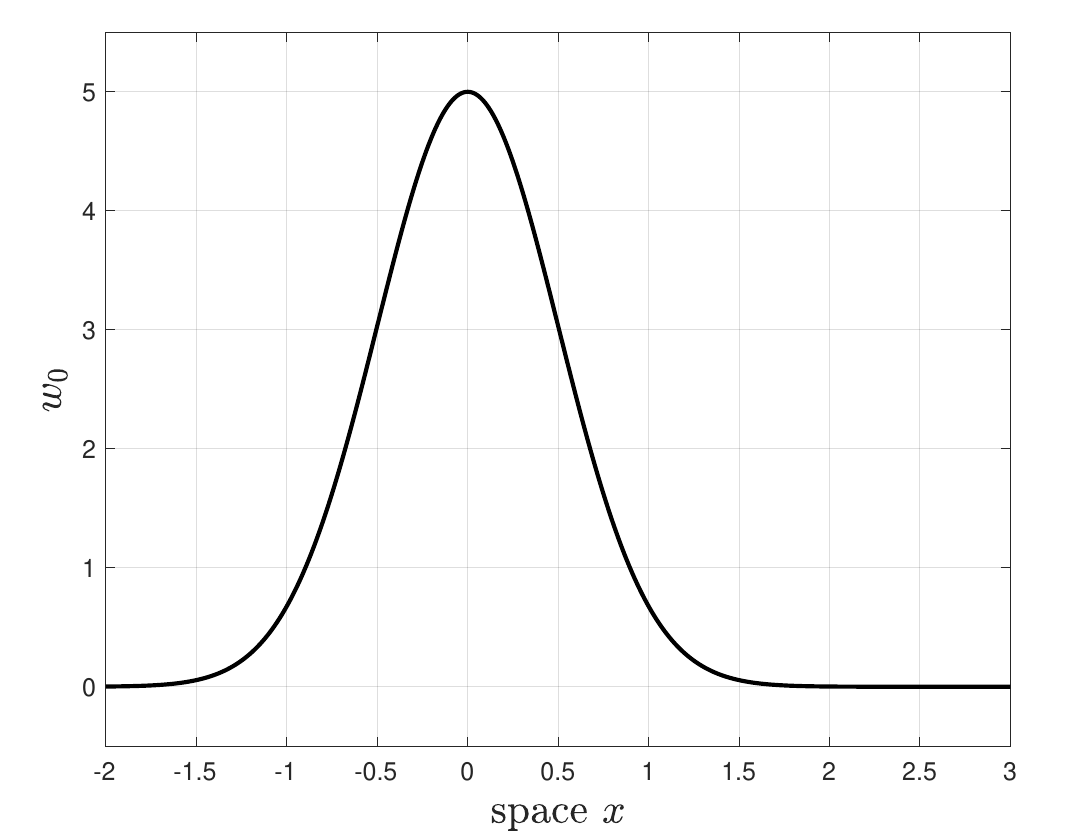}
  \end{minipage}\hfill
  \begin{minipage}{0.25\linewidth}
    \centering
    \includegraphics[trim=1.4cm 7.5cm 1.5cm 8cm, clip, width=1\linewidth]{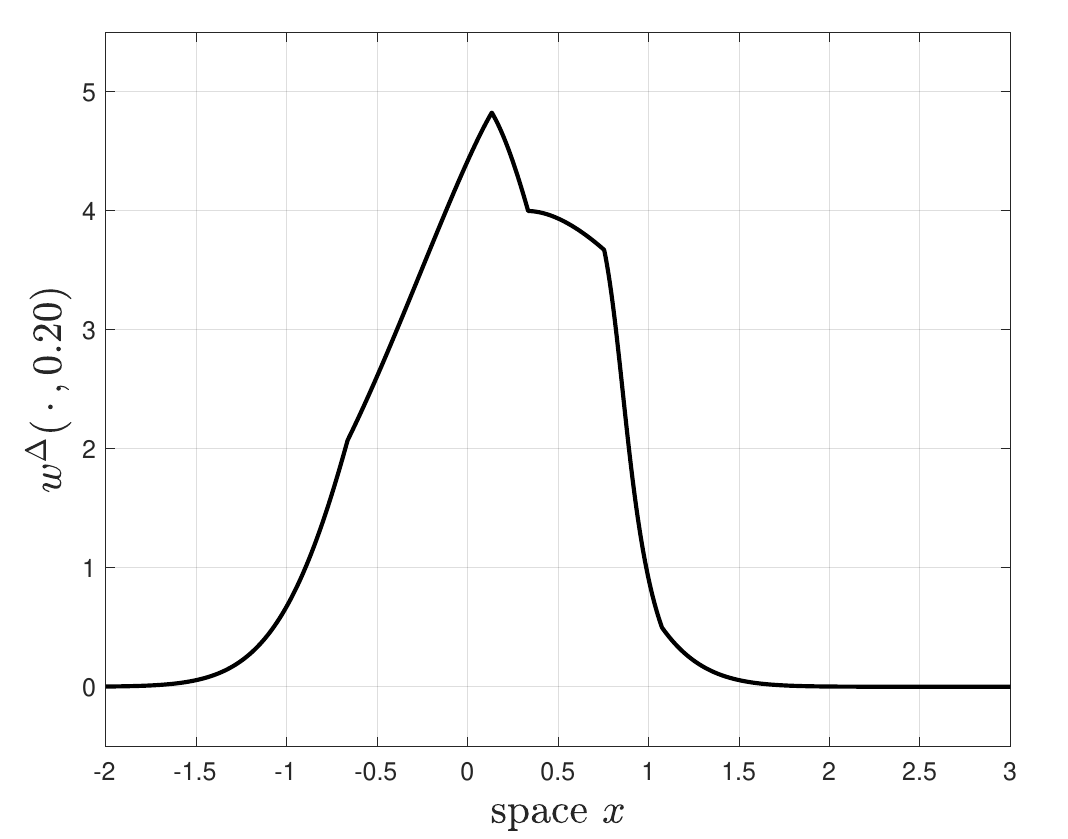}
  \end{minipage}\hfill
  \begin{minipage}{0.25\linewidth}
    \centering
    \includegraphics[trim=1.4cm 7.5cm 1.5cm 8cm, clip, width=1\linewidth]{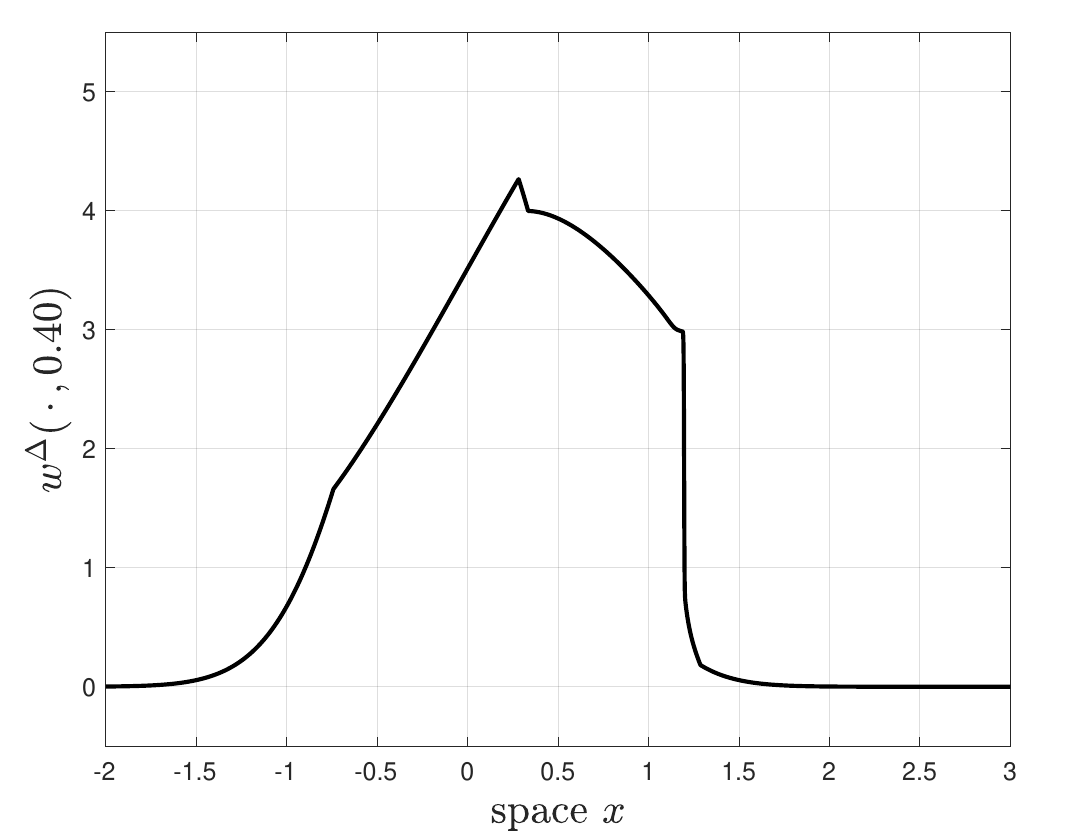}
  \end{minipage}\hfill
  \begin{minipage}{0.25\linewidth}
    \centering
    \includegraphics[trim=1.4cm 7.5cm 1.5cm 8cm, clip, width=1\linewidth]{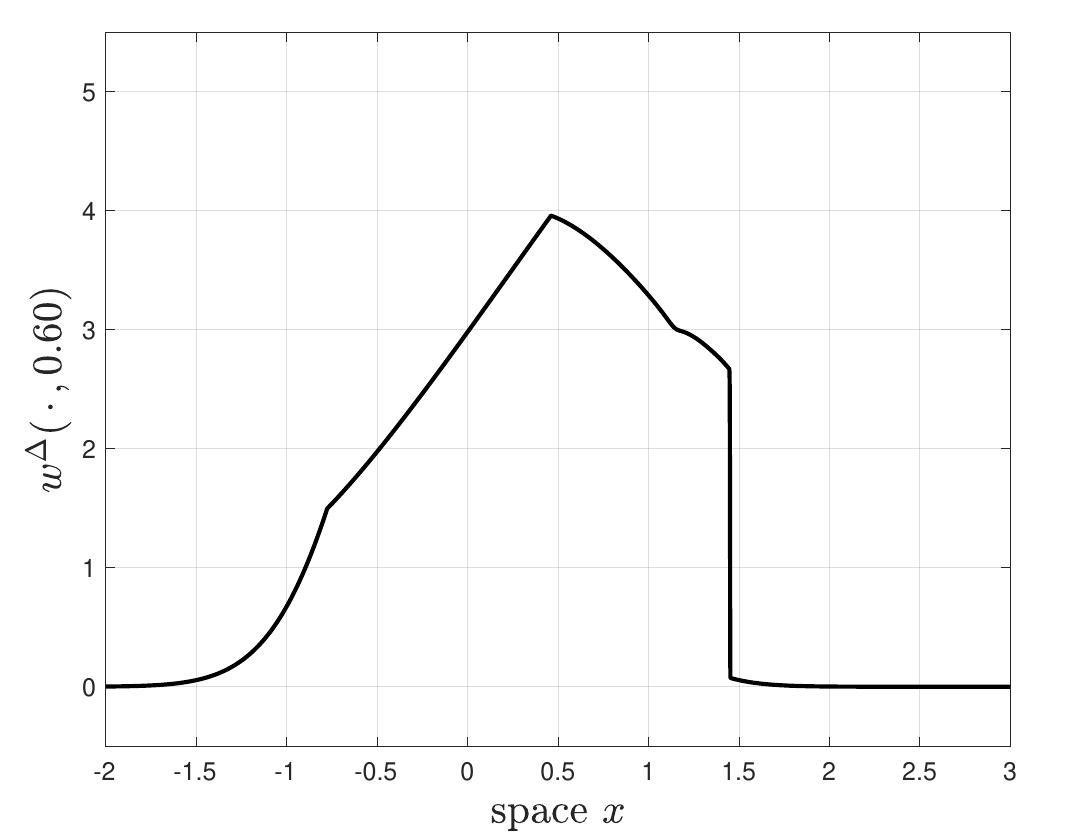}
  \end{minipage}
  \caption{Graphs of $u^\Delta(\cdot,t)$ (top row) and $u^\Delta(\cdot,t)$ (bottom row) computed at increasing times. The top plots include also the solutions to the equations without hysteresis $\partial_t u + \partial_x f(u)=0$ (dashed red) and $\partial_t u + \tfrac{1}{2}\partial_x f(u)=0$ (dot-dashed blue).}
  \label{fig: cauchyNum}
\end{figure}

\begin{table}[H]
\centering
\begin{tabular}{|c||c c c c |}
\hline
\textbf{$\L{2}$ norm square} & $t=0$ & $t=0.2$ & $t=0.4$ & $t=0.6$ \\
\hline
$||u^\Delta(\cdot\,,t)||_{\L{2}(\R)}^2$ & 22.1557 & 19.2978 &  17.0486 &  14.9098  \\
$||w^\Delta(\cdot\,,t)||_{\L{2}(\R)}^2$& 22.1557 & 22.5789 & 21.8581  &  21.0941  \\
$||u^\Delta(\cdot\,,t)||_{\L{2}(\R)}^2$ + $||w^\Delta(\cdot\,,t)||_{\L{2}(\R)}^2$ & 44.3114 &  41.8767 &  38.9067 &  36.0039 \\
\hline

\end{tabular}
\caption{The $\L{2}(\R)$ norm square of $u^\Delta(\cdot\,,t)$, $w^\Delta(\cdot\,,t)$ and their decreasing in time sum.}
\label{tab: sommaNormeL2}
\end{table}


\section{Stability}\label{S6}

Now we show how the condition \eqref{eq: hweaksol} in the definition of entropy weak solution characterizes the entropy-allowed shock discontinuities.

\begin{prop}\label{prop: entropia}
    Suppose a couple $(u,w)\in \C{0}([0,T[,\Lloc{1}(\R))$
    to have a shock discontinuity on a curve $(\sigma(t),t)$ between two constant states $(u_-,w_-)\not=(u_+,w_+)$ and to be an entropy weak solution away from that curve. Then this couple satisfies \eqref{eq: hweaksol} and is an entropy weak solution on the whole domain, i.e. this shock is entropy admissible, if and only if $\sigma'$ satisfies the \RH condition \eqref{eq: hrh}, $u_- \geq u_+$ and one of the following holds 
    \begin{enumerate}[i)]
         \item $w_-= w_+$;\label{casoi}
        \item $f(u_-)=f(u_+)$ and $\sigma' = 0$;\label{casoii}
        \item $w_- > w_+$, $f(u_-)>f(u_+)$, $w_- = u_- -a$ and $w_+ = u_+-a$;\label{casoiii}
        \item $w_- > w_+$, $f(u_-)>f(u_+)$, $w_- = u_- -a$ and $\mu_+ \leq \mu_-$ where \[\mu_+= \frac{f(w_++a)-f(u_+)}{(w_++a)-u_+}, \quad \text{and} \quad \mu_-= \frac{1}{2}\frac{f(u_-)-f(w_++a)}{u_--(w_++a)};\]\label{casoiv}
        \item $w_- > w_+$, $f(u_-)<f(u_+)$, $w_+ = u_+ +a$ and $w_- = u_- +a$;\label{casov}
        \item $w_- > w_+$, $f(u_-)<f(u_+)$, $w_+ = u_+ +a$ and $\nu_+ \leq \nu_-$ where \[\nu_-= \frac{f(u_-)-f(w_--a)}{u_--(w_--a)} , \quad \text{and} \quad \nu_+ = \frac{1}{2} \frac{f(w_--a)-f(u_+)}{w_--a-u_+}.\]\label{casovi}
        \end{enumerate}
\end{prop} 

\begin{rmk}
    Before proving the above proposition, let us underline that, when solving the Riemann problem in Case \ref{caso: ul>ur}, Section \ref{S2},  we only used entropy admissible shocks, see Remarks \ref{rmk: entropyshock1} and \ref{rmk: entropyshock2}.
\end{rmk}

\begin{proof}
     Integrating by parts \eqref{eq: hweaksol} and following the same argument as in \cite[Theorem 4.3]{AB3} we can rewrite the entropy condition along the discontinuity as \begin{equation}\label{eq: hdisugentropy}   
     \begin{split}
     \sigma'&\left[\left(|u_+-k|+|w_+-\hat{k}|\right)-
            \left(|u_--k|+|w_--\hat{k}|\right)\right]\\&-\left[\sgn(u_+-k)(f(u_+)-f(k))-\sgn(u_- -k)\left(f(u_-)-f(k)\right)\right]\geq 0.
        \end{split}
   \end{equation} 
   So the couple $(u,w)$ is an entropy solution if and only if \eqref{eq: hdisugentropy} holds  for every $(k,\hat{k}) \in \LM$. In particular, the \RH condition \eqref{eq: hrh}  can be deduced as a necessary condition for weak solutions. 

   First of all, let us notice that, if $w_- = w_+$, then \eqref{eq: hdisugentropy} reduces to the classical entropy condition on shocks and also the \RH condition \eqref{eq: hrh} becomes the classical one. In such case, it is well known that the necessary and sufficient condition for entropic shocks is $\sigma'$ to satisfy the \RH condition with $u_->u_+$; this deals with case \textit{\ref{casoi})}. 
   
   From now on, we suppose $w_-\not = w_+$.
We distinguish the following cases according to the relationship between $f(u_-)$, $ f(u_+)$, $u_-$ and $u_+$.

   \textbf{Case $\mathbf{f(u_-)=f(u_+)}$:} From \eqref{eq: hrh}, either $\sigma'=0$ or $u_--u_++w_--w_+=0$. But if the latter is true, it is sufficient to first choose in \eqref{eq: hdisugentropy}, $(k,\hat{k)}\in \LM$ with $k\geq \max(u_-,u_+)$ and $\hat{k}\leq \min(w_-,w_+)$, then choose instead $(k,\hat{k})\in \LM$ with $k\leq \min(u_-,u_+)$ and $\hat{k}\geq \max(w_-,w_+)$, to infer that 
   \[
   \sigma'  [u_--u_++w_+-w_-]=0,
   \] 
   implying $\sigma' = 0$. Notice that the choices of $(k,\hat{k})\in \LM$ are possible as $(u_-,w_-), (u_+,w_+)\in \LM$ with $u_--u_++w_--w_+=0$, moreover $u_--u_++w_+-w_-\not = 0$ since $u_--u_++w_--w_+=0$ and $(u_-,w_-)\not=(u_+,w_+).$ So, in any case when $f(u_-)=f(u_+)$, necessarily $\sigma'=0$, thus \eqref{eq: hdisugentropy} becomes 
   \begin{equation}\label{eq: casouguale}
       - \sgn(u_+-k)(f(u_+)-f(k))+ \sgn(u_--k)(f(u_-)-f(k)) \geq 0.
   \end{equation}
   Now it is easy to show, thanks to convexity of $f$ and since $f(u_-)=f(u_+)$, that \eqref{eq: casouguale} is true if and only if $u_-\geq u_+.$ So we just proved that under the \RH condition, when $f(u_-)=f(u_+)$, $(u,w)$ is an entropy solution if and only if $\sigma'=0$ and $u_-\geq u_+$; thus if and only if $u_-\geq u_+$ and case \textit{\ref{casoii})} holds. 
   
   Before proceeding with the other cases, we notice the latter case includes also the case when $u_-=u_+$ and does not pose any restriction on $w_-$ and $w_+$. Moreover, we can now also assume $u_--u_++w_--w_+\not=0$, as the opposite would imply $\sigma'=0$, so by the \RH condition \eqref{eq: hrh} $f(u_-)=f(u_+)$. This also allows us, by using again \eqref{eq: hrh}, to write $\sigma'$  as a function of $(u_-,w_-)$ and $(u_+,w_+)$ and substitute it in \eqref{eq: hdisugentropy}. Thus, we get the equivalent condition
    \begin{equation}\label{eq: fratta}
        \frac{h(k,\hat{k})}{u_--u_++w_--w_+} \geq 0,
    \end{equation} where \begin{equation}\label{eq: hkappa}
    \begin{split}
        h(k,\hat{k}):=&\left(f(u_-)-f(u_+)\right)\left[\left(|u_+-k|+|w_+-\hat{k}|\right)-
            \left(|u_--k|+|w_--\hat{k}|\right)\right]\\
            &-(u_--u_++w_--w_+)\left[\sgn(u_+-k)\left(f(u_+)-f(k)\right)\right.\\
            &\qquad\qquad\qquad\qquad\qquad\qquad-\sgn(u_- -k)\left.\left(f(u_-)-f(k)\right)\right],
        \end{split}
   \end{equation} 
   that must hold for every $(k,\hat{k})\in \LM.$ 
   (For sake of precision,  $h=h(k,\hat{k}; u_-,w_-,u_+,w_+)$,
   but for simplicity we will omit the explicit dependence on $(u_-,w_-)$ and $(
   u_+,w_+)$ in the notation.) We also notice that 
   \[
   h(u_-,w_-)= \left(f(u_-)-f(u_+)\right)(w_+-w_-)\left(\sgn(w_+-w_-)-\sgn(u_+-u_-)\right),
   \] 
   and 
   \[
   h(u_+,w_+)= -\left(f(u_-)-f(u_+)\right)(w_+-w_-)\left(\sgn(w_+-w_-)-\sgn(u_+-u_-)\right),
   \] 
   so, if $\sgn(w_+-w_-)\not = \sgn(u_+-u_-)$, then we would have $h(u_-,w_-)$ with opposite sign with respect to $h(u_+,w_+)$. As $h$ is continuous with respect to $(k,\hat{k})$ and $(u_-,w_-),(u_+,w_+)\in \LM$, then we could find multiple couples $(k,\hat{k})\in \LM$ either in neighbourhood of $(u_-,w_-)$ or $(u_+,w_+)$ so that \eqref{eq: fratta} does not hold. Consequently \eqref{eq: fratta} implies that either $u_->u_+$ and $w_+>w_-$ or $u_-<u_+$ and $w_- < w_+$. This allows us to reduce the study to only the following remaining $4$ cases:
   
   \textbf{Case $\mathbf{u_->u_+}$ and $\mathbf{f(u_-) > f(u_+)}$:} by the previous observation $u_--u_++w_--w_+>0$, hence \eqref{eq: fratta} holds if and only if $h(k,\hat{k})\geq0$ for every $(k,\hat{k})\in \LM$. To understand the sign of $h$, we split the $(k,\hat{k})-$plane in $9$ regions, representing the cases when $k\leq u_+$, $u_+\leq k \leq u_-$ or $u_- \leq k$ combined with the cases $\hat{k}\leq w_+$, $w_+\leq \hat{k} \leq w_-$ or $w_- \leq \hat{k}$, see Figure \ref{fig: segni}. We also denote the interior of each of these regions by $R_i$, numbered as shown in the figure.

   \begin{figure}
    \centering
    \begin{tikzpicture}[scale=0.6]

    \draw[->, gray] (-0.5,1.5) -- (10.7,1.5) node[right] {$k$};
    \draw[->, gray] (0.4,-0.1) -- (0.4,9.7) node[above] {$\hat{k}$};
    
    \draw[-,dashed](3,0)--(3,9);
    \draw[-,dashed](7,0)--(7,9);
    \draw[-,dashed](0,3)--(10,3);
    \draw[-,dashed](0,6)--(10,6);

    \node[text = black, left] (r) at (3,5.6) {$A$};
    \node[text = black, right] (r) at (7,5.6) {$B$};
    \node[text = black, left] (r) at (3,2.6) {$C$};
    \node[text = black, right] (r) at (7,2.6) {$D$};

    \filldraw [black] (7,3) circle (2.8pt);
    \filldraw [black] (3,6) circle (2.8pt);
    \filldraw [black] (3,3) circle (2.8pt);
    \filldraw [black] (7,6) circle (2.8pt);
    
    \node[text = black, below] (r) at (1,9) {\scriptsize{$R_1$}};
    \node[text = black, below] (r) at (5,9) {\scriptsize{$R_2$}};
    \node[text = black, below] (r) at (9,9) {\scriptsize{$R_3$}};
    \node[text = black, above] (r) at (1,0) {\scriptsize{$R_7$}};
    \node[text = black, above] (r) at (5,0) {\scriptsize{$R_8$}};
    \node[text = black, above] (r) at (9,0) {\scriptsize{$R_9$}};
    \node[text = black, above] (r) at (1,4) {\scriptsize{$R_4$}};
    \node[text = black, above] (r) at (5,4) {\scriptsize{$R_4$}};
    \node[text = black, above] (r) at (9,4) {\scriptsize{$R_6$}};

    \end{tikzpicture}
    \caption{Subdivision of $\mathbb{R}^2$ into the regions $R_i$. 
    The points are defined as follows: 
    $A = (\min(u_-, u_+), \max(w_-, w_+))$, 
    $B = (\max(u_-, u_+), \max(w_-, w_+))$, 
    $C = (\min(u_-, u_+), \min(w_-, w_+))$, and 
    $D = (\max(u_-, u_+), \min(w_-, w_+))$.}
    \label{fig: segni}
    \end{figure}
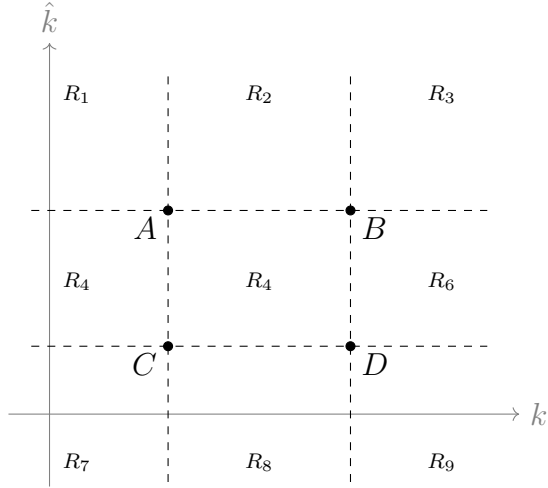

   It can be checked that
   \begin{equation}\label{eq: hesplicita}
       h(k,\hat{k}) = \begin{cases}
           2(f(u_-)-f(u_+))(w_--w_+) \,& \text{in } R_1,\\
           2[(f(u_-)-f(u_+))(k-u_-)+(f(u_-)-f(k))(u_--u_++w_--w_+)] \,& \text{in } R_2,\\
           0 \,& \text{in } R_3,\\
           2(f(u_-)-f(u_+))(\hat{k}-w_+) \,& \text{in } R_4,\\
            h_5(k,\hat{k}) \,& \text{in } R_5,\\
           -2(f(u_-)-f(u_+))(w_--\hat{k}) \,& \text{in } R_6,\\
           0 \,& \text{in } R_7,\\
           -2[(f(u_-)-f(u_+))(u_+-k)+(f(k)-f(u_+))(u_--u_++w_--w_+)] \,& \text{in } R_8,\\
           -2(f(u_-)-f(u_+))(w_--w_+)\,& \text{in } R_9.\\
       \end{cases}
   \end{equation}
   The explicit computation of $h$ in regions such as $R_2$, $R_4$, $R_5$, $R_6$, and $R_8$ is not really necessary; instead, we can deduce the properties of $h$ in this regions by analysing its derivatives, as we will see later, which are easier to compute.
   
   We see now that $h>0$ in $R_1$ and $h<0$ in $R_9$. Moreover, in $R_4$ and $R_6$ it is constant in $k$ and affine in $\hat{k}$, hence necessarily $h>0$ in $R_4$ and $h<0$ in $R_6$. In particular, if we want $h \geq 0$ for every $(k,\hat{k}) \in \LM$, it follows that $(\max(u_-,u_+),\max(w_-,w_+))=(u_-,w_+)$ must lie on the lower boundary of $\LM$, see Figure \ref{fig: regioniR2R5}, left. Therefore, a necessary condition is $w_- = u_- - a$. In $R_2$ instead, $h$ is constant in $\hat{k}$ and it holds
   \[
   \frac{\partial^2 h}{\partial^2 k} = -2(u_--u_++w_--w_+)f''(k) <0,
   \]
   so $k \mapsto h(k,\hat{k})$ is concave on $R_2$. Hence, since $h$ is positive on $R_1$ and zero on $R_3$, then we can conclude that $h\geq0$ on $R_2$. Notice that as $w_-=u_--a$, then $R_2\cap \LM \not= \emptyset,$ see again Figure \ref{fig: regioniR2R5}, left, so the positivity of $h$ on $R_2$ is needed in order to not contradict \eqref{eq: hkappa}. 

   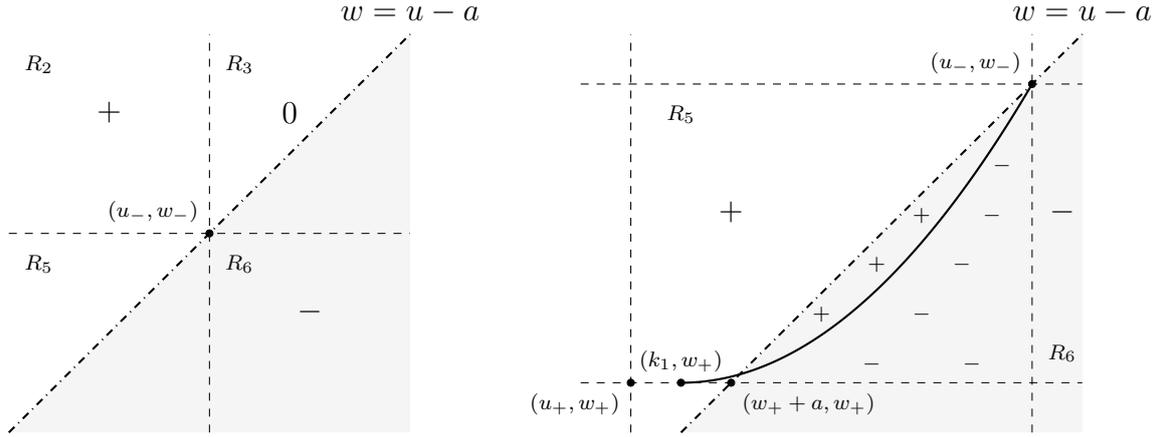
\begin{figure}
    \centering
    \begin{tikzpicture}[scale=1.32]

    \begin{scope}[shift={(0,0)}]

    \draw[-,dashed](1.5,-0.5)--(1.5,3.5);
    \draw[-,dashed](-0.5,1.5)--(3.5,1.5);
    \draw[dash dot, thick] (-0.5,-0.5)--(3.5,3.5)  node[above] {$w=u-a$};

    \draw[fill=black] (1.5,1.5) circle(1pt) node[above left] {\scriptsize{$(u_-,w_-)$}};

    \node[above] at (-0.2,3) {\scriptsize{$R_2$}};
    \node[above] at (-0.2,1) {\scriptsize{$R_5$}};
    \node[above] at (1.8,3) {\scriptsize{$R_3$}};
    \node[above] at (1.8,1) {\scriptsize{$R_6$}};

    \node[above] at (2.5,0.5) {$-$};
    \node[above] at (0.5,2.5) {$+$};
    \node[above] at (2.3,2.5) {$0$};

    \fill[gray, opacity=0.08] (-0.5,-0.5) -- (3.5,3.5) -- (3.5,-0.5) -- cycle;
    
    \end{scope}

    \begin{scope}[shift={(5.7,0)}]

    \draw[-,dashed](0,-0.5)--(0,3.5);
    \draw[-,dashed](4,-0.5)--(4,3.5);
    \draw[-,dashed](-0.5,0)--(4.5,0);
    \draw[-,dashed](-0.5,3)--(4.5,3);
    \draw[dash dot, thick] (0.5,-0.5)--(4.5,3.5)  node[above] {$w=u-a$};

    \draw[fill=black] (1,0) circle(1pt) node[below right,] {\scriptsize{$(w_++a,w_+)$}};
    \draw[fill=black] (0,0) circle(1pt) node[below left] {\scriptsize{$(u_+,w_+)$}};
    \draw[fill=black] (4,3) circle(1pt) node[above left] {\scriptsize{$(u_-,w_-)$}};
    \draw[fill=black] (0.5,0) circle(1pt) node[above] {\scriptsize{$(k_1,w_+)$}};

    \node[above] at (0.5,2.5) {\scriptsize{$R_5$}};
    \node[above] at (4.3,0.1) {\scriptsize{$R_6$}};
    
     \node[above] at (4.3,1.5) {$-$};

    \node[above] at (2.4,0) {\scriptsize{$-$}};
    \node[above] at (3.4,0) {\scriptsize{$-$}};
    \node[above] at (2.9,0.5) {\scriptsize{$-$}};
    \node[above] at (3.3,1) {\scriptsize{$-$}};
    \node[above] at (3.6,1.5) {\scriptsize{$-$}};
    \node[above] at (3.7,2) {\scriptsize{$-$}};

     \fill[gray, opacity=0.08] (0.5,-0.5) -- (4.5,3.5) -- (4.5,-0.5) -- cycle;

    \node[above] at (1,1.5) {$+$};

    \node[above] at (1.9,0.5) {\scriptsize{$+$}};
    \node[above] at (2.45,1) {\scriptsize{$+$}};
    \node[above] at (2.9,1.5) {\scriptsize{$+$}};

    \draw[domain=0.5:4, smooth, samples=100, variable=\x, thick] plot ({\x}, {(\x-0.5)*(\x-0.5)*3/12.25 });

    \end{scope}

    \end{tikzpicture}
    \caption{Left: focus on the regions $R_2, R_3$ and $R_6$ in the case $u_->u_+$ with $f(u_-)>f(u_+)$; since $(u_-,w_-)\in \LM$ and $h<0$ in $R_6$, then necessarily $(u_-,w_-)$ must lie on the lower boundary of $\LM$, so that $\LM \cap R_6 = \emptyset$. Right: focus on $R_5$, where we can see that, to have $\{h<0\}\cap R_5 \cap \LM = \emptyset$, we need to ensure $k_1 \geq (w_++a).$ The gray area represents the region $w < u-a $, which is not included in $\LM$.}
    \label{fig: regioniR2R5}
\end{figure}
   
   We focus now on $R_5$ where we have 
   \[
   \begin{split}
   h(k,\hat{k})=h_5(k,\hat{k})=&(f(u_-)-f(u_+))[-u_+-u_-+2k-w_+-w_-+2\hat{k}]\\&-(u_--u_++w_--w_+)[-f(u_+)-f(u_-)+2f(k)],
   \end{split}
   \] 
   and we study the level set $h(k,\hat{k})=0$. Now as before 
   \[
   \frac{\partial^2 h}{\partial^2 k} = -(u_--u_++w_--w_+)2f''(k)<0,
   \] 
   so the function $k\mapsto h(k,\hat{k})$ is concave, and since for $\hat{k}\in (w_+,w_-)$, $h(u_+,\hat{k})>0$ and $h(u_-,\hat{k})<0$,
   then for every $\hat{k}\in (w_+,w_-)$ fixed there exists one and only one $k$ such that $(k,\hat{k})\in R_5$ and $h(k, \hat{k})=0$. Let us consider the set $J=\{k\in\, ]u_-,u_+[\,, \,|\, \exists \hat{k}\,\hbox{s.t. } h(k,\hat{k})=0\}$ and denote by $k_1:=\inf J$ and by $k_2:=\sup J$. As 
   \[
   \frac{\partial h}{\partial \hat{k}} = 2(f(u_-)-f(u_+))> 0,
   \] 
   then also for each $k\in\, ]k_1,k_2[$  there exists exactly one $\hat{k}$ such that $h(k,\hat{k})=0$. We then can apply the implicit function theorem at each of such points, and get the existence of $\varphi:\, ]k_1,k_2[\, \to \,]w_+,w_-[$ of class $C^1$ such that $h(k, \varphi(k))=0$ for each $k\in\, ]k_1,k_2[$.  The idea is to apply such theorem locally on each point, then use the uniqueness of the zeros, to extend it globally on $]k_1,k_2[$. Moreover, the implicit function theorem also states that 
   \[
   \varphi'(k)= -\frac{h_k(k,\varphi(k))}{h_{\hat{k}}(k,\varphi(k))}= -\frac{2(f(u_-)-f(u_+))-2(u_--u_++w_--w_+)f'(k)}{2(f(u_-)-f(u_+))},
   \] consequently 
   \[
   \varphi''(k)= \frac{(u_--u_++w_--w_+)f''(k)}{f(u_-)-f(u_+)}>0,
   \] 
   meaning that $\varphi $ is convex. We can do the same reasoning for the point $(u_-,w_-)$ as $h(u_-,w_-)=0$ and get that there exists a function $\phi$ defined in a left neighbourhood of $u_-$ such that $h(k, \phi(k))=0$ and 
   \[
   \phi'(u_- -)=-\frac{ 2((f(u_-)-f(u_+))-2(u_--u_++w_-+w_+)f'(u_+)}{2(f(u_-)-f(u_+))}.
   \] 
   But since $(u_--u_++w_--w_+)> u_--u_+$ and $f'(u_+) \leq (f(u_-)-f(u_+))/(u_--u_+)$, by convexity of $f$, it can be shown that $\phi'(u_--)>0$ so actually, by a uniqueness arguments, $\phi$ coincides with $\varphi$ in a left neighbourhood of $u_-$. So we conclude that $k_2=u_-$, i.e. $\varphi:\,]k_1,u_-]\to \,]w_+,w_-],$ with $\varphi(u_-)=w_-.$ 
   Regarding  $k_1$ instead, using the notations $I_+=(w_++a)-u_+$, $I_-=u_--(w_++a)$ and \[\mu=\frac{f(u_-)-f(u_+)}{u_--u_++w_--w_+}= \frac{\mu_+ I_++2\mu_-I_-}{I_++2I_-},\] it can be checked that \[h(w_++a,w_+)=(I_++2I_-)[(\mu-\mu_+)I_+-(\mu-\mu_-)2I_-].\] Since $I_+\geq 0,I_->0$, then $\mu$ is between $\mu_+$ and $\mu_-$ with equality if and only if $\mu_+=\mu_-=\mu,$ and with $\mu=\mu_-$ when $I_+=0$. 
   If now $\mu_+>\mu_-$ and $I_+\not = 0$, then $\mu_+>\mu>\mu_-$ hence $h(w_++a,w_+)<0$. 
   If instead $\mu_+\leq \mu_-$, then $\mu_+\leq \mu \leq \mu_-$, which implies that $h(w_++a,w_+) \geq 0$. 
   Also if $I_+=0$ i.e. $u_+=w_+-a$, then $\mu_-=\mu$ so $h(w_++a,w_+) = 0$. From this and the properties of $h(k,w_+)$, we can  conclude that $k_1 < w_++a$ when $\mu_+>\mu_-\not=\mu$ and instead $k_1 \geq w_++a$ when either $\mu_+\leq \mu_-$ or $I_+=0$, that is $u_+=w_+-a$. By convexity of $\varphi$, $k_1 \geq w_++a$ is sufficient and necessary to have $h\geq0 $ in $R_5\cap \LM$, see Figure \ref{fig: regioniR2R5}, right. So, in the current case, $\mu_+\leq \mu_-$ or $u_+=w_+-a$ is required by the positivity of $h$ in $R_5.$ The only region left to study is $R_8$, which is constant in $\hat{k}$ and so in that region $h(k,\hat{k}) = h(k,w_+)$, so we can deduce again that conditions $\mu_+\leq \mu_-$ or $u_+=w_+-a$ are implied. Finally, by looking at Figure \ref{fig: hm1m2}, where we picture the sign of $h$ in $\LM$, we can conclude that in this case we get an entropy solution if and only if either $u_+=w_+-a$ or $\mu_+\leq \mu_-$, that is if and only if \textit{\ref{casoiii})} or \textit{\ref{casoiv})} hold.

   \begin{figure}
    \centering
    \begin{tikzpicture}[scale=0.69]

    \begin{scope}[shift={(11,0)}]
    
   \draw[-,dashed](3,0)--(3,9);
    \draw[-,dashed](7,0)--(7,9);
    \draw[-,dashed](0,3)--(10,3);
    \draw[-,dashed](0,6)--(10,6);
    \draw[dash dot, thick] (1,0)--(10,9)  node[above] {$w=u-a$};

    \draw[domain=3.5:7, smooth, samples=100, variable=\x, thick] plot ({\x}, {(\x-3.5)*(\x-3.5)*0.2449 + 3});
    \draw[-,thick] (3.5,3)--(3.5,0);

    \fill[gray, opacity=0.08] (1,0) -- (10,0) -- (10,9) -- cycle;

    \node[below] at (1.5,2) {$0$};
    \node[below] at (1.5,5) {$+$};
    \node[below] at (1.5,8) {$+$};
    \node[below] at (8.5,2) {$-$};
    \node[below] at (8.5,5) {$-$};
    \node[below] at (8.5,8.2) {$0$};
    \node[below] at (5.3,2) {$-$};
    \node[below] at (3.27,2) {$+$};
    \node[below] at (6.3,4) {$-$};
    \node[above] at (4.5,4.5) {$+$};
    \node[below] at (5,8) {$+$};

    \node[below] at (1,9) {\scriptsize{$R_1$}};
    \node[below] at (5,9) {\scriptsize{$R_2$}};
    \node[below] at (9,9) {\scriptsize{$R_3$}};
    \node[above] at (1,0) {\scriptsize{$R_7$}};
    \node[above] at (5,0) {\scriptsize{$R_8$}};
    \node[above] at (9,0) {\scriptsize{$R_9$}};
    \node[above] at (0,4) {\scriptsize{$R_4$}};
    \node[above] at (9.5,4) {\scriptsize{$R_6$}};
    \end{scope}

    \begin{scope}[shift={(0,0)}]
    
    \draw[-,dashed](3,0)--(3,9);
    \draw[-,dashed](7,0)--(7,9);
    \draw[-,dashed](0,3)--(10,3);
    \draw[-,dashed](0,6)--(10,6);
    \draw[dash dot, thick] (1,0)--(10,9)  node[above] {$w=u-a$};

    \draw[domain=5:7, smooth, samples=100, variable=\x, thick] plot ({\x}, {(\x-5)*(\x-5)*0.75 + 3});
    \draw[-,thick] (5,3)--(5,0);

    \fill[gray, opacity=0.08] (1,0) -- (10,0) -- (10,9) -- cycle;

    \node[below] at (1.5,2) {$0$};
    \node[below] at (1.5,5) {$+$};
    \node[below] at (1.5,8) {$+$};
    \node[below] at (8.5,2) {$-$};
    \node[below] at (8.5,5) {$-$};
    \node[below] at (8.5,8.2) {$0$};
    \node[below] at (6,2) {$-$};
    \node[below] at (4,2) {$+$};
    \node[below] at (6.5,4) {$-$};
    \node[above] at (5,4) {$+$};
    \node[below] at (5,8) {$+$};

    \node[below] at (1,9) {\scriptsize{$R_1$}};
    \node[below] at (5,9) {\scriptsize{$R_2$}};
    \node[below] at (9,9) {\scriptsize{$R_3$}};
    \node[above] at (1,0) {\scriptsize{$R_7$}};
    \node[above] at (6,0) {\scriptsize{$R_8$}};
    \node[above] at (9,0) {\scriptsize{$R_9$}};
    \node[above] at (0,4) {\scriptsize{$R_4$}};
    \node[above] at (9.5,4) {\scriptsize{$R_6$}};
    
    \end{scope}

    \end{tikzpicture}
    \caption{Sign of the function $h$ when $u_- > u_+$ and $f(u_-) > f(u_+)$. Left: case $\mu_+ \leq \mu_-$, where choosing $(u_-, w_+) = (u_-, u_- - a)$ ensures that $h \geq 0$ for all $(k, \hat{k}) \in \LM$. Right: case $\mu_+ > \mu_-$ and $w_-\not = u_+-a$, where even after imposing the necessary condition $w_+ = u_+ - a$, it does not hold $h \geq 0$ on $\LM$. The gray region represents the region $w<u-a$ which is not in $\LM$. From these pictures, we also deduce that $w_- = u_- - a$ together with $w_+ = u_+ - a$ would be enough to get an entropy solution.
    }
    \label{fig: hm1m2}
\end{figure}
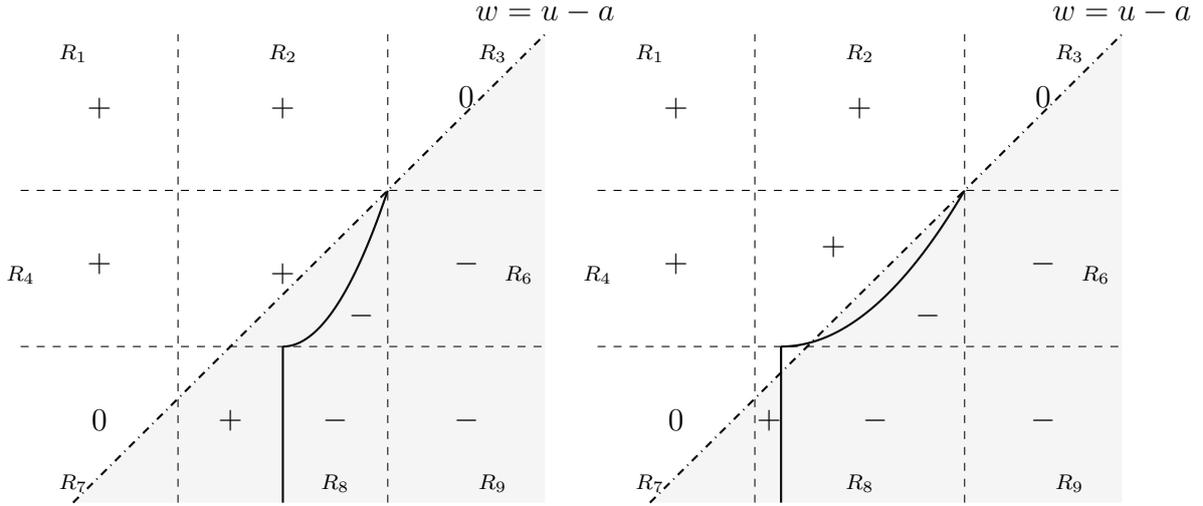

   \textbf{Case $\mathbf{u_- > u_+}$ and $\mathbf{f(u_-)<f(u_+)}$:} we can treat this case as the previous one with the  explicit form \eqref{eq: hesplicita} of $h$ remaining the same. The main difference is in the sign of $h$, since $\sgn(f(u_-)-f(u_+))$ is opposite with respect to the previous case, see Figure \ref{fig: hn1n2}. Under these assumptions, \eqref{eq: fratta} implies that $(u_+,w_+)$ must be on the upper boundary of $\LM$ and, moreover, in region $R_5$ we get that $h(k,\hat{k})=0$, seen as a graph over $k$, is concave. Then conditions $w_-=u_-+
   a$ or $\nu_-\geq \nu_+$ are deduced, so that $\{h <0\}\cap \LM =\emptyset $, see again Figure \ref{fig: hn1n2}. Therefore, in this case we get an entropy solution if and only if  either \textit{\ref{casov})} or \textit{\ref{casovi})} holds. 

    \begin{figure}
    \centering
    \begin{tikzpicture}[scale=0.69]

    \begin{scope}[shift={(11,0)}]
    
   \draw[-,dashed](3,0)--(3,9);
    \draw[-,dashed](7,0)--(7,9);
    \draw[-,dashed](0,3)--(10,3);
    \draw[-,dashed](0,6)--(10,6);
    \draw[dash dot, thick] (0,0)--(9,9)  node[above] {$w=u+a$};

    \draw[domain=3:6.5, smooth, samples=100, variable=\x, thick] plot ({\x}, {-(\x-6.5)*(\x-6.5)* 3/12.25+ 6});
    \draw[-,thick] (6.5,6)--(6.5,9);

    \fill[gray, opacity=0.08] (0,0) -- (9,9) -- (0,9) -- cycle;

    \node[below] at (1.3,2.2) {$0$};
    \node[below] at (1.5,5) {$-$};
    \node[below] at (1.5,8) {$-$};
    \node[below] at (8.5,2) {$+$};
    \node[below] at (8.5,5) {$+$};
    \node[below] at (8.5,8.2) {$0$};
    \node[below] at (5,2) {$+$};
    
    \node[above] at (5,4) {$+$};
    \node[below] at (5,8) {$-$};
    \node[above] at (3.5,4.8) {$-$};
    \node[below] at (6.75,8) {$+$};

    \node[below] at (1,9) {\scriptsize{$R_1$}};
    \node[below] at (5,9) {\scriptsize{$R_2$}};
    \node[below] at (9,9) {\scriptsize{$R_3$}};
    \node[above] at (1,0) {\scriptsize{$R_7$}};
    \node[above] at (5,0) {\scriptsize{$R_8$}};
    \node[above] at (9,0) {\scriptsize{$R_9$}};
    \node[above] at (0.5,4) {\scriptsize{$R_4$}};
    \node[above] at (9.5,4) {\scriptsize{$R_6$}};
    \end{scope}
    
 \begin{scope}[shift={(0,0)}]
    
    \draw[-,dashed](3,0)--(3,9);
    \draw[-,dashed](7,0)--(7,9);
    \draw[-,dashed](0,3)--(10,3);
    \draw[-,dashed](0,6)--(10,6);
    \draw[dash dot, thick] (0,0)--(9,9)  node[above] {$w=u+a$};

    \draw[domain=3:5, smooth, samples=100, variable=\x, thick] plot ({\x}, {-(\x-5)*(\x-5)* 0.75+ 6});
    \draw[-,thick] (5,6)--(5,9);

    \fill[gray, opacity=0.08] (0,0) -- (9,9) -- (0,9) -- cycle;

    \node[below] at (1.3,2.2) {$0$};
    \node[below] at (1.5,5) {$-$};
    \node[below] at (1.5,8) {$-$};
    \node[below] at (8.5,2) {$+$};
    \node[below] at (8.5,5) {$+$};
    \node[below] at (8.5,8.2) {$0$};
    \node[below] at (5,2) {$+$};
    
    \node[above] at (5,4) {$+$};
    \node[below] at (4,8) {$-$};
    \node[above] at (3.5,4.8) {$-$};
    \node[below] at (6,8) {$+$};

    \node[below] at (1,9) {\scriptsize{$R_1$}};
    \node[below] at (4,9) {\scriptsize{$R_2$}};
    \node[below] at (9,9) {\scriptsize{$R_3$}};
    \node[above] at (1,0) {\scriptsize{$R_7$}};
    \node[above] at (5,0) {\scriptsize{$R_8$}};
    \node[above] at (9,0) {\scriptsize{$R_9$}};
    \node[above] at (0.5,4) {\scriptsize{$R_4$}};
    \node[above] at (9.5,4) {\scriptsize{$R_6$}};
    \end{scope}

    \end{tikzpicture}
    \caption{Sign of $h$ when $u_->u_+$ and $f(u_-)<f(u_+)$. Left: admissible case when $\nu_-\geq\nu_+$. Right: non admissible case when $\nu_- < \nu_+$.}
    \label{fig: hn1n2}
    \end{figure}
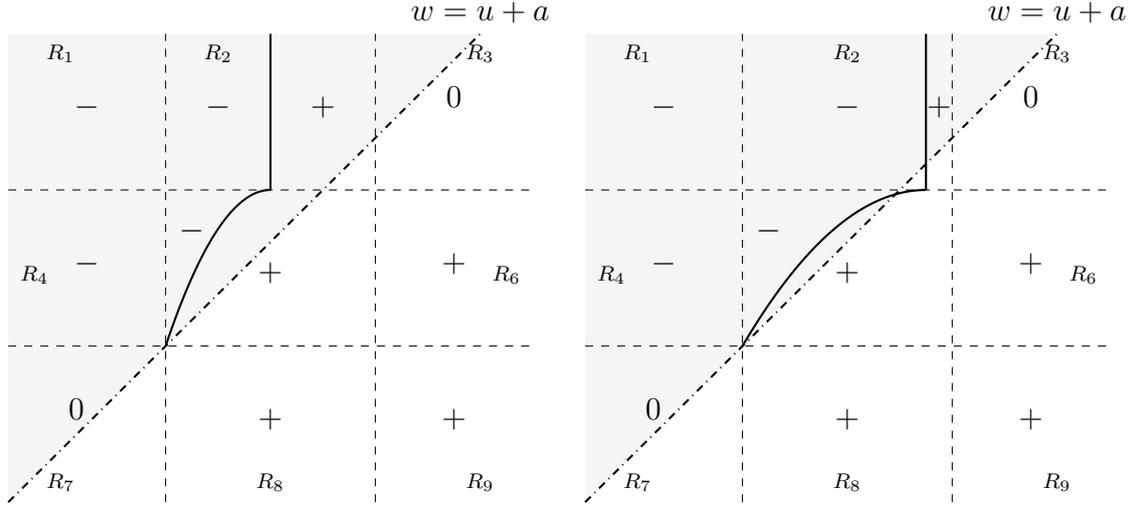

   \textbf{Case $\mathbf{u_-<u_+}$ and $\mathbf{f(u_-)>f(u_+)}$:} now $u_--u_++w_--w_+ <0$, so \eqref{eq: fratta} implies $h\leq 0$ in $\LM.$ Reasoning as before, we can check that in $R_3$ and $R_7$, $h\equiv0$, $h<0$ in $R_1$ and $h>0$ in $R_9$. Again,  $h$ is affine in $R_4$ and $R_6$, respectively negative and positive. As a result, we can still conclude that the point that is now $(u_+,w_+)$ must be on the lower boundary of $\LM$, that is $w_+=u_+-a$, otherwise $R_6 \cap \LM \not=\emptyset$, contradicting $h\geq 0 $. However, in $R_5$ we have
   \[\begin{split}
   h(k,\hat{k})=&(f(u_-)-f(u_+))[u_++u_--2k+w_++w_--2\hat{k}]\\&-(u_--u_++w_--w_+)[f(u_+)+f(u_-)-2f(k)],\end{split}\] and we can prove that there exists a $C^2$ function  $\varphi : [u_-,k_2[\,\to [w_-,w_+[$ such that $\varphi(u_-)=w_-$, $h(k,\varphi(k))=0$ and 
   \[\varphi''(k)= \frac{(u_--u_++w_--w_+)f''(k)}{f(u_-)-f(u_+)}<0.\] 
   So the graph of $\varphi$ is concave, meaning that, even in the best scenario, that is if $k_2=u_+$ and $\varphi(u_+)=w_+$, we would have $R_5 \cap \LM \cap \{h >0 \} \not = \emptyset$, see Figure \ref{fig: haltricasi}, left. So this case is not admitted by \eqref{eq: fratta}. 

   \textbf{Case $\mathbf{u_-<u_+}$ and $\mathbf{f(u_-)<f(u_+)}$:} also this case is incompatible with \eqref{eq: fratta}. Indeed,  we infer again that $(u_+,w_+)$ must lie on the upper boundary of $\LM$ and that instead in $R_5$ we have a convex level set $0$ of the function $h$. This makes it impossible to have $\{h \geq 0\}\cap \LM\not = \emptyset $, contradicting \eqref{eq: fratta}.

    \begin{figure}
    \centering
    \begin{tikzpicture}[scale=0.69]

    \begin{scope}[shift={(0,0)}]
    
   \draw[-,dashed](3,0)--(3,9);
    \draw[-,dashed](7,0)--(7,9);
    \draw[-,dashed](0,3)--(10,3);
    \draw[-,dashed](0,6)--(10,6);
    \draw[dash dot, thick] (1,0)--(10,9)  node[above] {$w=u-a$};

    \draw[domain=3:6, smooth, samples=100, variable=\x, thick] plot ({\x}, {-(\x-6)*(\x-6)* 0.3333+ 6});
    \draw[-,thick] (6,6)--(6,9);

    \fill[gray, opacity=0.08] (1,0) -- (10,0) -- (10,9) -- cycle;

    \node[below] at (1.5,2) {$0$};
    \node[below] at (1.5,5) {$-$};
    \node[below] at (1.5,8) {$-$};
    \node[below] at (8.5,2) {$+$};
    \node[below] at (8.5,5) {$+$};
    \node[below] at (8.5,8.2) {$0$};
    \node[below] at (5,2) {$+$};

    \node[above] at (5,4.5) {$+$};
    \node[above] at (3.7,4.8) {$-$};
    \node[below] at (5,8) {$-$};
    \node[below] at (6.5,8) {$+$};

    \node[below] at (1,9) {\scriptsize{$R_1$}};
    \node[below] at (5,9) {\scriptsize{$R_2$}};
    \node[below] at (9,9) {\scriptsize{$R_3$}};
    \node[above] at (1,0) {\scriptsize{$R_7$}};
    \node[above] at (5,0) {\scriptsize{$R_8$}};
    \node[above] at (9,0) {\scriptsize{$R_9$}};
    \node[above] at (0.5,4) {\scriptsize{$R_4$}};
    \node[above] at (9.5,4) {\scriptsize{$R_6$}};
    \end{scope}

    \begin{scope}[shift={(11,0)}]
    
    \draw[-,dashed](3,0)--(3,9);
    \draw[-,dashed](7,0)--(7,9);
    \draw[-,dashed](0,3)--(10,3);
    \draw[-,dashed](0,6)--(10,6);
    \draw[dash dot, thick] (0,0)--(9,9)  node[above] {$w=u+a$};

     \draw[domain=4:7, smooth, samples=100, variable=\x, thick] plot ({\x}, {(\x-4)*(\x-4)*0.333 + 3});
    \draw[-,thick] (4,3)--(4,0);

    \fill[gray, opacity=0.08] (0,0) -- (9,9) -- (0,9) -- cycle;

    \node[below] at (1.3,2.2) {$0$};
    \node[below] at (1.5,5) {$+$};
    \node[below] at (1.5,8) {$+$};
    \node[below] at (8.5,2) {$-$};
    \node[below] at (8.5,5) {$-$};
    \node[below] at (8.5,8.2) {$0$};
    \node[below] at (5,2) {$-$};
    \node[below] at (3.5,2) {$+$};
    \node[below] at (6.5,4) {$-$};
    \node[above] at (5,4) {$+$};
    \node[below] at (5,8) {$+$};

    \node[below] at (1,9) {\scriptsize{$R_1$}};
    \node[below] at (5,9) {\scriptsize{$R_2$}};
    \node[below] at (9,9) {\scriptsize{$R_3$}};
    \node[above] at (1,0) {\scriptsize{$R_7$}};
    \node[above] at (6,0) {\scriptsize{$R_8$}};
    \node[above] at (9,0) {\scriptsize{$R_9$}};
    \node[above] at (0.5,4) {\scriptsize{$R_4$}};
    \node[above] at (9.5,4) {\scriptsize{$R_6$}};
    
    \end{scope}

    \end{tikzpicture}
    \caption{Left:  sign of the function $h$ when $u_- < u_+$ and $f(u_-) > f(u_+)$; independently from $(u_-,w_-),(u_+,w_+)\in \LM$, $\LM \cap \{h>0\}\not=\emptyset$. Right: $u_- < u_+$ and $f(u_-) < f(u_+)$;  $\LM \cap \{h>0\}\not=\emptyset$, so when $u_- < u_+$ no shock wave solutions are admissible.}
    \label{fig: haltricasi}
    \end{figure}
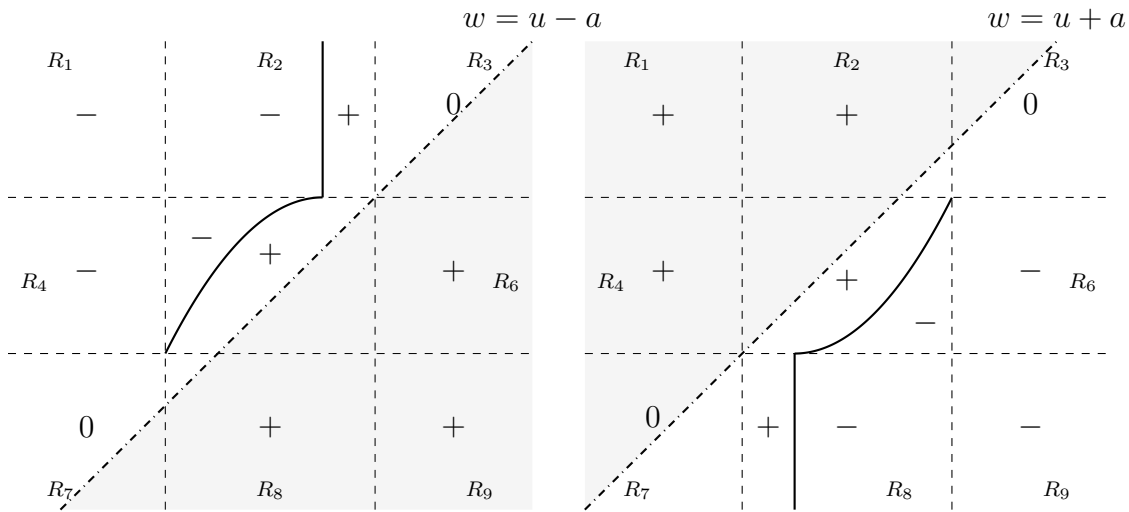

    To conclude we notice that if $u_- <u_+$, no entropy solution is admissible, so $u_-\geq u_+$ is a necessary condition.
\end{proof}

Then the following stability result holds, which implies uniqueness of entropy weak solutions. 

\begin{teo}
    Consider the Cauchy problems with initial conditions $(u_0^1,w_0^1)$ and $(u_0^2,w_0^2)$ respectively, where $u_0^i,w_0^i \in \L1(\R)\cap \BV(\R)$. Let us denote by $(u_1,w_1)$ and $(u_2,w_2)$ two entropy weak solutions of the corresponding Cauchy problems. Then, it holds  \begin{equation}
        \int_{-\infty}^{+\infty} \left(|u_1-u_2|(x,t)+|w_1-w_2|(x,t)\right)dx \leq \int_{-\infty}^{+\infty}\left(|u_0^1-u_0^2|+|w_0^1-w_0^2|\right)dx, 
    \end{equation} for almost every $t\in [0,T[$.
\end{teo}

The proof is an adaptation of the standard doubling of variables method by Kruzkov \cite{Krukov}. In \cite{AVH1} an analogous result for the case of delayed-relay hysteresis is proven, see also \cite[Section~IV.1]{AVH} for the description of such operator. The adjustment to Play hysteresis is straightforward.

\appendix
\section{Details of the proof of the Theorem \ref{teo: esistenza}}\label{appendice}

We show that the sequences $u_m,w_m$ defined in Theorem \ref{teo: esistenza} converge in $\Lloc{1}(\R \times [0,T[)$ to a couple of functions $u,w\in \C{0} ( [0,T[; \Lloc{1}(\R))$ which satisfy \eqref{eq: hweaksol}.

\begin{proof}
 Using a classical argument (see e.g. \cite[Theorem 2.6]{AB3}), by Propositions \ref{prop: linfty} and \ref{prop: stimaBV}, Lemma \ref{cor: contl1} and as a consequence of Helly's compactness theorem, there exists $u,w\in \C{0} ( [0,T[; \Lloc{1}(\R))$ such that  (up to a subsequence) 
 \[
 u_m,w_m\to u,w \quad \text{in } \Lloc1(\R \times [0,T[), \text{ as } m \to+ \infty.
 \]
    
 Now, following the proof of \cite[Theorem 5.3]{FVM}, we show that $u,w$ satisfies \eqref{eq: hweaksol}. We consider $\{(u_i^{n},w_i^n) \,|\, i\in \mathbb{Z} \text{ and } n\in \N\}$ given by the scheme \eqref{eq: schema2} associated to $\Delta x_m$ and $\Delta t_ m$, and we fix $\phi \in \Cc{1}(\R \times [0,T[)$ non negative. Notice that $\forall i\in \mathbb{Z}, n\in \mathbb{N}$, $u_i^{n}$ and $w_i^n$ depend also on $m$, but to easy the notation, we will omit this index. From Proposition \ref{prop: entropiadiscr}, the discrete entropy condition \eqref{eq: discrEntropy} holds for every $m$.  Multiplying then this inequality by $\phi(x, t^n)$, integrating it over $x \in K_i$ and summing over $i$ and $n$, we get 
 \begin{equation}\label{eq: somma}
 A'_m + A''_m+ B_m  \leq 0,
\end{equation}
where
 \begin{equation*}
 \begin{split}
    &A'_m := \sum_{i\in \mathbb{Z}} \sum_{n\in \mathbb{N}} (|u_i^{n+1}-k|-|u_i^n-k|) \int_{K_i} \phi(x, t^n) \, dx,\\
    & A''_m := \sum_{i\in \mathbb{Z}} \sum_{n\in \mathbb{N}} (|w_i^{n+1}-\hat{k}|-|w_i^n-\hat{k}|) \int_{K_i} \phi(x,t^n) \, dx,\\
    & B_m := \sum_{i\in \mathbb{Z}} \sum_{n\in \mathbb{N}} (G_k(u_i^n,u_{i+1}^n)-G_k(u_{i-1}^n,u_{i}^n))\frac{\Delta t_m}{\Delta x_m} \int_{K_i} \phi(x, t^n) \, dx,
    \end{split}
 \end{equation*}
 with $G_k(\alpha,\beta):= g(\alpha \top k, \beta \top k)-g(\alpha \perp k, \beta \perp k),$ (recall \eqref{eq: gGodunov}).
 
 Regarding $A'_m$, we can see that
 \[\begin{split}
 -&\int_{\R}\int_{\Delta t_m}^T |u_m(x,t)-k| \partial_t \phi(x,t-\Delta t_m) \, dt dx - \sum_{i\in \mathbb{Z}} |u_i^0-k| \int_{K_i}\phi(x,0)\, dx
 \\&= - \sum_{n=1}^{+\infty}\sum_{i\in \mathbb{Z}} \iint_{K_i^n} |u_{i}^n-k| \partial_t\phi(x,t-\Delta t_m) \, dtdx - \sum_{i\in \mathbb{Z}} |u_i^0 -k|\int_{K_i}\phi(x,0)\, dx\\ 
 &= -\sum_{n=1}^{+\infty}\sum_{i\in \mathbb{Z}} \int_{K_i} |u_{i}^n-k| (\phi(x,t^n)-\phi(x, (n-1)\Delta t_m) \, dx - \sum_{i\in \mathbb{Z}} |u_i^0-k| \int_{K_i}\phi(x,0)\, dx 
 \\
 & = - \sum_{n\in \mathbb{N}}\sum_{i\in \mathbb{Z}} (|u_i^n-k|-|u_i^{n+1}-k|) \int_{K_i} \phi (x, t^n) \, dx =A'_m.
 \end{split}
 \]
 Hence, from this equality, as $u_m \to u$ in $\Lloc{1}(\R \times [0,T[)$, $\sum u_i^0 \1_{K_i} \to u_0$ in $\Lloc{1}(\R)$, $\phi$ has compact support and $\partial_t \phi(\cdot, \cdot- \Delta t_m) \to \partial_t \phi$ in $\L{\infty}(\R \times[0,T[)$ by uniform continuity, we get 
 \begin{equation}\label{eq: A11}
     \lim_{m\to \infty} A'_m = - \int_0^T \int_{-\infty}^{+\infty} |u-k| \partial_t \phi \, dx \, dt - \int_{-\infty}^{+\infty} |u_0(x)-k| \phi(x,0) \, dx.
 \end{equation}
 By the same reasoning, it also holds  
 \begin{equation}\label{eq: A12}
    \lim_{m\to \infty} A''_m = - \int_0^T \int_{-\infty}^{+\infty} |w-\hat{k}| \partial_t \phi \, dx \, dt - \int_{-\infty}^{+\infty} |w_0(x)-\hat{k}| \phi(x,0) \, dx.
 \end{equation}
 To compute instead the limit of $B_m$, we introduce the following quantity $B_{1,m}$ 
 \[
 B_{1,m}:=- \sum_{n\in\N} \int_{t^n}^{t^{n+1}} \int_{-\infty}^{+\infty} G_k(u_m,u_m) \partial_x\phi(x, t^n) \, dx dt. 
 \] 
By the strong $\Lloc{1}(\R\times[0,T[)$ convergence for $u_m$,  the Lipschitz continuity of $G_k$ (this is a consequence of the Lipschitz continuity of $g$ defined by \eqref{eq: gGodunov}) and  the convergence of $\partial_x \phi(\cdot, n\Delta t) \1_{[n\Delta t_m, (n+1) \Delta t_m[}(\cdot)$ in $\L{\infty}(\R \times [0,T[)$, we can immediately notice that 
 \[
 \lim_{m\to \infty} B_{1,m} = - \int_{0}^T \int_{-\infty}^{+\infty} G_k(u,u) \partial_x \phi  \, dx dt = - \int_{0}^T \int_{-\infty}^{+\infty} \sgn{(u-k)}(f(u)-f(k)) \partial_x \phi \, dx dt.
 \]
 We then have just to compare $B_m$ to $B_{1,m}$, so we rewrite the latter as follows
 \[
 \begin{split}
 B_{1,m} & = -  \Delta t_m \sum_{i\in \N} \sum_{i\in \mathbb{Z}} G_k(u_{i}^n,u_i^n)(\phi (x_{i+1/2}, t^n)-\phi (x_{i-1/2}, t^n))\\
 &=- \Delta t_m \sum_{i\in \N} \sum_{i\in \mathbb{Z}} (G_k(u_{i-1}^n,u_{i-1}^n)-G_k(u_{i}^n,u_{i}^n))\phi (x_{i-1/2}, t^n)\\
 & = \Delta t_m\sum_{i\in \N} \sum_{i\in \mathbb{Z}} (G_k(u_{i}^n,u_{i}^n)- G_k(u_{i-1}^n,u_i^n))\phi (x_{i-1/2}, t^n) + \\
 & ~~~\, + \Delta t_m \sum_{i\in \N} \sum_{i\in \mathbb{Z}} (G_k(u_{i-1}^n,u_{i}^n)- G_k(u_{i-1}^n,u_{i-1}^n))\phi (x_{i-1/2}, t^n)
 \end{split}
 \]
 Similarly, we rewrite
 \[
 \begin{split}
 B_{m}= &\sum_{i\in \mathbb{Z}} \sum_{n\in \mathbb{N}} (G_k(u_i^n,u_{i+1}^n)- G_k(u_i^n,u_i^n)+G_k(u_i^n,u_i^n)-G_k(u_{i-1}^n,u_{i}^n)) \frac{\Delta t_m}{\Delta x_m}\int_{K_i} \phi(x,t^n) \, dx \\ 
 & =  \sum_{i\in \mathbb{Z}} \sum_{n\in \mathbb{N}} (G_k(u_i^n,u_{i}^n)-G_k(u_{i-1}^n,u_{i}^n))\frac{\Delta t_m}{\Delta x_m}\int_{K_i} \phi(x, t^n) \, dx +\\
 &~~~\, + \sum_{i\in \mathbb{Z}} \sum_{n\in \mathbb{N}} (G_k(u_{i-1}^n,u_{i}^n)-G_k(u_{i}^n,u_{i}^n)) \frac{\Delta t_m}{\Delta x_m}\int_{K_i} \phi(x, t^n) \, dx.
 \end{split}
 \]
 Then
 \[
 |B_{1,m}- B_m| \leq C_1 C_2 \Delta t_m \Delta x_m \sum_{i=i_0} ^{i_1} \sum_{n=0}^{N} | u_i^n-u_{i-1}^n|,
 \]
 where $i_0,i_1$ and $N$ (all dependent on $m$) are the indices for which the support of $\phi$ is contained in $[i_0 \Delta x_m, i_1 \Delta x_m]\times[0, (N+1) \Delta t_m[$, $C_1$ is a constant given by the Lipschitz continuity of $G_k$ on $[U_m,U_M]$ and $C_2$ is a constant given as a consequence of the $\C{1}$ regularity of $\phi.$ Indeed,  $C_2$ uniform in $i$, $n$ and $m$, can be found such that \[
 \phi(x_{i-1/2}, t^n) - C_2 \Delta x_m \leq \frac{1}{\Delta x_m} \int_{K_i} \phi(x,t^n)\,dx \leq \phi(x_{i-1/2}, t^n) + C_2 \Delta x_m.
 \]
  In particular, by Frechet-Kolmogorov's compactness theorem for $\L{p}$ spaces (see e.g. \cite{BREZIS}), taking as translation parameter $h_m:=\Delta x_m$, we have, as $\Delta x_m \to 0$, that 
 \begin{equation*}
     0=\lim_{m \to \infty} \iint_{\text{spt}(\phi)} |u_m(x,t)-u_m(x-\Delta x_m, t)| \, dxdt = \lim_{m \to \infty} \Delta t_m \Delta x_m \sum_{i=i_0} ^{i_1} \sum_{n=0}^{N} | u_i^n-u_{i-1}^n|,
 \end{equation*}
 meaning that 
 \begin{equation}\label{eq: B1}
     \lim_{m\to \infty} B_m = \lim_{m\to \infty} B_{1,m} = - \int_{0}^T \int_{-\infty}^{+\infty} \sgn{(u-k)}(f(u)-f(k)) \partial_x \phi \, dx dt.
 \end{equation}
 Finally, taking the limit as $m$ goes to infinity in  \eqref{eq: somma}, from \eqref{eq: A11}, \eqref{eq: A12} and \eqref{eq: B1} we obtain \eqref{eq: hweaksol}.
 
 \end{proof}

\section*{Acknowledgments}
This work was mainly written while Stefan Moreti was visiting the ACUMES team at Inria Centre at Universit\'e C\^ote d'Azur in Sophia Antipolis, France.
SM was partially supported by the INdAM - GNAMPA Project, CUP E53C25002010001: ``Analisi e controllo per alcuni problemi di evoluzione''. He also gratefully acknowledges the Inria financial support during his research stay, as well as the hospitality of the ACUMES team.


{ \small
	\bibliography{biblio}
	\bibliographystyle{abbrv}
}

\end{document}